\documentclass[aos,noinfoline]{imsart}

\RequirePackage[OT1]{fontenc}
\RequirePackage{amsmath,amssymb,amsthm,mathrsfs,enumerate,bm,makeidx,xcolor,verbatim}
\RequirePackage[colorlinks,citecolor=blue,urlcolor=blue]{hyperref}
\RequirePackage[authoryear]{natbib}

\RequirePackage{xr} 
\numberwithin{equation}{section}

\newcommand{\N}{\mathbb{N}}
\newcommand{\R}{\mathbb{R}}
\newcommand{\Q}{\mathbb{Q}}

\newcommand{\pnorm}[2]{\lVert#1\rVert_{#2}}
\newcommand{\abs}[1]{\left\lvert#1\right\rvert}
\newcommand{\iprod}[2]{\left\langle#1,#2\right\rangle}
\newcommand{\intdom}[1]{\mathrm{int(dom}(#1))}
\newcommand{\conv}[1]{\mathrm{conv}\left(#1\right)}
\renewcommand{\d}[1]{\mathrm{d}#1}
\newcommand{\dist}[2]{\mathrm{dist}\left(#1,#2\right)}

\makeindex

\theoremstyle{definition}\newtheorem{problem}{Problem}[section]
\theoremstyle{definition}
\theoremstyle{remark}
\theoremstyle{remark}\newtheorem{remark}[problem]{Remark}
\theoremstyle{definition}\newtheorem{example}[problem]{Example}
\theoremstyle{plain}\newtheorem{theorem}[problem]{Theorem}
\theoremstyle{plain}\newtheorem{lemma}[problem]{Lemma}
\theoremstyle{plain}
\theoremstyle{plain}\newtheorem{corollary}[problem]{Corollary}
\theoremstyle{plain}

\begin{document}

\begin{frontmatter}
\title{Approximation and Estimation of $s$-Concave Densities via R\'enyi Divergences}
\runtitle{$s$-Concave Estimation}

\begin{aug}
\author{\fnms{Qiyang} \snm{Han}\ead[label=e1]{royhan@uw.edu}}
\and
\author{\fnms{Jon A.} \snm{Wellner}\thanksref{t2}\ead[label=e2]{jaw@stat.washington.edu}}
\ead[label=u1,url]{http://www.stat.washington.edu/jaw/}

\thankstext{t2}{Supported in part by NSF Grant DMS-1104832 and NI-AID grant R01 AI029168}
\runauthor{Han and Wellner}

\affiliation{University of Washington}

\address{Department of Statistics, Box 354322\\
University of Washington\\
Seattle, WA 98195-4322\\
\printead{e1}}

\address{Department of Statistics, Box 354322\\
University of Washington\\
Seattle, WA 98195-4322\\
\printead{e2}\\
\printead{u1}}
\end{aug}


\begin{abstract}
In this paper, we study the approximation and estimation of $s$-concave 
densities via R\'enyi divergence. 
We first show that the approximation of a probability measure $Q$ by an $s$-concave 
density exists and is unique via the procedure of minimizing a divergence functional 
proposed by \cite{koenker2010quasi} 
if and only if $Q$ admits full-dimensional support 
and a first moment. We also show continuity of the divergence functional in $Q$:   
if $Q_n \to Q$ in the Wasserstein metric, then the projected densities converge in weighted 
$L_1$ metrics and uniformly on closed subsets of the continuity set of the limit. 
Moreover, directional derivatives of the projected densities also enjoy local uniform convergence. 
This contains both on-the-model and off-the-model situations, and entails strong consistency 
of the divergence estimator of an $s$-concave density under mild conditions. 
One interesting and important feature for the R\'{e}nyi divergence estimator of an $s$-concave 
density is that the estimator
is intrinsically related with the estimation of log-concave densities via maximum likelihood methods. 
In fact, we show that for $d=1$ at least, 
the R\'enyi divergence estimators for $s$-concave densities converge to the maximum 
likelihood estimator of a log-concave density as $s \nearrow 0$.
The R\'enyi divergence estimator shares similar characterizations as the MLE 
for log-concave distributions, which allows us to develop pointwise 
asymptotic distribution theory assuming that the underlying density is $s$-concave.
\end{abstract}



\end{frontmatter}

\newpage
\tableofcontents

\newpage

\section{Introduction}
\subsection{Overview}
The class of $s$-concave densities on $\R^d$ is defined by the generalized means of order $s$ as follows. Let
\[
M_s(a,b;\theta):=
\begin{cases}
\big((1-\theta)a^s+\theta b^s\big)^{1/s}, & s\neq 0, a,b > 0,\\
0, & s <0, ab = 0,\\
a^{1-\theta}b^\theta, &s=0,\\
a\wedge b, &s=-\infty.
\end{cases}
\]
Then a density $p(\cdot)$ on $\R^d$ is called $s$-concave, i.e. $p \in \mathcal{P}_s$ 
if and only if for all $x_0,x_1\in \R^d$ and $\theta \in (0,1)$,
$p\big((1-\theta)x_0+\theta x_1\big)\geq M_s(p(x_0),p(x_1);\theta)$.
This definition apparently goes back to 
\cite{MR0301151} 
with further studies by \cite{MR0388475, MR0404559}, 
\cite{MR0651042}, 
\cite{rinott1976convexity}, 
and \cite{MR735860}; 
see also \cite{dharmadhikari1988unimodality} for a nice summary. 
It is easy to see that the densities $p(\cdot)$ have the form $p=\varphi_+^{1/s}$ for some concave function $\varphi$ if $s>0$, $p=\exp(\varphi)$ for some concave $\varphi$ if $s=0$, and $p=\varphi_+^{1/s}$ for some convex $\varphi$ if $s<0$. The function classes $\mathcal{P}_s$ are nested in $s$ in that for every $r>0>s$, we have
$\mathcal{P}_r\subset \mathcal{P}_0\subset \mathcal{P}_s\subset \mathcal{P}_{-\infty}.$

Nonparametric estimation of $s$-concave densities has been under intense research efforts in recent years. In particular, much attention has been paid to estimation in the special case $s=0$ which corresponds to all log-concave densities on $\R^d$. 
The nonparametric maximum likelihood estimator (MLE) of a log-concave density was studied in the univariate setting by 
\cite{MR1941467}, \cite{dumbgen2009maximum}, \cite{pal2007estimating};  
and in the multivariate setting by \cite{cule2010maximum,cule2010theoretical}. 
The limiting distribution theory at fixed points when $d=1$ was studied in \cite{balabdaoui2009limit}, and rate results in \cite{doss2013global,kim2014global}. 
\cite{dumbgen2011approximation} also studied stability properties of the 
MLE projection of any probability measure onto the class of log-concave densities.

Compared with the well-studied log-concave densities (i.e. $s=0$), 
much remains unknown concerning estimation and inference procedures for the larger classes  $\mathcal{P}_s, s<0$. 
One important feature for this larger class is that the densities in 
$\mathcal{P}_s(s<0)$ are allowed to have heavier and heavier tails as 
$s \to -\infty$. In fact, $t-$distributions with $\nu$ degrees of freedom 
belong to $\mathcal{P}_{-1/(\nu+1)} (\R) $ (and hence also to $\mathcal{P}_s (\R)$ 
for any $s < - 1/(\nu +1)$). The study of maximum likelihood estimators 
(MLE's in the following) for general $s$-concave densities in 
\cite{seregin2010nonparametric} shows that the MLE exists and is 
consistent for $s\in(-1,\infty)$. However there is no known result 
about uniqueness of the MLE of $s$-concave densities except for $s=0$. 
The difficulties in the theory of estimation via MLE lie in the fact we 
have still very little knowledge of `good' characterizations of the MLE 
in the $s$-concave setting. This has hindered further development 
of both theoretical and statistical properties of the estimation procedure.

Some alternative approaches to estimation of $s$-concave densities 
have been proposed in the literature by using divergences other than the log-likelihood functional (Kullback-Leibler divergence in some sense). \cite{koenker2010quasi} proposed an alternative to maximum likelihood based on generalized R\'enyi entropies. Similar procedures were also proposed in parametric settings by \cite{MR1665873} using a family of discrepancy measures.  In our setting of $s$-concave densities with $s<0$, the methods
of \cite{koenker2010quasi} can be formulated as follows.

Given i.i.d. observations $\underline{X}=(X_1,\ldots,X_n)$, consider the primal optimization problem $(\mathcal{P})$:
\begin{equation}\label{primalempirical}
(\mathcal{P})\quad \min_{g \in \mathcal{G}(\underline{X})} L(g,\mathbb{Q}_n)\equiv {1 \over n}\sum_{i=1}^n g(X_i)+{1 \over \abs{\beta}}\int _{\R^d} g(x)^\beta\ \d{x},
\end{equation}
where $\mathcal{G}(\underline{X})$ denotes all non-negative closed convex functions supported on the convex set $\textrm{conv}(\underline{X})$, $\mathbb{Q}_n={1 \over n}\sum_{i=1}^n\delta_{X_i}$ the empirical measure and $\beta=1+1/s<0$. As is shown by \cite{koenker2010quasi}, the associated dual problem $(\mathcal{D})$ is
\begin{equation}\label{dualempirical}
\begin{split}
(\mathcal{D})\quad & \max_{f}\int_{\R^d} \frac{(f(y))^\alpha}{\alpha}\ \d{y},\\
&\textrm{ subject to } f=\frac{\d{(\mathbb{Q}_n-G)}}{\d{y}}\textrm{ for some }G\in \mathcal{G}(\underline{X})^\circ\\
\end{split}
\end{equation}
where $\mathcal{G}(\underline{X})^\circ\equiv \big\{ G\in\mathcal{C}^\ast (\underline{X})|\int g\ \d{G}\leq 0,\textrm{for all } g \in \mathcal{G}(\underline{X})\big\}$ is the polar cone of $\mathcal{G}(\underline{X})$, and $\alpha$ is the conjugate index of $\beta$, i.e. $1/\alpha+1/\beta=1$. Here $\mathcal{C}^\ast (\underline{X})$, the space of signed Radon measures on $\textrm{conv}(\underline{X})$, is the topological dual of $\mathcal{C}(\underline{X})$, the space of continuous functions on $\textrm{conv}(\underline{X})$. We also note that the constraint $G\in \mathcal{G}(\underline{X})^\circ$ in the dual form (\ref{dualempirical}) comes from the `dual' of the primal constraint $g\in \mathcal{G}(\underline{X})$, and the constraint $f=\frac{\d{(\mathbb{Q}_n-G)}}{\d{y}}$ can be derived from the dual computation of $L(\cdot,\mathbb{Q}_n)$: 
\begin{equation*}
\begin{split}
\big(L(\cdot,\mathbb{Q}_n)\big)^\ast(G)&=\sup_{g}\bigg(\iprod{G}{g}-{1 \over n}\sum_{i=1}^n g(X_i)-{1 \over \abs{\beta}}\int _{\R^d} g(x)^\beta\ \d{x}\bigg)\\
&=\sup_g\bigg(\iprod{G-\mathbb{Q}_n}{g}-\int \psi_s(g(x))\d{x}\bigg)=\Psi_s^\ast(G-\mathbb{Q}_n).
\end{split}
\end{equation*}
Here we used the notation $\iprod{G}{g}:=\int g\ \d{G}$, $\psi_s(\cdot):=(\cdot)^\beta/\abs{\beta}$ and $\Psi_s$ is the functional defined by $\Psi_s(g):=\int \psi_s(g(x))\ \d{x}$ for clarity. Now the dual form (\ref{dualempirical}) follows by the well known fact (e.g. \cite{rockafellar71} Corollary 4A) that the form of the above dual functional is given by
\begin{equation*}
\Psi^\ast(G)=
\begin{cases}
\int \psi^\ast(\d{G}/\d{x})\ \d{x} & \textrm{ if } G \textrm{ is absolute continuous with respect to }\\
& \qquad \textrm{ Lebesgue measure,}\\
+\infty &\textrm{ otherwise.}
\end{cases}
\end{equation*}

For the primal problem $(\mathcal{P})$ and the dual problem $(\mathcal{D})$, \cite{koenker2010quasi} proved the following results:
\begin{theorem}[Theorem 4.1, \cite{koenker2010quasi}]
$(\mathcal{P})$ admits a unique solution $g_n^\ast$ if $\mathrm{int}(\mathrm{conv}(\underline{X}))\neq \emptyset$, where $g_n^\ast$ is a polyhedral convex function supported on $\mathrm{conv}(\underline{X})$.
\end{theorem}
\begin{theorem}[Theorem 3.1, \cite{koenker2010quasi}]
Strong duality between $(\mathcal{P})$ and $(\mathcal{D})$ holds. Any dual feasible solution is actually a density on $\R^d$ with respect to the canonical Lebesgue measure. The dual optimal solution $f_n^\ast$ exists, and satisfies $f_n^\ast = (g_n^\ast)^{1/s}.$
\end{theorem}
We note that the above results are all obtained in the empirical setting. At the population level, given a probability measure $Q$ with suitable regularity conditions, consider
\begin{equation}\label{primalpopulation}
(\mathcal{P}_Q)\quad \min_{g \in \mathcal{G}}  L_s(g,Q),
\end{equation}
where
\[
L(g,Q)\equiv L_s(g,Q)\equiv \int g(x)\ \d{Q} + {1 \over \abs{\beta}}\int _{\R^d} g(x)^\beta\ \d{x},
\]
and $\mathcal{G}$ denotes the class of all (non-negative) closed convex functions with non-empty interior, which are coercive in the sense that $g(x)\to \infty,\textrm{ as } \pnorm{x}{}\to \infty$. \cite{koenker2010quasi} show that Fisher consistency holds at the population level:  Suppose $Q(A):=\int_A f_0\ \d{\lambda}$ is defined for some $f_0=g_0^{1/s}$ where $g_0 \in \mathcal{G}$; then $g_0$ is an optimal solution for $(\mathcal{P}_Q)$.

\cite{koenker2010quasi} also proposed a general discretization scheme corresponding to the primal form (\ref{primalempirical}) and the dual form (\ref{dualempirical}) for fast computation, by which the one dimensional problem can be solved via linear programming and the two dimensional problem via semi-definite programming. These have been implemented in the \texttt{R} package \texttt{REBayes} by \cite{KoenkerMizera2014}. Koenker's package depends
in turn on the \texttt{MOSEK} implementation of \cite{andersen-mosek};  see Appendix B of \cite{koenker2010quasi} 
for further details. On the other hand, in the special case $s=0$, computation of the MLE's of log-concave densities has been implemented in the \texttt{R} package \texttt{LogConcDEAD} developed in \cite{cule2010maximum} in arbitrary dimensions. However, expensive search for the proper triangulation of the support $\textrm{conv}(\underline{X})$ renders computation difficult in high dimensions.

In this paper, we show that the estimation procedure proposed by \cite{koenker2010quasi} is the `natural' way to estimate $s$-concave densities. As a starting point, since the classes $\mathcal{P}_s$ are nested in $s$, it is natural to consider estimation of the extreme case $s=0$ (the class of log-concave densities) as some kind of `limit' of estimation of the larger class $s<0$. As we will see, estimation of $s$-concave distributions via R\'enyi divergences is intrinsically related with the estimation of log-concave distributions via maximum likelihood methods. In fact we show that in the empirical setting in dimension 1, the R\'enyi divergence estimators converge to the maximum likelihood estimator for log-concave densities as $s \nearrow 0$.

We will show that the R\'enyi divergence estimators share characterization and stability properties similar to the analogous properties established in the log-concave setting by \cite{dumbgen2009maximum,cule2010theoretical} and \cite{dumbgen2011approximation}. Once these properties are available, further theoretical and statistical considerations in estimation of $s$-concave densities become possible. In particular, the characterizations developed here enable us to overcome some of the difficulties of maximum likelihood estimators as proposed by \cite{seregin2010nonparametric}, and to develop limit distribution theory at fixed points assuming that the underlying model is $s$-concave. The pointwise rate and limit distribution results follow a pattern similar to the corresponding
results for the MLE's in the log-concave setting obtained by \cite{balabdaoui2009limit}. This local point of view also underlines the results on global rates of convergence considered in \cite{doss2013global}, showing that the difficulty of estimation for such densities with tails light or heavy, comes almost solely from the shape constraints, namely, the convexity-based constraints.

The rest of the paper is organized as follows. In Section \ref{theoreticalproperty}, 
we study the basic theoretical properties of the approximation/projection 
scheme defined by the procedure (\ref{primalpopulation}). 
In Section \ref{limitbehaviour}, we study the limit behavior of $s$-concave 
probability measures in the setting of weak convergence under 
dimensionality conditions on the supports of the limiting sequence. 
In Section \ref{limitdistributiontheory}, we develop limiting distribution 
theory of the divergence estimator in dimension 1 under curvature 
conditions with tools developed in Sections \ref{theoreticalproperty} 
and \ref{limitbehaviour}. Related issues and further problems are 
discussed in Section \ref{discussion}. 
Proofs are given in Sections \ref{sec:proofs} and \ref{sec:appendix}.

\subsection{Notation}
In this paper, we denote  the canonical Lebesgue measure on $\R^d$ by $\lambda$ or $\lambda_d$ and write $\pnorm{\cdot}{p}$ for the canonical Euclidean $p$-norm in $\R^d$, and $\pnorm{\cdot}{}=\pnorm{\cdot}{2}$ unless otherwise specified. $B(x,\delta)$ stands for the open ball of radius $\delta$ centered at $x$ in $\R^d$, and ${\bf 1}_A$ for the indicator function of $A\subset \R^d$. We use $L_p(f)\equiv\pnorm{f}{L_p}\equiv \pnorm{f}{p}=(\int \abs{f}^p\d{\lambda_d})^{1/p}$ to denote the $L_p(\lambda_d)$ norm of a measurable function $f$ on $\R^d$ if no confusion arises. 

We write $\textrm{csupp}(Q)$ for the convex support of a measure $Q$ defined on $\R^d$, i.e.
\[
\textrm{csupp}(Q)=\bigcap\{C:C\subset \R^d \textrm{ closed and convex}, Q(C)=1\}.
\]
We let $\mathcal{Q}_0$ denote all probability measures on $\R^d$ whose convex support has non-void interior, while $\mathcal{Q}_1$ denotes the set of all probability measures $Q$ with finite first moment: $\int \pnorm{x}{} Q(\d{x})<\infty$.

We write $f_n \rightarrow_d f$ 
if $P_n$ converges weakly to $P$ for the corresponding probability measures $P_n (A) \equiv \int_A f_n d \lambda $ and 
$P(A) \equiv \int_A f d\lambda $.

We write $\alpha:=1+s,\beta := 1+1/s, r:=-1/s$ unless otherwise specified.



\section{Theoretical properties of the divergence estimator}\label{theoreticalproperty}
In this section, we study the basic theoretical properties of the proposed projection scheme via R\'enyi divergence (\ref{primalpopulation}). Starting from a given probability measure $Q$, we first show the existence and uniqueness of such projections via R\'enyi divergence under assumptions on the index $s$ and $Q$. We will call such a projection the \emph{R\'enyi divergence estimator} for the given probability measure $Q$ in the following discussions. We next show that the projection scheme is continuous in $Q$ in the following sense: if a sequence of probability measures $Q_n$, for which the projections onto the class of $s$-concave densities exist, converge to a limiting probability measure $Q$ in Wasserstein distance, then the corresponding projected densities converge in weighted $L_1$ metrics and uniformly on closed subsets of the continuity set of the limit. The directional derivatives of such projected densities also converge uniformly in all directions in a local sense. We then turn our attention the explicit characterizations of the R\'enyi divergence estimators, especially in dimension 1. This helps in two ways. First, it helps to understand the continuity of the projection scheme in the index $s$, i.e. answers affirmatively the question: For a given probability measure $Q$, does the R\'enyi divergence estimator converge to the log-concave projection as studied in \cite{dumbgen2011approximation} as $s \nearrow 0$? Second, the explicit characterizations are exploited in the development of asymptotic distribution theory presented in Section \ref{limitdistributiontheory}. 

\subsection{Existence and uniqueness}
For a given probability measure $Q$, let $L(Q)=\inf_{g \in \mathcal{G}}L(g,Q)$.
\begin{lemma}\label{lem:finitefunctional}
Assume $-1/(d+1)<s<0$ and $Q \in \mathcal{Q}_0$. Then $L(Q)<\infty$ if and only if $Q \in \mathcal{Q}_1$.
\end{lemma}

Now we state our main theorem for the existence of R\'enyi divergence projection corresponding to a general measure $Q$ on $\R^d$.
\begin{theorem}\label{existence}
Assume $-1/(d+1)<s<0$ and $Q \in \mathcal{Q}_0\cap \mathcal{Q}_1$. 
Then (\ref{primalpopulation}) achieves its nontrivial minimum for some 
$\tilde{g} \in \mathcal{G}$. Moreover, $\tilde{g}$ is bounded away from zero, 
and $\tilde{f}\equiv {\tilde{g}}^{1/s}$ is a bounded density with respect to $\lambda_d$.
\end{theorem}
The uniqueness of the solution follows immediately from the strict convexity of the functional $L(\cdot,Q)$.
\begin{lemma}\label{unique}
$\tilde{g}$ is the unique solution for $(\mathcal{P}_Q)$ if $\mathrm{int}(\mathrm{dom}(\tilde{g}))\neq \emptyset$.
\end{lemma}
\begin{remark}
\label{ExistOfProjection}
By the above discussion, we conclude that the map $Q\mapsto \arg\min_{g \in \mathcal{G}}L(g,Q)$
is well-defined for probability measures $Q$ with suitable regularity conditions: 
in particular, if $Q \in \mathcal{Q}_0$ and $-1/(d+1)<s<0$, it is well-defined 
if and only if $Q \in \mathcal{Q}_1$. From now on we denote the optimal solution 
as $g_s(\cdot|Q)$ or simply $g(\cdot|Q)$ if no confusion arises, and write 
$P_Q$ for the corresponding $s$-concave distribution, 
and say that $P_Q$ is the R{\'e}nyi projection of $Q$ to $P_Q \in {\cal P}_s$.
\end{remark}

\subsection{Weighted global convergence in $\pnorm{\cdot}{L_1}$ and $\pnorm{\cdot}{\infty}$}
\begin{theorem}\label{mainthm1}
Assume $-1/(d+1)<s<0$. Let $\{Q_n\}\subset \mathcal{Q}_0$ be a sequence of probability measures converging weakly to $Q \subset \mathcal{Q}_0\cap\mathcal{Q}_1$. Then
\begin{equation}\label{momentliminf}
\int \pnorm{x}{}\ \d{Q}\leq \liminf_{n \to \infty} \int \pnorm{x}{}\ \d{Q}_n.
\end{equation}
If we further assume that
\begin{equation}\label{momentconv}
\lim_{n \to \infty}\int \pnorm{x}{}\ \d{Q_n}=\int \pnorm{x}{}\ \d{Q},
\end{equation}
then,
\begin{equation}\label{functionalconv}
L(Q)=\lim_{n \to \infty}L(Q_n).
\end{equation}
Conversely, if (\ref{functionalconv}) holds, then  (\ref{momentconv}) holds true.
In the former case(i.e. (\ref{momentconv}) holds), let $g:=g(\cdot|Q)$ 
and $g_n:=g(\cdot|Q_n)$, then $f:=g^{1/s}$, $f_n:=g_n^{1/s}$ satisfy
\begin{equation}\label{pointwiseconv}
\begin{split}
\lim_{n \to \infty, x \to y}f_n(x) &= f(y), \quad \textrm{for all } y \in \R^d\setminus \partial\{f>0\},\\
\limsup_{n \to \infty, x \to y}f_n(x) &\leq f(y), \quad \textrm{for all } y \in \R^d.
\end{split}
\end{equation}
For $\kappa< r-d\equiv-1/s-d$,
\begin{equation}\label{L1conv}
\lim_{n \to \infty}\int (1+\pnorm{x}{})^\kappa\abs{f_n(x)-f(x)}\ \d{x}=0.
\end{equation}
For any closed set $S$ contained in the continuity points of $f$ and $\kappa<r$,
\begin{equation}\label{polyLinfty}
\lim_{n \to \infty}\sup_{x \in S}\big(1+\pnorm{x}{})^\kappa\abs{f_n(x)-f(x)}=0.
\end{equation}
Furthermore, let $\mathcal{D}_f:=\{x \in \intdom{f}: f\textrm{ is differentiable at }x\}$, 
and $T\subset \mathrm{int}(\mathcal{D}_f)$ be any compact set. Then
\begin{equation}\label{eqn:gradient_conv}
\lim_{n \to \infty}\sup_{x \in T,\pnorm{\xi}{2}=1}\abs{\nabla_\xi f_n(x)-\nabla_\xi f(x)}=0
\end{equation}
where $\nabla_\xi f(x):=\lim_{h \searrow 0}\frac{f(x+h\xi)-f(x)}{h}$ denotes the (one-sided) directional derivative along $\xi$.
\end{theorem}

\begin{remark}
The one-sided directional derivative for a convex function $g$ is well-defined and 
$\nabla_\xi g(x)=\inf_{h>0}\frac{g(x+h\xi)-g(x)}{h}$, hence well-defined for $f\equiv g^{1/s}$. 
See Section 23 in \cite{rockafellar1997convex} for more details.
\end{remark}

As a direct consequence, we have the following result covering both on and off-the-model cases.
\begin{corollary}\label{cor:off-model-consistency}
Assume $-1/(d+1)<s<0$. Let $Q$ be a probability measure such that 
$Q\in\mathcal{Q}_0\cap \mathcal{Q}_1$, with $f_{Q}:=g(\cdot|Q)^{1/s}$ 
the density function corresponding to the R{\'e}nyi projection $P_{Q}$ (as in Remark~\ref{ExistOfProjection}). 
Let $\mathbb{Q}_n={1 \over n}\sum_{i=1}^n \delta_{X_i}$ be the 
empirical measure when $X_1,\ldots,X_n$ are i.i.d. with distribution 
$Q$ on $\R^d$. Let $\hat{g}_n := g(\cdot|\mathbb{Q}_n)$ and 
$\hat{f}_n:=\hat{g}_n^{1/s}$ be the R\'enyi divergence estimator of $Q$. Then, almost surely we have
\begin{equation}\label{eqn:off-model-ptconv}
\begin{split}
\lim_{n \to \infty, x \to y}\hat{f}_n(x) &= f_Q(y), \quad \textrm{for all } y \in \R^d\setminus \partial\{f>0\},\\
\limsup_{n \to \infty, x \to y}\hat{f}_n(x) &\leq f_Q(y), \quad \textrm{for all } y \in \R^d.
\end{split}
\end{equation}
For $\kappa< r-d\equiv-1/s-d$,
\begin{equation}\label{eqn:off-model-L1}
\lim_{n \to \infty}\int (1+\pnorm{x}{})^\kappa\abs{\hat{f}_n(x)-f_Q(x)}\ \d{x}=_{a.s.}0.
\end{equation}
For any closed set $S$ contained in the continuity points of $f$ and $\kappa<r$,
\begin{equation}\label{eqn:off-model-uniform}
\lim_{n \to \infty}\sup_{x \in S}\big(1+\pnorm{x}{})^\kappa\abs{\hat{f}_n(x)-f_Q(x)}=_{a.s.}0.
\end{equation}
Furthermore, for any compact set $T \subset \mathrm{int}(\mathcal{D}_{f_Q})$,
\begin{equation}\label{eqn:off-model-derivative}
\lim_{n \to \infty}\sup_{x \in T,\pnorm{\xi}{2}=1}\abs{\nabla_\xi \hat{f}_n(x)-\nabla_\xi f_Q(x)}=_{a.s.}0.
\end{equation}
\end{corollary}

Now we return to the correctly specified case and relax the previous assumption 
that $s> -1/(d+1)$ for the case of the empirical measure 
$Q_n\equiv\mathbb{Q}_n$ and some measure $Q $ 
with finite mean and bounded density $f \in {\cal P}_{s'} \subset {\cal P}_s$ 
with $s'>s$.


\begin{corollary}\label{cor:on-the-model-consistency}
Assume $-1/d<s<0$. Let $Q$ be a probability measure on 
$\R^d$ with density $f \in {\cal P}_{s}$ if $-1/(d+1) < s$ and $f \in {\cal P}_{s'}$ where $s'>  -1/(d+1)\}$ 
if $s \in (-1/d , -1/(d+1)]$.
(Thus $f$ is bounded and $f$ has a finite mean.)
Let $\hat{f}_n \equiv \hat{f}_{n,s}$ be defined as in Corollary \ref{cor:off-model-consistency}. 
Then (\ref{eqn:off-model-ptconv}), (\ref{eqn:off-model-L1}), (\ref{eqn:off-model-uniform}), 
and (\ref{eqn:off-model-derivative}) hold with $f_Q$ replaced by $f$.
\end{corollary}

\subsection{Characterization of the R\'enyi divergence projection and estimator}
We now develop characterizations for the R\'enyi divergence projection, 
especially in dimension $1$. All proofs for this subsection can be found in
Appendix~\ref{appendix:supp_proof_2}.

We note that 
the assumption $-1/(d+1)<s<0$ is imposed only for the 
existence and uniquess of the R\'enyi divergence projection. 
For the specific case of empirical measure $Q_n\equiv \mathbb{Q}_n$, 
this condition can be relaxed to $-1/d<s<0$.

Now we give a variational characterization in the spirit of Theorem 2.2 in 
\cite{dumbgen2009maximum}. This result holds for all dimensions $d\geq 1$.
\begin{theorem}\label{projchar}
Assume $-1/(d+1)<s<0$ and $Q\in \mathcal{Q}_0\cap \mathcal{Q}_1$. Then $g=g(\cdot|Q)$ if and only if
\begin{equation}\label{projchareqn}
\int h \cdot g^{1/s}\ \d{\lambda}\leq \int h\ \d{Q},
\end{equation}
holds for all $h:\R^d \to \R$ such that there exists $t_0>0$ with $g+th \in \mathcal{G}$ holds for all $t\in (0,t_0)$.
\end{theorem}

\begin{corollary}\label{projchar2}
Assume $-1/(d+1)<s<0$ and $Q\in \mathcal{Q}_0\cap \mathcal{Q}_1$ 
and let $h$ be any closed convex function. Then
\[
\int h\ \d{P}\leq \int h\ \d{Q},
\]
where $P=P_Q$ is the R\'enyi projection of $Q$ to $P_Q \in \mathcal{P}_s$ (as in Remark~\ref{ExistOfProjection}).
\end{corollary}

As a direct consequence, we have
\begin{corollary}[Moment Inequalities]\label{momentinequality}
Assume $-1/(d+1)<s<0$ and $Q\in \mathcal{Q}_0\cap \mathcal{Q}_1$. Let $\mu_Q:=\mathbb{E}_Q[X]$. Then $\mu_P=\mu_Q$. Furthermore if $-1/(d+2)<s<0$, we have $\lambda_{\mathrm{max}}(\Sigma_P)\leq \lambda_{\mathrm{max}}(\Sigma_Q)$ and $\lambda_{\mathrm{min}}(\Sigma_P)\leq \lambda_{\mathrm{min}}(\Sigma_Q)$ where $\Sigma_Q$ is the covariance matrix defined by $\Sigma_Q:=\mathbb{E}_Q[(X-\mu_Q)(X-\mu_Q)^T]$. Generally if $-1/(d+k)<s<0$ for some $k\in \N$, then $\mathbb{E}_P[\pnorm{X}{}^l]\leq \mathbb{E}_Q[\pnorm{X}{}^l]$ holds for all $l=1,\ldots,k$.
\end{corollary}

Now we restrict our attention to $d=1$, and in the following we will give a full characterization of the R\'enyi divergence estimator. Suppose we observe $X_1,\ldots,X_n$ i.i.d. $Q$ on $\R$, and let $X_{(1)}\leq X_{(2)}\leq \ldots\leq X_{(n)}$ be the order statistics of $X_1,\ldots,X_n$. Let $\mathbb{F}_n$ be the empirical distribution function corresponding to the empirical probability measure $\mathbb{Q}_n:={1 \over n}\sum_{i=1}^n \delta_{X_i}$. Let $\hat{g}_n:=g(\cdot|\mathbb{Q}_n)$ and $\hat{F}_n(t):=\int_{-\infty}^t \hat{g}_n^{1/s}(x)\ \d{x}$. From Theorem 4.1 in \cite{koenker2010quasi} it follows that $\hat{g}_n$ is a convex function supported on $[X_{(1)},X_{(n)}]$, and linear on $[X_{(i)},X_{(i+1)}]$ for all $i=1,\ldots,n-1$. For a continuous piecewise linear function $h:[X_{(1)},X_{(n)}]\to \R$, define the set of knots to be
\[
\mathcal{S}_n(h):=\{t \in (X_{(1)},X_{(n)}):h'(t-)\neq h'(t+)\}\cap\{X_1,\ldots,X_n\}.
\]

\begin{theorem}\label{secondintegralchar}
Let $g_n$ be a convex function taking the value $+\infty$ on $\R\setminus[X_{(1)},X_{(n)}]$ and linear on $[X_{(i)},X_{(i+1)}]$ for all $i=1,\ldots,n-1$. Let
\[
F_n(t):=\int_{-\infty}^t g_n^{1/s}(x)\ \d{x}.
\]
Assume $F_n(X_{(n)})=1$. Then $g_n=\hat{g}_n$ if and only if
\begin{equation}\label{distcond}
\int_{X_{(1)}}^t \big(F_n(x)-\mathbb{F}_n(x)\big)\ \d{x}
\begin{cases}
= 0 & \textrm{ if } t \in \mathcal{S}_n(g_n)\\
\leq 0 & \textrm{ otherwise}.
\end{cases}
\end{equation}
\end{theorem}

\begin{corollary}\label{distchar}
For $x_0 \in \mathcal{S}_n(\hat{g}_n)$, we have
\[
\mathbb{F}_n(x_0)-{1 \over n}\leq \hat{F}_n(x_0)\leq \mathbb{F}_n(x_0).
\]
\end{corollary}

Finally we give a characterization of the R\'enyi divergence estimator in terms of distribution function as Theorem 2.7 in \cite{dumbgen2011approximation}.
\begin{theorem}\label{thm:characterization_distribution}
Assume $-1/2 <s<0$ and $Q \in \mathcal{Q}_0\cap \mathcal{Q}_1$ 
is a probability measure on $\R$ with distribution function $G(\cdot)$. 
Let $g \in \mathcal{G}$ be such that $f\equiv g^{1/s}$ is a density on $\R$, 
with distribution function $F(\cdot)$. Then $g=g(\cdot|Q)$ if and only if
\begin{enumerate}
\item $\int_\R (F-G)(t) \d{t}=0$;
\item $\int_{-\infty}^x (F-G)(t) \d{t}\leq 0$ for all $x \in \R$ with equality when $x \in \tilde{\mathcal{S}}(g)$.
\end{enumerate}
Here $\tilde{\mathcal{S}}(g):=\{x \in \R: g(x)<{1 \over 2}\big(g(x+\delta)+g(x-\delta)\big) 
\textrm{ holds for }\delta>0 \textrm{ small enough}.\}$.
\end{theorem}
The above theorem is useful for understanding the projected 
$s$-concave density given an arbitrary probability measure $Q \in {\cal Q}_0 \cap {\cal Q}_1$. 
The following example illustrates these projections and also 
gives some insight concerning the boundary properties of the class of $s$-concave densities.
\begin{example}
Consider the class of densities ${\cal Q}$ defined by
\[
\mathcal{Q}=\bigg\{q_\tau(x)=\frac{\tau-1}{2(\tau-2)}\bigg(1+\frac{\abs{x}}{\tau-2}\bigg)^{-\tau}, \tau>2\bigg\}.
\]
Note that $q_\tau$ is $-1/\tau$-concave and \emph{not} $s$-concave for any $0>s>-1/\tau$. We start from arbitrary $q_\tau \in \mathcal{Q}$ with $\tau>2$, and we will show in the following that the projection of $q_\tau$ onto the class of $s$-concave ($0>s>-1/\tau$) distribution through $L(\cdot,q_\tau)$ will be given by $q_{-1/s}$. Let $Q_\tau$ be the distribution function of $q_\tau(\cdot)$, then we can calculate
\[
Q_\tau(x)=
\begin{cases}
{1 \over 2}\big(1-\frac{x}{\tau-2}\big)^{-(\tau-1)} &\textrm{ if } x\leq 0,\\
1-{1 \over 2}\big(1+\frac{x}{\tau-2}\big)^{-(\tau-1)} &\textrm{ if } x>0.
\end{cases}
\]
It is easy to check by direct calculation that
$\int_{-\infty}^x \big(Q_r(t)-Q_\tau(t)\big)\ \d{t}\leq 0$
with equality attained if and only if $x=0$. It is clear that $\tilde{\mathcal{S}}(q_\tau)=\{0\}$ and hence the conditions in Theorem \ref{thm:characterization_distribution} are verified. Note that, in Example 2.9 of \cite{dumbgen2011approximation}, the log-concave approximation of the 
rescaled $t_2$ density is the Laplace distribution. 
It is easy to see from the above calculation that the log-concave projection 
of the whole class $\mathcal{Q}$ will be the Laplace distribution 
$q_\infty={1 \over 2}\exp(-\abs{x})$. Therefore the log-concave approximation 
fails to distinguish densities at least amongst the class $\mathcal{Q}\cup \{t_2\}$.
\end{example}

\subsection{Continuity of the R\'enyi divergence estimator in $s$}
Recall that $\alpha = 1+s$, and then $\alpha,\beta$ is a conjugate pair with $\alpha^{-1}+\beta^{-1}=1$ where $\beta = 1+1/s$. For $1-1/d<\alpha<1$, let
\begin{equation*}
\begin{split}
F_\alpha(f)&={1 \over {\alpha -1}}\log\int f^\alpha (x)\ \d{x},\\
F_1(f)     &=\int f(x)\log f(x)\ \d{x}.
\end{split}
\end{equation*}
For a given index $-1/d<s<0$, and data $\underline{X}=(X_1,\ldots,X_n)$ with non-void $\mathrm{int}(\mathrm{conv}(\underline{X}))$, solving the dual problem (\ref{dualempirical}) for the primal problem (\ref{primalempirical}) is equivalent to solving
\begin{equation}\label{func:dualalpha}
\begin{split}
(\mathcal{D}_\alpha)\quad & \min_{f} F_\alpha(f)={1 \over {\alpha -1}}\log\int f^\alpha (x)\ \d{x}\\
&\textrm{ subject to } f=\frac{\d{(\mathbb{Q}_n-G)}}{\d{y}}\textrm{ for some }G\in \mathcal{G}(\underline{X})^\circ\\
\end{split}
\end{equation}
where $\mathcal{G}(\underline{X})^\circ$ is the polar cone of $\mathcal{G}(\underline{X})$ and $\mathbb{Q}_n={1 \over n}\sum_{i=1}^n\delta_{X_i}$ is the empirical measure. The maximum likelihood estimation of a log-concave density has dual form
\begin{equation}\label{func:dualone}
\begin{split}
(\mathcal{D}_1)\quad & \min_{f}F_1(f)=\int f(x)\log f(x)\ \d{x},\\
&\textrm{ subject to } f=\frac{\d{(\mathbb{Q}_n-G)}}{\d{y}}\textrm{ for some }G\in \mathcal{G}(\underline{X})^\circ.\\
\end{split}
\end{equation}
Let $f_\alpha$ and $f_1$ be the solutions of $(\mathcal{D}_\alpha)$ and $(\mathcal{D}_1)$. For simplicity we drop the explicit notational dependence of $f_\alpha,f$ on $n$. Since $F_\alpha(f)\to F_1(f)$ as $\alpha \nearrow 1$ for $f$ smooth enough, it is natural to expect some convergence property of $f_\alpha$ to $f_1$. The main result is summarized as follows.
\begin{theorem}\label{thm:continuity_s}
Suppose $d=1$. For all $\kappa>0,p\geq 1$, we have the following weighted convergence

$$\lim_{\alpha \nearrow 1}\int (1+\pnorm{x}{})^\kappa \abs{f_\alpha(x)-f_1(x)}^p\ \d{x}=0,$$
Moreover, for any closed set $S$ contained in the continuity points of $f$,
\[
\lim_{\alpha \nearrow 1}\sup_{x \in S}\big(1+\pnorm{x}{}\big)^{\kappa}\abs{f_\alpha(x)-f_1(x)}=0
\]
for all $\kappa>0$. 
\end{theorem}


\section{Limit behavior of $s$-concave densities}\label{limitbehaviour}
Let $\{f_n\}_{n \in \N}$ be a sequence of $s$-concave densities with 
corresponding measures $\d{\nu_n}=f_n\d{\lambda}$. Suppose 
$\nu_n \to_d \nu$. From \cite{MR0388475, MR0404559} 
and \cite{MR0450480}, 
we know that each $\nu_n$ is a $t-$concave measure with $t = s/(1+sd)$ 
if $-1/d< s < \infty$, $t=-\infty$ if $s = -1/d$, and $t=1/d$ if $s = \infty$. 
This result is proved via different methods by \cite{rinott1976convexity}.
Furthermore, if the dimension of the support of $\nu$ is $d$, then it follows from 
\cite{MR0388475}, Theorem 2.2 that the limit measure $\nu$ is $t-$concave 
and hence has a Lebesgue density with $s = t/(1-td)$. 
Here we pursue this type of result in somewhat more detail. 
Our key dimensionality condition will be formulated in terms of the set 
$C := \{ x \in \R^d : \liminf f_n (x) > 0 \}$.  We will show that if 
\begin{enumerate}
\item[(D1)] Either  $\mbox{dim} ( \mbox{csupp} (\nu)) = d$ or $\mbox{dim} (C) = d$
\end{enumerate}
holds, then the limiting probability measure $\nu$ admits an upper semi-continuous $s$-concave 
density on $\R^d$.  Furthermore, if a sequence of $s$-concave densities $\{ f_n \}$ converges weakly
to some density $f$ (in the sense that the corresponding probability measures converge weakly), 
then $f$ is $s$-concave, and $f_n$ converges to $f$ in weighted $L_1$ metrics and uniformly on any closed set 
of continuity points of $f$. The directional derivatives of $f_n$ also converge uniformly in all directions in a local sense.

In the following sections, we will not fully exploit the strength of the results we have obtained. 
The results obtained will be interesting in their own right, and careful readers 
will find them useful as technical support for Sections \ref{theoreticalproperty} and \ref{limitdistributiontheory}.

\subsection{Limit characterization via dimensionality condition}

Note that $C$ is a convex set. For a general convex set $K$, we follow the convention 
(see \cite{rockafellar1997convex}) that $\dim K=\dim(\mathrm{aff}(K))$, 
where $\mathrm{aff}(K)$ is the affine hull of $K$.
It is well known that the dimension of a convex set $K$ is the 
maximum of the dimensions of the various simplices included in $K$ 
(cf. Theorem 2.4, \cite{rockafellar1997convex}).

We first extend several results in \cite{kim2014global} and \cite{cule2010theoretical} 
from the log-concave setting to our $s$-concave setting. 
The proofs will all be deferred to Appendix~\ref{appendix:supp_proof_3}. 
\begin{lemma}\label{lem:relate_support_set}
Assume (D1). Then $\mathrm{csupp}(\nu)=\overline{C}$.
\end{lemma}

\begin{lemma}\label{pwconv}
Let $\{\nu_n\}_{n \in \N}$ be probability measures with upper semi-continuous 
$s$-concave densities $\{f_n\}_{n \in \N}$ such that $\nu_n \to \nu$ weakly as 
$n \to \infty$. Here $\nu$ is a probability measure with density $f$. 
Then $f_n\to_{a.e.} f$, and $f$ can be taken as $f=\mathrm{cl}(\lim_n f_n)$ 
and hence upper semi-continuous $s$-concave.
\end{lemma}

In many situations, uniform boundedness of a sequence of $s$-concave 
densities give rise to good stability and convergence property.
\begin{lemma}\label{unifbound}
Assume $-1/d<s<0$. Let $\{f_n\}_{n \in \N}$ be a sequence of $s$-concave 
densities on $\R^d$. If $\dim C=d$ where $C=\{\liminf_n f_n>0\}$ as above, 
then $\sup_{n \in \N}\pnorm{f_n}{\infty}<\infty$. 
\end{lemma}
Now we state one limit characterization theorem.
\begin{theorem}\label{limitchar}
Assume $-1/d<s<0$. Under either condition of (D1), $\nu$ is absolutely 
continuous with respect to $\lambda_d$, with a version of the Radon-Nikodym 
derivative $\mathrm{cl}(\lim_n f_n)$, which is an upper semi-continuous and an $s$-concave density on $\R^d$.
\end{theorem}

\subsection{Modes of convergence}
It is shown above that the weak convergence of $s$-concave probability measures 
implies almost everywhere pointwise convergence at the density level. In many 
applications, we wish different/stronger types of convergence. This subsection is 
devoted to the study of the following two types of convergence:
\begin{enumerate}
\item Convergence in $\pnorm{\cdot}{L_1}$ metric;
\item Convergence in $\pnorm{\cdot}{\infty}$ metric.
\end{enumerate}
We start by investigating convergence property in $\pnorm{\cdot}{L_1}$ metric.
\begin{lemma}\label{weakconvbound}
Assume $-1/d<s<0$. Let $\nu,\nu_1,\ldots,\nu_n,\ldots$ be probability measures 
with upper semi-continuous $s$-concave densities $f,f_1,\ldots,f_n,\ldots$ 
such that $\nu_n \to \nu$ weakly as $n \to \infty$.  Then there exists $a,b>0$ 
such that $f_n(x)\vee f(x)\leq \big(a\pnorm{x}{}+b\big)^{1/s}.$
\end{lemma}
Once the existence of a suitable integrable envelope function is established, 
we conclude naturally by dominated convergence theorem that
\begin{theorem}\label{weakconvmodeconv}
Assume $-1/d<s<0$. Let $\nu,\nu_1,\ldots,\nu_n,\ldots$ be probability measures 
with upper semi-continuous $s$-concave densities $f,f_1,\ldots,f_n,\ldots$ such that 
$\nu_n \to \nu$ weakly as $n \to \infty$. Then for $\kappa<r-d$,
\begin{equation}
\lim_{n \to \infty}\int (1+\pnorm{x}{})^\kappa\abs{f_n(x)-f(x)}\ \d{x}=0.
\end{equation}
\end{theorem}

Next we examine convergence of $s$-concave densities in $\pnorm{\cdot}{\infty}$ norm. 
We denote $g=f^{s},g_n=f_n^{s}$ unless otherwise specified. 
Since we have established pointwise convergence in Lemma \ref{pwconv}, 
classical convex analysis guarantees that the convergence is uniform over 
compact sets in $\intdom{f}$. To establish global uniform convergence result, 
we only need to control the tail behavior of the class of $s$-concave functions 
and the region near the boundary of $f$. 
This is accomplished via Lemmas \ref{tailsconcave} and \ref{shiftballcontrol}.

\begin{theorem}\label{localconv}
Let $\nu,\nu_1,\ldots,\nu_n,\ldots$ be probability measures with upper 
semi-continuous $s$-concave densities $f,f_1,\ldots,f_n,\ldots$ such that 
$\nu_n \to \nu$ weakly as $n \to \infty$. 
Then for any closed set $S$ contained in the continuity points of $f$ and $\kappa<r=-1/s$,
\[
\lim_{n \to \infty}\sup_{x \in S}\big(1+\pnorm{x}{}\big)^{\kappa}\abs{f_n(x)-f(x)}=0.
\]
\end{theorem}

We note that no assumption on the index $s$ is required here.



\subsection{Local convergence of directional derivatives}
It is known in convex analysis that if a sequence of convex functions 
$g_n$ converges pointwise to $g$ on an open convex set, then the 
subdifferential of $g_n$ also `converges' to the subdifferential of $g$. 
If we further assume smoothness of $g_n$, then local uniform convergence 
of the derivatives automatically follows. See Theorems 24.5 and 25.7 in 
\cite{rockafellar1997convex} for precise statements. Here we pursue this 
issue at the level of transformed densities.
\begin{theorem}\label{thm:conv_derivatives}
Let $\nu,\nu_1,\ldots,\nu_n,\ldots$ be probability measures with 
upper semi-continuous $s$-concave densities $f,f_1,\ldots,f_n,\ldots$ 
such that $\nu_n \to \nu$ weakly as $n \to \infty$. 
Let $\mathcal{D}_f:=\{x \in \intdom{f}: f\textrm{ is differentiable at }x\}$, 
and $T\subset \mathrm{int}(\mathcal{D}_f)$ be any compact set. Then
\begin{equation*}
\lim_{n \to \infty}\sup_{x \in T,\pnorm{\xi}{2}=1}\abs{\nabla_\xi f_n(x)-\nabla_\xi f(x)}=0.
\end{equation*}
\end{theorem}



\section{Limiting distribution theory of the divergence estimator}\label{limitdistributiontheory}

In this section we establish local asymptotic distribution theory of the divergence estimator $\hat{f}_n$ at a fixed point $x_0 \in \R$. Limit distribution theory in shape-constrained estimation was pioneered for monotone density and regression estimators by \cite{rao1969estimation}, \cite{MR0277070}, \cite{wright1981asymptotic} and \cite{groenenboom1984estimating}. \cite{groeneboom2001estimation} established pointwise limit theory for the MLE's and LSE's of a convex decreasing density, and also treated pointwise limit theory estimation of a convex regression function. \cite{balabdaoui2009limit} established pointwise limit theorems for the MLEs of log-concave densities
on $\R$. On the other hand, for nonparametric estimation of $s$-concave densities, 
asymptotic theory beyond the Hellinger consistency 
results for the MLE's established by \cite{seregin2010nonparametric} has been non-existent.   
\cite{doss2013global} have shown in the case of $d=1$ that the MLE's have Hellinger convergence rates 
of order $O_p (n^{-2/5})$ for each $s \in (-1, \infty)$ (which includes the log-concave case $s=0$).  
However, due at least in part to the lack of explicit characterizations of the MLE for $s$-concave classes, 
no results concerning limiting distributions of the MLE at fixed points are currently available.  In the remainder 
of this section we formulate results of this type  for the R\'{e}nyi divergence estimators. 
These results are comparable to the pointwise limit distribution results for the MLE's of log-concave densities 
obtained by \cite{balabdaoui2009limit}.

In the following, we will see how natural and strong characterizations developed in Section \ref{theoreticalproperty} help us to understand the limit behavior of the R\'enyi divergence estimator at a fixed point. For this purpose, we assume the true density $f_0=g_0^{-r}$ satisfies the following:
\begin{enumerate}
\item[(A1).] $g_0 \in \mathcal{G}$ and $f_0$ is an $s$-concave density on $\R$, where $-1<s<0$;
\item[(A2).] $f_0(x_0)>0$;
\item[(A3).] $g_0$ is locally $C^k$ around $x_0$ for some $k\geq 2$.
\item[(A4).] Let $k:=\mathrm{max}\{k \in \N: k\geq 2, g_0^{(j)}(x_0)=0, \textrm{ for all } 2\leq j \leq k-1, g_0^{(k)}(x_0)\neq 0\}$, and $k=2$ if the above set is empty. Assume $g_0^{(k)}$ is continuous around $x_0$.
\end{enumerate}
\subsection{Limit distribution theory}\label{subsection:limit_distribution_theory}
Before we state the main results concerning the limit distribution theory for the R\'enyi divergence estimator, let us sketch the route by which the theory is developed. We first denote
$ \hat{F}_n(x):=\int_{-\infty}^x \hat{f}_n(t)\ \d{t}$, $\hat{H}_n(x):=\int_{-\infty}^x \hat{F}_n(t)\ \d{t}$ and $\mathbb{H}_n(x):=\int_{-\infty}^x\mathbb{F}_n(t)\ \d{t}$.
We also denote $r_n:=n^{(k+2)/(2k+1)}$ and $\bm{l}_{n,x_0}=[x_0,x_0+n^{-1/(2k+1)}t]$. Due to the form of the characterizations obtained in Theorem \ref{secondintegralchar}, we define \emph{local processes} at the level of integrated distribution functions as follows:
\begin{equation*}
\begin{split}
\mathbb{Y}^\mathrm{loc}_n(t):&=r_n\int_{\bm{l}_{n,x_0}}\bigg(\mathbb{F}_n(v)-\mathbb{F}_n(x_0)-\int_{x_0}^v\big(\sum_{j=0}^{k-1}\frac{f_0^{(j)}(x_0)}{j!}(u-x_0)^j\big)\ \d{u}\bigg)\ \d{v};\\
\mathbb{H}^\mathrm{loc}_n(t):&=r_n\int_{\bm{l}_{n,x_0}}\bigg(\hat{F}_n(v)-\hat{F}(x_0)-\int_{x_0}^v\big(\sum_{j=0}^{k-1}\frac{f_0^{(j)}(x_0)}{j!}(u-x_0)^j\big)\ \d{u}\bigg)\ \d{v}\\
&\qquad+\hat{A}_n t+\hat{B}_n,
\end{split}
\end{equation*}
where $ \hat{A}_n:= n^{\frac{k+1}{2k+1}}\big(\hat{F}_n(x_0)-\mathbb{F}_n(x_0)\big)$ and $\hat{B}_n:= n^{\frac{k+2}{2k+1}}\big(\hat{H}_n(x_0)-\mathbb{H}_n(x_0)\big)$ are defined so that $\mathbb{Y}^\mathrm{loc}_n(\cdot)\geq \mathbb{H}^\mathrm{loc}_n(\cdot)$ by virtue of Theorem \ref{secondintegralchar}. Since we wish to derive asymptotic theory at the level of the underlying convex function, we modify the processes by 
\begin{equation}\label{eqn:modified_loc_proc}
\begin{split}
\mathbb{Y}^\mathrm{locmod}_n(t):&= \frac{\mathbb{Y}^\mathrm{loc}_n(t)}{f_0(x_0)}-r_n\int_{\bm{l}_{n,x_0}}\int_{x_0}^v \hat{\Psi}_{k,n,2}(u)\d{u}\d{v},\\
\mathbb{H}^\mathrm{locmod}_n(t):&= \frac{\mathbb{H}^\mathrm{loc}_n(t)}{f_0(x_0)}-r_n\int_{\bm{l}_{n,x_0}}\int_{x_0}^v \hat{\Psi}_{k,n,2}(u)\d{u}\d{v}.
\end{split}
\end{equation}
where
\begin{equation}
\begin{split}
\hat{\Psi}_{k,n,2}(u)&=\frac{1}{f_0(x_0)}\left(\hat{f}_n(u)-\sum_{j=0}^{k-1}\frac{f_0^{(j)}(x_0)}{j!}(u-x_0)^j\right)\\
&\quad+\frac{r}{g_0(x_0)}\left(\hat{g}_n(u)-g_0(x_0)-g_0'(x_0)(u-x_0)\right).
\end{split}
\end{equation}
A direct calculation reveals that with $r=-1/s>0$, 
\[
\mathbb{H}^\mathrm{locmod}_n(t)=\frac{-r\cdot r_n}{g_0(x_0)}\int_{\bm{l}_{n,x_0}}\int_{x_0}^v \bigg(\hat{g}_n(u)-g_0(x_0)-(u-x_0)g_0'(x_0)\bigg)\ \d{u} \d{v}+\frac{\hat{A}_n t+\hat{B}_n}{f_0(x_0)},
\]
and hence
\begin{equation}\label{rel1}
\begin{split}
n^{\frac{k}{2k+1}}\big(\hat{g}_n(x_0+s_nt)-g_0(x_0)-s_ntg_0'(x_0)\big)&=\frac{g_0(x_0)}{-r}\frac{\d{}^2}{\d{t^2}}\mathbb{H}_n^{\textrm{locmod}}(t),\\
n^{\frac{k-1}{2k+1}}\big(\hat{g}_n'(x_0+s_nt)-g_0'(x_0)\big)&=\frac{g_0(x_0)}{-r}\frac{\d{}^3}{\d{t^3}}\mathbb{H}_n^{\textrm{locmod}}(t).
\end{split}
\end{equation}
It is clear from (\ref{eqn:modified_loc_proc}) that the order relationship $\mathbb{Y}^\mathrm{locmod}_n(\cdot)\geq \mathbb{H}^\mathrm{locmod}_n(\cdot)$ is still valid for the modified processes. Now by tightness arguments, the limit process $\mathbb{H}$ of $\mathbb{H}^\mathrm{locmod}_n$, including its derivatives, exists uniquely, giving us the possibility of taking the limit in (\ref{rel1}) as $n \to \infty$.  Finally we relate $\mathbb{H}$ to the canonical process $H_k$ defined in Theorem \ref{thm:limittheory} by looking at their respective `envelope' functions $\mathbb{Y}$ and $Y_k$, where $\mathbb{Y}$ denotes the limit process of $\mathbb{Y}^{\mathrm{locmod}}_n$ and $Y_k(t)=\int_{0}^{t}W(s)\ \d{s}-t^{k+2}$. Careful calculation of the limit of $\mathbb{Y}^{\mathrm{loc}}_n$ and $\hat{\Psi}_{k,n,2}$ reveals that
\[
\mathbb{Y}^{\mathrm{locmod}}_n(t) \to_d \frac{1}{\sqrt{f_0(x_0)}}\int_0^t W(s)\ \d{s}-\frac{rg_0^{(k)}(x_0)}{g_0(x_0)(k+2)!}t^{k+2},
\]
Now by the scaling property of Brownian motion, $W(at) =_d \sqrt{a} W(t)$, we get the following theorem.

\begin{theorem}\label{thm:limittheory}
Under assumptions (A1)-(A4), we have
\begin{equation}\label{eqn:limit_ground_g}
\begin{pmatrix}
n^{\frac{k}{2k+1}}\big(\hat{g}_n(x_0)-g_0(x_0)\big)\\
n^{\frac{k-1}{2k+1}}\big(\hat{g}_n'(x_0)-g_0'(x_0)\big)\\
\end{pmatrix}
\to_d
\begin{pmatrix}
-\bigg(\frac{g_0^{2k}(x_0)g_0^{(k)}(x_0)}{r^{2k}f_0(x_0)^k(k+2)!}\bigg)^{1/(2k+1)}H_k^{(2)}(0)\\
-\bigg(\frac{g_0^{2k-2}(x_0)\big[g_0^{(k)}(x_0)\big]^3}{r^{2k-2}f_0(x_0)^{k-1}\big[(k+2)!\big]^3}\bigg)^{1/(2k+1)}H_k^{(3)}(0)
\end{pmatrix},
\end{equation}
and
\begin{equation}\label{eqn:limit_cvxfcn}
\begin{pmatrix}
n^{\frac{k}{2k+1}}\big(\hat{f}_n(x_0)-f_0(x_0)\big)\\
n^{\frac{k-1}{2k+1}}\big(\hat{f}_n'(x_0)-f_0'(x_0)\big)\\
\end{pmatrix}
\to_d
\begin{pmatrix}
\bigg(\frac{rf_0(x_0)^{k+1}g_0^{(k)}(x_0)}{g_0(x_0)(k+2)!}\bigg)^{1/(2k+1)}H_k^{(2)}(0)\\
\bigg(\frac{r^3f_0(x_0)^{k+2}\big(g_0^{(k)}(x_0)\big)^3}{g_0(x_0)^3\big[(k+2)!\big]^3}\bigg)^{1/(2k+1)}H_k^{(3)}(0)
\end{pmatrix},
\end{equation}
where $H_k$ is the unique lower envelope of the process $Y_k$ satisfying
\begin{enumerate}
\item $H_k(t)\leq Y_k(t)$ for all $t \in \R$;
\item $H^{(2)}_k$ is concave;
\item $H_k(t)=Y_k(t)$ if the slope of $H^{(2)}_k$ decreases strictly at $t$.
\end{enumerate}
\end{theorem}
\begin{remark}
We note that the minus sign appearing in (\ref{eqn:limit_ground_g}) is due to the convexity of $\hat{g}_n,g_0$ and the concavity of the limit process $H_k^{(2)}(0)$. The dependence of the constant appearing in the limit is optimal in view of Theorem 2.23 in \cite{seregin2010nonparametric}.
\end{remark} 
\begin{remark}
Assume $-1<s<0$ and $k=2$. Let $f_0=\exp(\varphi_0)$ be a log-concave density where $\varphi_0 :\R\to \R$ is the underlying concave function. Then $f_0$ is also $s$-concave. Let $g_s:=f_0^{-1/r}=\exp(-\varphi_0/r)$ be the underlying convex function when $f_0$ is viewed as an $s$-concave density. Then direct calculation yields that
\[g_s^{(2)}(x_0)={1 \over r^2}g_s(x_0)\left(\varphi_0'(x_0)^2-r\varphi_0''(x_0)\right).\]
Hence the constant before $H_k^{(2)}(0)$ appearing in (\ref{eqn:limit_cvxfcn}) becomes
\[\left(\frac{f_0(x_0)^3\varphi'_0(x_0)^2}{4!r}+\frac{f_0(x_0)^3 \abs{\varphi_0''(x_0)}}{4!}\right)^{1/5}.\]
Note that the second term in the above display is exactly the constant involved in the limiting distribution when $f_0(x_0)$ is estimated via the log-concave MLE, see (2.2), page 1305 in \cite{balabdaoui2009limit}. The first term is non-negative and hence illustrates the price we need to pay by estimating a true log-concave density via the R\'enyi divergence estimator over a larger class of $s$-concave densities. We also note that the additional term vanishes as $r \to \infty$, or equivalently $s \nearrow 0$.
\end{remark}

\subsection{Estimation of the mode}
We consider the estimation of the mode of an $s$-concave density $f(\cdot)$ defined by $M(f):=\inf\{t \in \R: f(t)=\sup_{u \in \R} f(u)\}.$
\begin{theorem}\label{thm:estimation_mode}
Assume (A1)-(A4) hold. Then
\begin{equation}\label{eqn:asym_mode}
n^{1/(2k+1)}\big(\hat{m}_n-m_0\big)\to_d \bigg(\frac{g_0(m_0)^{2}(k+2)!^2}{r^2f_0(m_0)g_0^{(k)}(m_0)^2}\bigg)^{1/(2k+1)}M(H^{(2)}_k),
\end{equation}
where $\hat{m}_n=M(\hat{f}_n),m_0=M(f_0)$.
\end{theorem}
By Theorem 2.26 in \cite{seregin2010nonparametric}, the dependence of the constant on local smoothness is optimal when $k=2$. Here we show that this dependence is also optimal for $k> 2$.

Consider a class of densities $\mathcal{P}$ dominated by the canonical Lebesgue measure on $\R^d$. Let $T:\mathcal{P}\to \R$ be any functional. For an increasing convex loss function $l(\cdot)$ on $\mathbb{R}_+$, we define the \emph{minimax risk} as
\begin{equation}
R_l(n;T,\mathcal{P}):=\inf_{t_n}\sup_{p \in \mathcal{P}}\mathbb{E}_{p^{\times n}}l\big(\abs{t_n(X_1,\ldots,X_n)-T(p)}\big),
\end{equation}
where the infimum is taken over all possible estimators of $T(p)$ based on $X_1,\ldots,X_n$. Our basic method of deriving minimax lower bound based on the following work in \cite{jongbloed2000minimax}.
\begin{theorem}[Theorem 1 \cite{jongbloed2000minimax}]\label{minimaxlowerboundgeneric}
Let $\{p_n\}$ be a sequence of densities in $\mathcal{P}$ such that $\limsup_{n \to \infty} n h^2(p_n,p)\leq \tau^2$ for some density $p \in \mathcal{P}$. Then
\begin{equation}
\liminf_{n \to \infty}\frac{R_l(n;T,\{p,p_n\})}{l\big(\exp(-2\tau^2)/4\cdot\abs{T(p_n)-T(p)}\big)}\geq 1.
\end{equation}
\end{theorem}
For fixed $g \in \mathcal{G}$ and $f:=g^{1/s}=g^{-r}$, let $m_0:=M(g)$ be the mode of $g$. Consider a class of local perturbations of $g$: For every $\epsilon>0$, define
\begin{equation*}
\tilde{g}_\epsilon(x)=
\begin{cases}
g(m_0-\epsilon c_\epsilon)+(x-m_0+\epsilon c_\epsilon)g'(m_0-\epsilon c_\epsilon) 
& x \in [m_0-\epsilon c_\epsilon,m_0-\epsilon)\\
     g(m_0+\epsilon)+(x-m_0-\epsilon)g'(m_0+\epsilon) 
& x\in [m_0-\epsilon,m_0+\epsilon)\\
g(x)  
& \textrm{otherwise}.
\end{cases}
\end{equation*}
Here $c_\epsilon$ is chosen so that $g_\epsilon$ is continuous at $m_0-\epsilon$. 
This construction of a perturbation class is also seen in 
\cite{balabdaoui2009limit,groeneboom2001estimation}. 
By Taylor expansion at $m_0-\epsilon$ we can easily see $c_\epsilon=3+o(1)$ 
as $\epsilon \to 0$. Since $\tilde{f}_\epsilon:=\tilde{g}_\epsilon^{-r}$ is not a density, we normalize it by
$f_\epsilon(x):=\frac{\tilde{f}_\epsilon(x)}{\int_{\mathbb{R}} \tilde{f}_\epsilon(y) \d{y}}.$
Now $f_\epsilon$ is $s$-concave for each $\epsilon>0$ with mode $m_0-\epsilon$.

The following result follows from direct calculation. 
For a proof, we refer to Appendix section~\ref{appendix:supp_proof_4} .  
\begin{lemma}\label{lem:osc_gap}
Assume (A1)-(A4). Then
\[
h^2(f_\epsilon,f)=\zeta_k\frac{r^2f(m_0)(g^{(k)}(m_0))^2}{g(m_0)^2}\epsilon^{2k+1}+o(\epsilon^{2k+1}),
\]
where
\begin{equation*}
\begin{split}
\zeta_k &= \frac{1}{108(k!)^2(k+1)(k+2)(2k+1)}\bigg[-4\cdot 3^{k+2}(2k+1)(3^{k+2}+k^2+k-3)\\
&\quad +(k+1)(k+2)\bigg(27(3^{2k+1}-1)+2\cdot 3^{2k}(2k+1)(2k(2k-9)+27)\bigg)\bigg]\\
&\quad\quad +\frac{2k^2(2k^2+1)}{3(k!)^2(k+1)(2k+1)}.
\end{split}
\end{equation*}
\end{lemma}
\begin{theorem}\label{minimaxmode}
For an $s$-concave density $f_0$, let $\mathcal{SC}_{n,\tau}(f_0)$ be defined by
\[
\mathcal{SC}_{n,\tau}(f_0):=\left\{f:s\textrm{-concave density}, h^2(f,f_0)\leq {\tau^2 \over n}\right\}.
\]
Let $m_0=M(f_0)$ be the mode of $f_0$. Suppose (A1)-(A4) hold. Then,
\[
\sup_{\tau>0}\liminf_{n \to \infty} n^{1/(2k+1)}\inf_{t_n}\sup_{f \in \mathcal{SC}_{n,\tau}}\mathbb{E}_f\abs{T_n-M(f)}\geq \rho_k\bigg(\frac{g_0(m_0)^2}{r^2f_0(m_0)g_0^{(k)}(m_0)^2}\bigg)^{1/(2k+1)},
\]
where $\rho_k=(2(2k+1)\zeta_k)^{-1/(2k+1)}/4$.
\end{theorem}
\begin{proof}
Take $l(x)=\abs{x}$. Let $\epsilon=cn^{-1/(2k+1)}$, and let $\gamma=\frac{r^2f(m_0)(g^{(k)}(m_0))^2}{g(m_0)^2}$, $f_n:=f_{cn^{-1/(2k+1)}}$. Then $\limsup_{n \to \infty}nh^2(f_n,f)=\zeta_k\gamma c^{2k+1}.$ Applying Theorem \ref{minimaxlowerboundgeneric}, we find that
\[
\liminf_{n \to \infty}n^{1/(2k+1)}R_l(n;T,\{f,f_n\})\geq {1 \over 4}c\exp\left(-2\zeta_k\gamma c^{(2k+1)}\right).
\]
Now we choose $c=(2(2k+1)\zeta_k\gamma)^{-1/(2k+1)}$ to conclude.
\end{proof}



\section{Discussion}\label{discussion}
We show in this paper that the class of $s$-concave densities can be 
approximated and estimated via R\'enyi divergences in a robust and stable way. 
We also develop local asymptotic distribution theory for the divergence estimator, 
which suggests that the convexity constraint is the main complexity within the class 
of $s$-concave densities regardless heavy tails. In the rest of this section, 
we will sketch some related problems and future research directions.

\subsection{Behavior of R\'enyi projection for generic measures $Q$ when $s<-1/(d+1)$}

We have considered in this paper two regions for the index $s$: (1) $-1/(d+1)<s<0$ 
and (2) $-1/d<s\leq -1/(d+1)$. In 
case (1), we showed that starting from a generic measure $Q$ with 
the interior of its convex support non-void and a first moment, the 
R\'enyi projection through (\ref{primalpopulation}) exists and enjoys nice 
continuity properties that cover both on and off-the-model situations. 
In 
case (2), we showed that the R\'enyi projection for the empirical 
measure still enjoys such continuity properties when $Q$ is a 
probability measure corresponding to a true $s$-concave density with a finite first moment. 
	
It remains open to investigate the behavior of the R\'enyi projection in the region 
(2) for a generic measure $Q$. If $Q$ does not admit a first moment, i.e. 
$\int \pnorm{x}{}\ \d{Q}(x)=\infty$, then the first term in the functional 
(\ref{primalpopulation}) diverges for any candidate convex function. 
We conjecture that 
the R\'enyi divergence projection 
fails to exist in this case. 
We do not know if the R\'enyi projection exists when $-1/d < s \le - 1/(d+1)$ and $Q \notin {\cal P}_s$ 
but $\int \pnorm{x}{} dQ(x) < \infty$.

It should be mentioned that the MLEs for the classes ${\cal P}_s$ exist (for an increasingly large 
sample size $n$ as $s \searrow -1/d$), and are Hellinger 
consistent for $-1/d < s < 0$
(cf. \cite{seregin2010nonparametric}).
Moreover, it is known from \cite{doss2013global} that the MLE does not exist for $s < -1/d$.
But we do not yet know any continuity properties of the 
Maximum Likelihood projection ``off the model''.
This leaves the interval $-1/d < s \le -1/(d+1)$ presently without a nicely stable 
nonparametric estimation procedure.  
See \cite{koenker2010quasi} pages 3008 and 3016 
for some further discussion.

\subsection{Global rates of convergence for R\'enyi divergence estimators}
Classical empirical process theory relates the maximum likelihood 
estimators with Hellinger loss via `basic inequalities' as coined in 
\cite{vdG2000} and \cite{van1996weak}. This reduces the problem of 
global rates of convergence to the study of modulus of continuity of empirical process 
indexed by a suitable transformation of the function class of interest. 
We expect that similar `basic inequalities' can be exploited to relate the 
R\'enyi divergence estimators to some divergence 
(not necessarily Hellinger distance). 
We also expect some uniformity in the rates of convergence for the 
R\'enyi divergence estimators as observed by \cite{kim2014global} in 
the case of the MLEs for log-concave densities.
\subsection{Conjectures about the global rates in higher dimensions}
It is now well-understood from the work of \cite{doss2013global} that the 
MLEs for $s$-concave densities($-1<s<0$) and log-concave densities in 
dimension $1$ converge at rates no worse than $O_p(n^{-2/5})$ in Hellinger loss. 
In higher dimensions, \cite{kim2014global} provide an important lower 
bound on the bracketing entropy for a subclass of log-concave densities 
on the order of $O(\epsilon^{-(d/2)\vee (d-1)})$ in Hellinger distance, 
and a matching upper bound up to logarithmic factors for $d\leq 3$. 
Lack of corresponding results in discrete convex geometry precludes 
further upper bounds beyond $d=3$. If a matching upper bound can be 
achieved for $d\geq 4$ (with possible logarithmic losses), the rates of 
convergence $r_n^2$ in squared Hellinger distances become
\begin{equation*}
r_n^2= O(n^{-1/(d-1)}),  d\geq 4
\end{equation*}
(up to logarithmic factors). It is also worth mentioning that 
minimum contrast estimator may well be rate inefficient in higher 
dimensions, as observed by \cite{MR1240719} in another context 
with `trans-Donsker' class of functions. Therefore it is also interesting 
to design sieved/regularized estimator to achieve the efficient rates.
\subsection{Adaptive estimation of concave-transformed class of functions}
The rates conjectured above are conservative in that they are derived 
from the \emph{global} point of view. From a local perspective, adaptive 
estimation may be possible when the underlying function/density exhibits 
special structures. In fact, it is shown by \cite{guntuboyina2013global} that in 
the univariate convex regression setting, if the underlying convex function is 
piecewise linear, then the rate of convergence for the global risk in the 
discrete $l_2$ norm adapts to nearly parametric rate $n^{-1/2}$ (up to logarithmic factors). 
It would be interesting to examine if same phenomenon can be observed for 
the MLEs/R\'enyi divergence estimators, and more generally for minimum 
contrast estimators of concave-transformed classes of functions.



\section{Proofs}\label{sec:proofs}

\subsection{Proofs for Section 2}\label{appendix:supp_proof_2}
\begin{proof}[Proof of Lemma \ref{lem:finitefunctional}]
Let $Q \in \mathcal{Q}_1$. Then by letting $g(x):=\pnorm{x}{}+1$, we have
\[L(Q)\leq L(g,Q)=\int (1+\pnorm{x}{})\ \d{Q}+{1 \over \abs{\beta}}\int \frac{\d{x}}{(1+\pnorm{x}{})^{-\beta}}<\infty,
\]
by noting $Q \in \mathcal{Q}_1$, and $-\beta=-1-1/s>d$. Now assume $L(Q)<\infty$. If $Q\notin \mathcal{Q}_1$, i.e.
$\int \pnorm{x}{}\ \d{Q}=\infty$, then since for each $g \in \mathcal{G}$, we can find some $a,b>0$ such that $g(x)\geq a\pnorm{x}{}-b$, we have
\[L(g,Q)=\int g\ \d{Q}+{1 \over \abs{\beta}}\int g^{\beta}\ \d{x}\geq \int (a\pnorm{x}{}-b)\ \d{Q}=\infty,\]
a contradiction. This implies $Q \in \mathcal{Q}_1$.
\end{proof}

\begin{proof}[Proof of Theorem \ref{existence}]
	We note that $L(Q)<\infty$ by Lemma \ref{lem:finitefunctional}. 
	Hence we can take a sequence $\{g_n\}_{n \in \N} \subset \mathcal{G}$ 
	such that $\infty>M_0\geq L(g_n,Q)\searrow L(Q)$ as $n \to \infty$ for some $M_0>0$. 
	Now we claim that, for all $ x_0 \in \textrm{int(csupp}(Q))$,
	\begin{equation}\label{intubexist}
	\sup_{n \in \N} g_n(x_0)<\infty.
	\end{equation}
	Denote $\epsilon_n\equiv \inf_{x \in \R^d} g_n(x)$. First note,
	\begin{equation*}
	\begin{split}
	L(g_n,Q)&\geq \int g_n\ \d{Q}=\int g_n{\bf 1}(g_n\leq g_n(x_0))\ \d{Q}+\int g_n{\bf 1}(g_n> g_n(x_0))\ \d{Q}\\
	&=\int \big(g_n-g_n(x_0)+g_n(x_0)\big){\bf 1}(g_n\leq g_n(x_0))\ \d{Q}+\int g_n{\bf 1}(g_n> g_n(x_0))\ \d{Q}\\
	&\geq g_n(x_0)-\big(g_n(x_0)-\epsilon_n \big)Q\big(\{g_n(\cdot)\leq g_n(x_0)\}\big).\\
	\end{split}
	\end{equation*}
	If $g_n(x_0)>\epsilon_n$, then $x_0$ is not an interior point of the 
	closed convex set $\{g_n\leq g_n(x_0)\}$, which implies 
	$Q\big(\{g_n(\cdot)\leq g_n(x_0)\}\big)\leq h(Q,x)$, where 
	$h(\cdot,\cdot)$ is defined in Lemma \ref{intpoint}. 
	Hence, in this case, the above term is lower bounded by
	\[
	L(g_n,Q)\geq g_n(x_0)-\big(g_n(x_0)-\epsilon_n\big)h(Q,x_0)\geq g_n(x_0)\big(1-h(Q,x_0)\big).\\
	\]
	This inequality also holds for $g_n(x_0)=\epsilon_n$, which implies that
	\[
	g_n(x_0)\leq \frac{L(g_n,Q)}{1-h(Q,x_0)}\leq \frac{M_0}{1-h(Q,x_0)}.
	\]
	by the first statement of Lemma \ref{intpoint}. Thus we verified (\ref{intubexist}). 
	Now invoking Lemma \ref{globallowerbound}, and we check conditions (A1)-(A2) as follows:
	(A1) follows by (\ref{intubexist});
	(A2) follows by the choice of $g_n$ since $\sup_{n \in \N}L(g_n,Q)\leq M_0$. 
	By Lemma \ref{convsubsequence} we can find a subsequence 
	$\{g_{n(k)}\}_{k\in\N}$ of $\{g_n\}_{n \in \N}$, and a function $\tilde{g} \in \mathcal{G}$ such that 
	$\{x \in \R^d:\sup_{n \in \N} g_n(x)<\infty\}\subset \textrm{dom}(\tilde{g})$, and
	\begin{equation*}
	\begin{split}
	&\lim_{k \to \infty,x \to y}g_{n(k)}(x)=\tilde{g}(y),\quad \textrm{ for all } y \in \textrm{int(dom}(\tilde{g})),\\
	&\liminf_{k \to \infty, x \to y}g_{n(k)}(x)\geq \tilde{g}(y),\quad \textrm{ for all } y \in \R^d.\\
	\end{split}
	\end{equation*}
	Again for simplicity we assume that $\{g_n\}$ satisfies the above properties. We note that
	\begin{equation*}
	\begin{split}
	L(Q) &= \lim_{n \to \infty}\bigg(\int g_n\ \d{Q} + {1 \over \abs{\beta}}\int g_n^\beta\ \d{x}\bigg)\\
	&\geq \liminf_{n \to \infty}\int g_n\ \d{Q}+{1 \over \abs{\beta}}\liminf_{n \to \infty}\int g_n^\beta\ \d{x}\\
	&\geq \int \tilde{g}\ \d{Q}+{1 \over \abs{\beta}}\int {\tilde{g}}^\beta\ \d{x}=L(\tilde{g},Q)\geq L(Q),\\
	\end{split}
	\end{equation*}
	where the third line follows from Fatou's lemma for the first term, and 
	Fatou's lemma and the fact that the boundary of a convex set has 
	Lebesgue measure zero for the second term (Theorem 1.1, \cite{lang1986}). 
	This establishes $L(\tilde{g},Q)=L(Q)$, and hence $\tilde{g}$ is the desired minimizer. 
	Since $\tilde{g}\in \mathcal{G}$ achieves its minimum, we may assume 
	$x_0 \in \mathrm{Arg}\min_{x \in \R^d} \tilde{g}(x)$. If $\tilde{g}(x_0)=0$, 
	since $\tilde{g}$ has domain with non-empty interior, we can choose 
	$x_1,\ldots,x_d \in \mathrm{dom}(\tilde{g})$ such that $\{x_0,\ldots,x_d\}$ 
	are in general position. Then by Lemma \ref{generalposition} we find 
	$L(\tilde{g},Q)=\infty$, a contradiction. This implies $\tilde{g}$ must be bounded away from zero.
	
	For the last statement, since $\tilde{g}$ is a minimizer of (\ref{primalpopulation}), 
	and the fact that $\tilde{g}$ is bounded away from zero, then 
	$L(\tilde{g}+c,Q)$ is well-defined for all $ \abs{c}\leq \delta$ with small $\delta>0$, and we must necessarily have 
	$\frac{\mathrm{d}}{\d{c}}L(\tilde{g}+c,Q)\vert_{c=0}=0.$ On the other hand it is easy to calculate that 
	$\frac{\mathrm{d}}{\d{c}}L(\tilde{g}+c,Q)=1-\int \big(\tilde{g}(x)+c\big)^{\beta-1}\ \d{x}.$
	This yields the desired result by noting $\beta-1=1/s$.
\end{proof}

\begin{proof}[Proof of Lemma \ref{unique}]
	Let $g,h$ be two minimizers for $\mathcal{P}_Q$. Since $\psi_s(x)={1 \over \abs{\beta}}x^{\beta}$ is strictly convex on $[0,\infty)$, $L(t\cdot g+(1-t)\cdot h,Q)$ is strictly convex in $t\in[0,1]$ unless $g=h$ a.e. with respect to the canonical Lebesgue measure. We claim if two closed functions $g,h$ agree a.e. with respect to the canonical Lebesgue measure, then it must agree everywhere, thus closing the argument. It is easy to see $\textrm{int}(\textrm{dom} g)= \textrm{int}(\textrm{dom} h)$. Since $\textrm{int}(\textrm{dom}(g))\neq \emptyset$, we have
	$\textrm{ri}(\textrm{dom} g)=\textrm{int}(\textrm{dom} g)=\textrm{int}(\textrm{dom} h)=\textrm{ri}(\textrm{dom} h).$
	Also note that a convex function is continuous in the interior of its domain, and hence almost everywhere equality implies everywhere equality within the interior of the domain, i.e. $g\big\vert_{\textrm{int}(\textrm{dom} g)}=h\big\vert_{\textrm{int}(\textrm{dom} h)}.$	Now by Corollary 7.3.4 in \cite{rockafellar1997convex}, and the closedness of $g,h$, we find that $g=\textrm{cl}g=\textrm{cl} h=h$.
\end{proof}

\begin{proof}[Proof of Theorem \ref{mainthm1}]
	\noindent To show (\ref{momentliminf}), we use Skorohod's theorem: 
	since $Q_n\to_d Q$, there exist random vectors $X_n \sim Q_n$ and 
	$X\sim Q$ defined on a common probability space $(\Omega,\mathcal{B},\mathbb{P})$ 
	satisfying $X_n\to_{a.s.} X$. Then by Fatou's lemma, we have 
	$\int \pnorm{x}{}\d{Q}=\mathbb{E}[\pnorm{X}{}]
	\leq \liminf_{n \to \infty}\mathbb{E}[\pnorm{X_n}{}]=\liminf_{n \to \infty} \int \pnorm{x}{}\d{Q_n}.$
	
	Assume (\ref{momentconv}). We first claim that
	\begin{equation}\label{direction2}
	\limsup_{n \to \infty}L(Q_n) \leq L(g,Q)=L(Q).
	\end{equation}
	Let $g_n(\cdot),g(\cdot)$ be defined as in the statement of the theorem. 
	Note that 
	$\limsup_{n \to \infty}L(g_n,Q_n)\leq \lim_{n \to \infty}L(g^{(\epsilon)},Q_n)=L(g^{(\epsilon)},Q)$. 
	Here $g^{(\epsilon)}$ is the Lipschitz approximation of $g$ defined in Lemma \ref{Lip}, 
	and the last equality follows from the moment convergence condition (\ref{momentconv}) 
	by rewriting  $g^{(\epsilon)}(x)=\frac{g^{(\epsilon)}(x)}{1+\pnorm{x}{}}(1+\pnorm{x}{})$, 
	and note the Lipschitz condition on $g^{(\epsilon)}$ implies boundedness of 
	$\frac{g^{(\epsilon)}(x)}{1+\pnorm{x}{}}$. By construction of $\{g^{(\epsilon)}\}_{\epsilon>0}$ 
	we know that if $x_0$ is a minimizer of $g$, then it is also a minimizer of $g^{(\epsilon)}$. 
	This implies that the function class $\{g^{(\epsilon)}\}_{\epsilon>0}$ is bounded away from 
	zero since $g$ is bounded away from zero by Theorem \ref{existence}, i.e. 
	$\inf_{x \in \R^d}g^{(\epsilon)}(x)\geq \epsilon_0$ holds for all $\epsilon>0$ with some 
	$\epsilon_0>0$. Now let $\epsilon \searrow 0$, in view of Lemma \ref{Lip}, by 
	the monotone convergence theorem applied to $g^{(\epsilon)}$ and 
	$\epsilon_0^\beta - (g^{(\epsilon)})^\beta$ we have verified (\ref{direction2}).
	
	Next, we claim that, for all $x_0 \in \intdom{Q}$,
	\begin{equation}\label{intub}
	\limsup_{n \to \infty} g_n(x_0)<\infty.
	\end{equation}
	Denote $\epsilon_n\equiv \inf_{x \in \R^d} g_n(x)$. Note by essentially the same 
	argument as in the proof of Theorem \ref{existence}, we have
	\[
	g_n(x_0)\leq \frac{L(Q_n)}{1-h(Q_n,x_0)}.
	\]
	By taking $\limsup$ as $n \to \infty$, (\ref{intub}) follows by virtue of Lemma \ref{intpoint} and (\ref{direction2}).
	
	Now we proceed to show (\ref{functionalconv}) and (\ref{pointwiseconv}). 
	By invoking Lemma \ref{globallowerbound}, we can easily check that all 
	conditions are satisfied (note we also used (\ref{direction2}) here). 
	Thus we can find a subsequence $\{g_{n(k)}\}_{k\in\N}$ of $\{g_n\}_{n \in \N}$ 
	with $g_{n(k)}(x)\geq a\pnorm{x}{}-b,$ holds for all $x \in \R^d$ and all $k \in \N$ 
	with some $a,b>0$. Hence by Lemma \ref{convsubsequence}, we can find a 
	function $\tilde{g} \in \mathcal{G}$ such that 
	$\{x \in \R^d:\limsup_{k \to \infty} g_{n(k)}(x)<\infty\}\subset \textrm{dom}(\tilde{g}),$
	and that
	\begin{equation*}
	\begin{split}
	&\lim_{k \to \infty,x \to y}g_{n(k)}(x)=\tilde{g}(y),\quad \textrm{for all } y \in \textrm{int(dom}(\tilde{g})),\\
	&\liminf_{k \to \infty, x \to y}g_{n(k)}(x)\geq \tilde{g}(y),\quad \textrm{for all } y \in \R^d.\\
	\end{split}
	\end{equation*}
	Again for simplicity we assume $\{g_n\}$ admit the above properties. 
	Now define random variables $H_n\equiv g_n(X_n)-(a\pnorm{X_n}{}-b)$. 
	Then by the same reasoning as in the proof of Theorem \ref{existence}, we have
	\begin{equation*}
	\begin{split}
	\liminf_{n \to \infty}L(Q_n)&=\liminf_{n \to \infty}\bigg(\int g_n\ \d{Q_n} + {1 \over \abs{\beta}}\int g_n^\beta\ \d{x}\bigg)\\
	&\geq \liminf_{n \to \infty}\mathbb{E}[H_n+a(X_n)-b]+{1 \over \abs{\beta}}\int {\tilde{g}}^\beta\ \d{x}\\
	&\geq  \mathbb{E}[\liminf_{n \to \infty} H_n]
	+a\liminf_{n \to \infty}\int \pnorm{x}{}\ \d{Q_n}-b+{1 \over \abs{\beta}}\int {\tilde{g}}^\beta\ \d{x}\\
	&=L(\tilde{g},Q)+a\bigg(\liminf_{n \to \infty}\int\pnorm{x}{}\ \d{Q}_n-\int \pnorm{x}{}\ \d{Q}\bigg)\\
	&\geq L(Q)+a\bigg(\liminf_{n \to \infty}\int\pnorm{x}{}\ \d{Q}_n-\int \pnorm{x}{}\ \d{Q}\bigg),\\
	\end{split}
	\end{equation*}
	Note the expectation is taken with respect to the probability space 
	$(\Omega,\mathcal{B},\mathbb{P})$ defined above. 
	This establishes that if (\ref{momentconv}) holds true, then
	\begin{equation}\label{direction1}
	\liminf_{n \to \infty}L(Q_n)\geq L(\tilde{g},Q)\geq L(Q).
	\end{equation}
	Conversely, if (\ref{momentconv}) does not hold true, then there exists 
	a subsequence $\{Q_{n(k)}\}$ such that $\liminf_{k \to \infty}\int\pnorm{x}{}\ \d{Q}_{n(k)}>\int \pnorm{x}{}\ \d{Q}$. 
	However, this means that 
	$ \liminf_{k \to \infty}L(Q_{n(k)})>L(Q),$
	which contradicts (\ref{functionalconv}). Hence if (\ref{functionalconv}) holds, 
	then  (\ref{momentconv}) holds true. Combine (\ref{direction1}) and (\ref{direction2}), 
	and by virtue of Lemma \ref{unique}, we find $\tilde{g}\equiv g$. 
	This completes the proof for (\ref{functionalconv}) and (\ref{pointwiseconv}).
	
	We show (\ref{L1conv}). First we claim that 
	$\{\hat{x}_n \in \mathrm{Arg}\min_{x \in \R^d}g_n(x)\}_{n\in \N}$
	is bounded. If not, then we can find a subsequence such that 
	$\pnorm{\hat{x}_{n(k)}}{}\to \infty$ as $k \to \infty$. However this 
	means that $g_{n(k)}(x)\geq g_{n(k)}(\hat{x}_{n(k)})\geq a\pnorm{\hat{x}_{n(k)}}{}-b \to \infty$ 
	as $k \to \infty$ for any $x$, a contradiction. Next we claim that there exists 
	$\epsilon_0>0$ such that $\inf_{k \in \N} \epsilon_{n(k)}\geq \epsilon_0$ 
	holds for some subsequence $\{\epsilon_{n(k)}\}_{k \in \N}$ of $\{\epsilon_n\}_{n \in \N}$. 
	This can be seen as follows: Boundedness of $\{\hat{x}_n\}$ implies 
	$\hat{x}_{n(k)} \to x^\ast$ as $k \to \infty$ for some subsequence 
	$\{\hat{x}_{n(k)}\}_{k \in \N}\subset \{\hat{x}_n\}_{n \in \N}$ and 
	some $x^\ast\in \R$. Hence by (\ref{pointwiseconv}) we have 
	$\limsup_{k \to \infty} f_{n(k)}(\hat{x}_{n(k)})\leq f(x^\ast)<\infty$, 
	since $f(\cdot)$ is bounded. This implies that $\sup_{k \in \N}\pnorm{f_{n(k)}}{\infty}<\infty$, 
	which is equivalent to the claim. As before, we will understand the notation 
	for whole sequence as a suitable subsequence. Now we have 
	$g_n(x)\geq \big(a\pnorm{x}{}-b\big)\vee \epsilon_0$ holds for all $ x\in \R^d.$ This gives rise to
	\begin{equation}\label{glb}
	f_n(x)\leq \bigg(\big(a\pnorm{x}{}-b\big)\vee \epsilon_0\bigg)^{1/s},\quad \textrm{ for all } x\in \R^d.
	\end{equation}
	Note that $-1/(d+1)<s<0$ implies $1/s < -(d+1)$, whence we get 
	an integrable envelope. Now a simple application of dominated 
	convergence theorem yields the desired result (\ref{L1conv}), in view 
	of the fact that the boundary of a convex set has Lebesgue measure zero (cf. Theorem 1.1 in \cite{lang1986}).
	
	Finally, (\ref{polyLinfty}) and (\ref{eqn:gradient_conv}) are direct 
	results of Theorems \ref{localconv} and \ref{thm:conv_derivatives} 
	by noting that (\ref{L1conv}) entails $f_n\to_d f$ (in the sense that the 
	corresponding probability measures converge weakly).
\end{proof}

\begin{proof}[Proof of Corollary \ref{cor:off-model-consistency}]
	It is known by Varadarajan's theorem (cf. \cite{MR1932358} Theorem 11.4.1), 
	$\mathbb{Q}_n$ converges weakly to $Q$ with probability 1. 
	Further by the strong law of large numbers (SLLN), we know that 
	$\int \pnorm{x}{}\ \d{\mathbb{Q}_n}\to_{a.s.}\int \pnorm{x}{} \d{Q}$. 
	This verifies all conditions required in Theorem \ref{mainthm1}.
\end{proof}

\begin{proof}[Proof of Corollary \ref{cor:on-the-model-consistency}]
	The conclusion follows from Corollary~\ref{cor:off-model-consistency} if $-1/(d+1) < s < 0$, so suppose $-1/d < s \le -1/(d+1)$. 
	Since $f \in \mathcal{P}_{s'}$, we may write $f=g^{1/{s'}}$ where $g$ is convex. If $f$ is unbounded, then $g(x_0)=0$ for some $x_0 \in \R$. 
	By Lemma \ref{generalposition} with $r'=-1/{s'}$, it follows that $\int f=\infty$, contradicting  the fact that $f$ is a density. 
	Thus $f$ must necessarily be bounded. 
	To see that $f$ has a finite mean, note that by
	Lemma \ref{weakconvbound} $f(x) = (b + a \| x \| )^{1/s'}$ where $a,b>0$ and $r'\equiv -1/s'> d+1$.  
	Thus $\int_{\R^d} \| x \| f(x) dx \le \int_{\R^d} \| x \| (b + a \| x \| )^{-r'} dx < \infty$.
	Now note that (\ref{eqn:off-model-ptconv}) holds by the existence of the 
	R\'enyi divergence estimator for the empirical measure (cf. Theorem 4.1 in \cite{koenker2010quasi}) 
	and the same argument in the proof of Theorem \ref{mainthm1}. 
	Also note that by the proof of Theorem \ref{localconv}, (\ref{eqn:off-model-ptconv}) 
	would be enough to ensure (\ref{eqn:off-model-uniform}). Since $f$ is continuous on the interior 
	of the domain, we see that (\ref{eqn:off-model-uniform}) implies weak convergence: 
	let $\hat{Q}_n$ be the measures corresponding to $\hat{f}_n$. 
	Then $\hat{Q}_n \to Q$ weakly as $n \to \infty$. 
	Now the rest follows immediately from Theorems \ref{weakconvmodeconv} and \ref{thm:conv_derivatives}.
\end{proof}

\begin{proof}[Proof of Theorem \ref{projchar}]
	Denote $L(\cdot):=L(\cdot,Q)$. We first claim:
	
	\noindent{\bf Claim.} $g=\arg\min_{g \in \mathcal{G}}L(g)$ 
	if and only if $\lim_{t \searrow 0} \frac{L(g+th)-L(g)}{t}\geq 0,$
	holds for all $h:\R^d \to \R$ such that there exists $t_0>0$ 
	with $g+th \in \mathcal{G}$ holds for all $t\in (0,t_0)$.
	
	To see this, we only have to show sufficiency. Now suppose $g $ 
	is not a minimizer of $L(\cdot)$. By Theorem \ref{existence} we 
	know there exists $\hat{g}\in \mathcal{G}$ such that $\hat{g}=g(\cdot|Q)$. 
	By convexity, we have that for any $t>0$, $L\big(g+t(\hat{g}-g)\big)\leq (1-t)L(g)+t L(\hat{g}).$
	This implies that if we let $h=\hat{g}-g$, and $t_0=1$, then
	\[
	\frac{L(g+th)-L(g)}{t}\leq {1 \over t}\big((1-t)L(g)+t L(\hat{g})-L(g)\big)=-t\big(L(g)-L(\hat{g})\big),
	\]
	and thus
	$\lim_{t \searrow 0} \frac{L(g+th)-L(g)}{t}\leq -\big(L(g)-L(\hat{g})\big)<0$, 
	where the strict inequality follows from Lemma \ref{unique}. 
	This proves our claim. Now the theorem follows from simple calculation:
	\[
	0\leq \lim_{t \searrow 0}{1 \over t}\bigg(L(g+th)-L(g)\bigg)=\int h\ \d{Q}-\int h\cdot g^{1/s}\ \d{\lambda},
	\]
	as desired.
\end{proof}

\begin{proof}[Proof of Corollary \ref{projchar2}]
Let $g\equiv g(\cdot|Q)$. Then by Theorem \ref{existence} and Lemma \ref{convlb}, we find that there exists some $a,b>0$ such that $g(x)\geq a\pnorm{x}{}+b$. Now take $v \in \partial h(0)$, i.e. $h(x)\geq h(0)+v^T x$ holds for all $x \in \R^d$. Hence for $t>0$, we have
\[g(x)+t h(x)\geq a\pnorm{x}{}+b+t(h(0)+v^T x)\geq (a-t\pnorm{v}{})\pnorm{x}{}+(b+th(0)),\]
which implies that $g+th\in \mathcal{G}$ for $t>0$ small enough. Now the conclusion follows from the Theorem \ref{projchar}.
\end{proof}

\begin{proof}[Proof of Theorem \ref{secondintegralchar}]
We first note that if $F$ is a distribution function for a probability measure supported on $[X_{(1)},X_{(n)}]$, and $h:[X_{(1)},X_{(n)}]\to \R$ an absolutely continuous function, then integration by parts (Fubini's theorem) yields
\begin{equation}\label{intbypart}
\int h\ \d{F}=h(X_{(n)})-\int_{X_{(1)}}^{X_{(n)}}h'(x) F(x)\ \d{x}.
\end{equation}
First we assume $g_n=\hat{g}_n$. For fixed $t \in [X_{(1)},X_{(n)}]$, let $h_1$ be a convex function whose derivative is given by $h_1'(x)=-{\bf 1}(x\leq t)$. Now by Theorem \ref{projchar} we find that $\int h_1\ \d{F_n}=\int h_1 \ \d{\hat{F}_n}\leq \int h_1\ \d{\mathbb{F}_n}.$
Plugging in (\ref{intbypart}) we find that $\int_{X_{(1)}}^t F_n(x)\ \d{x}\leq \int_{X_{(1)}}^t\mathbb{F}_n(x)\ \d{x}.$ For $t \in \mathcal{S}_n(g_n)$, let $h_2$ be the function with derivative $h_2'(x)={\bf 1}(x\leq t)$. It is easy to see $g_n+t h_2$ is convex for $t>0$ small enough, whence Theorem \ref{projchar} is valid, thus giving the reverse direction of inequality. This shows the necessity. 

For sufficiency, assume $g_n$ satisfies (\ref{distcond}). In view of the proof of Theorem \ref{projchar}, we only have to show (\ref{projchareqn}) holds for all function $h:\R \to \overline{\R}$ which is linear on $[X_{(i)},X_{(i+1)}](i=1,\ldots,n-1)$ and $g_n+t h$ convex for $t>0$ small enough. Since $g_n$ is a linear function between two consecutive knots, $h$ must be convex between consecutive knots. This implies that the derivative of such an $h$ can be written as $h'(x)=\sum_{j=2}^n\beta_j {\bf 1}(x\leq X_{(j)}),$ with $\beta_2,\ldots,\beta_n$ satisfying
$\beta_j\leq 0$ if $ X_{(j)} \notin \mathcal{S}_n(g_n).$
Now again by (\ref{intbypart}) we have
\begin{equation*}
\begin{split}
\int h\ \d{\hat{F}_n}&=h(X_n)-\sum_{j=2}^n\beta_j\int_{X_{(1)}}^{X_{(j)}}\hat{F}_n(x)\ \d{x}\\
&\leq h(X_n)-\sum_{j=2}^n\beta_j\int_{X_{(1)}}^{X_{(j)}}\mathbb{F}_n(x)\ \d{x}=\int h\ \d{\mathbb{F}_n},
\end{split}
\end{equation*}
as desired.
\end{proof}

\begin{proof}[Proof of Corollary \ref{distchar}]
This follows directly from the Theorem \ref{secondintegralchar} by noting for $x_1<x_0<x_2$ we have
\[{1 \over {x_2-x_0}}\int_{x_0}^{x_2} \hat{F}_n(x)\ \d{x}\leq {1 \over {x_2-x_0}}\int_{x_0}^{x_2} \mathbb{F}_n(x)\ \d{x},\]
and
\[{1 \over {x_0-x_1}}\int_{x_1}^{x_0}\hat{F}_n(x)\ \d{x}\geq {1 \over {x_0-x_1}}\int_{x_1}^{x_0}\mathbb{F}_n(x)\ \d{x}.\]
Now let $x_1 \nearrow x_0$ and $x_2\searrow x_0$ we find that $\hat{F}_n(x_0)\leq \mathbb{F}_n(x_0)$ by right continuity and $\hat{F}_n(x_0)\geq \mathbb{F}_n(x_0-)=\mathbb{F}_n(x_0)-{1 \over n}$.
\end{proof}

\begin{proof}[Proof of Theorem \ref{thm:characterization_distribution}]
The proof closely follows the proof of Theorem 2.7 of \cite{dumbgen2011approximation}. For the reader's convenience we give a full proof here. Let $P$ denote the probability distribution corresponding to $F$. We first show necessity by assuming $g=g(\cdot|Q)$. By Corollary \ref{projchar2} applied to $h(x)=\pm x$, we find by Fubini's theorem that
\[
0=\int_\R x\ \d{(Q-P)(x)}=\int_\R (F-G)(t)\d{t}
\]
which proves (1). Now we turn to (2). Since the map $s\mapsto (s-x)_+$ is convex, again by Corollary \ref{projchar2}, we find
\[0\leq \int_\R (s-x)_+\d{(Q-P)(s)}=-\int_{-\infty}^x(F-G)(t)\ \d{t},\]
where in the last equality we used the proved fact that $\int_\R (F-G)\d{\lambda}=0$. Now we assume $x \in \tilde{\mathcal{S}}(g)$, and discuss two different cases to conclude. If $x \in \partial(\mathrm{dom}(g))$, then let $h(s)=-(s-x)_+$, it is easy to see $g+th \in \mathcal{G}$ for $t>0$ small enough. Then by Theorem \ref{projchar}, we have
\[0\leq \int h(s)\d{(Q-P)(s)}=\int_{-\infty}^x (F-G)(t)\ \d{t}.\]
If $x \in \intdom{g}$, then $g'(x-\delta)<g'(x+\delta)$ for small $\delta>0$ by definition, and hence we define
\[ H_\delta'(u)=-\frac{g'(u)-g'(x-\delta)}{g'(x+\delta)-g'(x-\delta)}{\bf 1}_{\{u \in [x-\delta,x+\delta]\}}-{\bf 1}_{\{u> x+\delta\}},\]
whose integral $H_\delta(s):=\int_{-\infty}^s H_\delta'(u)\ \d{u}$
serves as an approximation of $-(s-x)_+$ as $\delta \searrow 0$. Note that
\[\big(g+t H_\delta\big)(s)=g(s)-\frac{t}{g'(x+\delta)-g'(x-\delta)}\int_{s\wedge (x-\delta)}^{s\wedge (x+\delta)} \big(g'(u)-g'(x-\delta)\big)\ \d{u}-t\big(s-(x+\delta)\big)_+,\]
implying $g+tH_\delta \in \mathcal{G}$ for $t>0$ small enough (which may depend on $\delta$). Then by Theorem \ref{projchar},
\[0\leq \int H_\delta(s)\d{(Q-P)(s)}\to -\int (s-x)_+ \d{(Q-P)(s)}=\int_{-\infty}^x (F-G)(t)\ \d{t},\]
as $\delta \searrow 0$, where the convergence follows easily from dominated convergence theorem. This proves (2). Now we show sufficiency by assuming (1)-(2). Consider a Lipschitz continuous function $\Delta(\cdot)$ with Lipschitz constant $L$. Then
\begin{equation*}
\begin{split}
\int \Delta\d{(Q-P)}&=\int \Delta'(F-G)\ \d{\lambda}=-\int (L-\Delta')(F-G)\ \d{\lambda}\\
&=-\int_\R \bigg(\int_{-L}^L {\bf 1}_{\{s>\Delta'(t)\}}\d{s}\bigg)(F-G)(t) \d{t}\\
&=-\int_{-L}^L\int_{A(\Delta',s)}(F-G)(t)\ \d{t}\d{s},
\end{split}
\end{equation*}
where the second line follows from (1), and $A(\Delta',s):=\{t \in \R:\Delta'(t)<s\}$. Now replace the generic Lipschitz function $\Delta$ with $g^{(\epsilon)}$ as defined in Lemma \ref{Lip} with Lipschitz constant $L=1/\epsilon$. Note in this case $A\big((g^{(\epsilon)})',s\big)=(-\infty,a(g,\epsilon))$, where $a(g,s)=\min\{t\in \R:g'(t+)\geq s\}$ and hence $a(g,s)\in \tilde{\mathcal{S}}(g)$. This implies that $\int_{A\big((g^{(\epsilon)})',s\big)}(F-G)(s)\d{s}=0$ for all $s \in (-L,L)$ by (2), yielding that $
\int g^{(\epsilon)}\ \d{(Q-P)}=0$. Similarly we have $\int g_0^{(\epsilon)}\ \d{(Q-P)}\geq 0$ where $g_0=g(\cdot|Q)$.  Now let $\epsilon \searrow 0$, by monotone convergence theorem we find that $ \int g\ \d{Q}=\int g\ \d{P}$ and that $\int g_0\ \d{Q}\geq \int g_0\ \d{P}$. This yields
\[L(g_0,Q)\geq L(g_0,P)\geq L(g,P)=L(g,Q),\]
where the second inequality follows from the Fisher consistency of functional $L(\cdot,\cdot)$ and the fact that $P$ is the distribution corresponding to $g$.
\end{proof}

Before we prove Theorem \ref{thm:continuity_s}, we will 
need an elementary lemma. 
\begin{lemma}\label{lem:limit_linearity}
	Fix a sequence $0<\alpha_n<1$ with $\alpha_n\nearrow 1$. 
	Let $f_{\alpha_n}$ be an $(\alpha_n-1)$-concave density on $\R$. 
	Let $g_{\alpha_n}:=f_{\alpha_n}^{\alpha_n-1}$ be the underlying convex function. 
	Suppose $\{g_{\alpha_n}\}$'s are linear on $[a,b]$ with 
	$\lim_{n \to \infty}f_{\alpha_n}(a)=\gamma_a \in [0,\infty]$ and 
	$\lim_{n \to \infty}f_{\alpha_n}(b)=\gamma_b \in [0,\infty]$. Then for all $x \in [a,b]$, 
	\begin{equation}
	f_{\alpha_n}(x)\to\exp \bigg(\frac{\log \gamma_b-\log \gamma_a}{b-a}(x-a)+\log \gamma_a\bigg)
	\end{equation}
	where $\exp(-\infty):=0$ and $\exp(\infty):=\infty$.
\end{lemma}

\begin{proof}[Proof of Lemma \ref{lem:limit_linearity}]
	First assume $\gamma_b\neq \gamma_a$ and $\gamma_a,\gamma_b \in (0,\infty)$. For notational convenience we drop explicit dependence on $n$ and the limit is taken as $\alpha \nearrow 1$. Let $\gamma_{a,\alpha}=f_\alpha(a)=g_\alpha(a)^{1/(\alpha-1)}$ and $\gamma_{b,\alpha}=f_\alpha(b)=g_\alpha(b)^{1/(\alpha-1)}$. For any $x \in [a,b]$,
	\begin{equation}\label{eqn:asymp_linear_limit_ope}
	\begin{split}
	\lim_{\alpha\to 1}\log f_\alpha(x)
	&=\lim_{\alpha\to 1}{1 \over {\alpha-1}}
	\log\bigg(\frac{\gamma_{b,\alpha}^{\alpha-1} -\gamma_{a,\alpha}^{\alpha-1}}{b-a}(x-a)
	+\gamma_{a,\alpha}^{\alpha-1}\bigg)\\
	&=\lim_{\alpha\to 1}{1 \over {\alpha-1}}
	\log\bigg(\frac{\gamma_b^{\alpha-1} -\gamma_a^{\alpha-1}}{b-a}(x-a)
	\cdot\frac{\gamma_{b,\alpha}^{\alpha-1} -\gamma_{a,\alpha}^{\alpha-1}}
	{\gamma_{b}^{\alpha-1} -\gamma_{a}^{\alpha-1}}+\gamma_{a,\alpha}^{\alpha-1}\bigg)\\
	&\equiv \log \gamma_{a}
	+ \lim_{\alpha \to 1} {1 \over {\alpha-1}}
	\log\bigg( (\gamma_{b}^{\alpha-1} -\gamma_{a}^{\alpha-1}) \frac{(x-a)}{(b-a)}\cdot \frac{1}{\gamma_{a,\alpha}^{\alpha-1}} 
	\cdot r_{\alpha}  + 1 \bigg ) . 
	\end{split}
	\end{equation}
	Since $\gamma_{a,\alpha}^{\alpha -1} \rightarrow 1$, we claim that it suffices to show that 
	\begin{eqnarray} 
	r_{\alpha} \equiv \frac{\gamma_{b,\alpha}^{\alpha-1} -\gamma_{a,\alpha}^{\alpha-1}}
	{\gamma_{b}^{\alpha-1} -\gamma_{a}^{\alpha-1}} \rightarrow 1
	\ \ \ \mbox{as} \ \ \alpha \rightarrow 1.                              
	\label{RatioFunctionLimitsToOne} 
	\end{eqnarray}
	To see this, assume without loss of generality 
	that $\gamma_a>\gamma_b$ and hence $\gamma_b^{\alpha-1}-\gamma_a^{\alpha-1}>0$. 
	Suppose that (\ref{RatioFunctionLimitsToOne}) holds and let $\epsilon>0$.
	Then the second term on right hand side of (\ref{eqn:asymp_linear_limit_ope}) can be bounded from above by
	\begin{equation*}
	\begin{split}
	&\quad \lim_{\alpha \nearrow 1}\frac{1}{\alpha-1}\log\bigg(\big(\gamma_b^{\alpha-1}-\gamma_a^{\alpha-1}\big)\frac{(x-a)}{(b-a)}(1-\epsilon)+1\bigg)\\
	& = \lim_{\alpha \nearrow 1}\left(\log \gamma_b \cdot \gamma_b^{\alpha-1} 
	- \log \gamma_a \cdot \gamma_a^{\alpha-1}\right)\frac{(x-a)}{(b-a)}(1-\epsilon)\\
	&=(\log \gamma_b-\log\gamma_a)\frac{(x-a)}{(b-a)}(1-\epsilon)
	\end{split}
	\end{equation*}
	where the second line follows from L'Hospital's rule. Similarly we can derive a lower bound:
	\[
	(\log \gamma_b-\log\gamma_a)\frac{(x-a)}{(b-a)}(1+\epsilon).
	\]
	Thus it remains to show that (\ref{RatioFunctionLimitsToOne}) holds.  
	But we can rewrite $r_{\alpha} $ as 
	\begin{eqnarray*}
		r_{\alpha} 
		& = & \frac{c_\alpha^{\alpha-1}-1}{c^{\alpha-1}-1}    \\
		&= & \frac{c^{\alpha-1}(c_\alpha/c)^{\alpha-1} - (c_\alpha/c)^{\alpha-1}+(c_\alpha/c)^{\alpha-1} - 1}
		{c^{\alpha-1}-1}\\
		&= & (c_\alpha/c)^{\alpha-1}+\frac{(c_\alpha/c)^{\alpha-1}-1}{c^{\alpha-1}-1}\\
		& \rightarrow & 1 + 0 \ \ \mbox{as} \ \ \alpha \rightarrow 1
	\end{eqnarray*}
	since $\log ( ( c_{\alpha}/c)^{\alpha-1} ) = (\alpha-1) \log (c_{\alpha}/c) \rightarrow 0 \cdot \log 1 = 0$, 
	and where the second limit follows from an upper and lower bound argument using $c_{\alpha}/c \rightarrow 1$. 
	where $c_\alpha:=\gamma_{b,\alpha}/\gamma_{a,\alpha}$ and $c=\gamma_{b}/\gamma_{a}\neq 1$. 
	
	This shows that (\ref{RatioFunctionLimitsToOne}) holds, thereby proving the case for $\gamma_a\neq \gamma_b \in (0,\infty)$. 
	For the case $\gamma_b=\gamma_a \in (0,\infty)$, similarly we have
	\begin{equation*}
	\begin{split}
	\lim_{\alpha\to 1}\log f_\alpha(x)
	&=\log \gamma_{a}+\lim_{\alpha \to 1} {1 \over {\alpha-1}}\log\bigg(\frac{c_\alpha^{\alpha-1} -1}{b-a}(x-a)+1\bigg).
	\end{split}
	\end{equation*}
	The second term is $0$ by an argument much as above by observing 
	$c_\alpha=\gamma_{b,\alpha}/\gamma_{a,\alpha} \to \gamma_b / \gamma_a =1$. 
	Finally, if $\gamma_a\wedge \gamma_b=0$, then by the first line of (\ref{eqn:asymp_linear_limit_ope}) we see that 
	$\log f_\alpha(x)\to -\infty$; if $\gamma_a\vee \gamma_b =\infty$, then again $\log f_\alpha(x)\to \infty$. 
\end{proof}

\begin{proof}[Proof of Theorem \ref{thm:continuity_s}]
	In the following, the notation $\sup_\alpha,\inf_\alpha,\lim_\alpha$ is 
	understood as taking corresponding operation over $\alpha$ close to 1 unless otherwise specified.
	We first show almost everywhere convergence by invoking Lemma 
	\ref{convsubsequence}. To see this, for fixed $s_0 \in (-1/2,0)$, let 
	$g_\alpha:= f_\alpha^{\alpha-1}$ and $ g_\alpha^{(s_0)}:=(f_\alpha)^{s_0}$. 
	Then for $\alpha>1+s_0$, the transformed function $g_\alpha^{(s_0)}$ is convex. 
	We need to check two conditions in order to apply Lemma \ref{convsubsequence} as follows:
	\begin{enumerate}
		\item[(C1)] The set $(X_{(1)},X_{(n)})\subset \{\liminf_\alpha f_\alpha(x)>0\}$;
		\item[(C2)] There is a uniform lower bound function $\tilde{g}^{s_0} \in \mathcal{G}$ 
		such that $g^{(s_0)}_\alpha \geq \tilde{g}^{s_0}$ holds for $\alpha$ sufficiently close to $1$.
	\end{enumerate}
	
	The first assertion can be checked by using the characterization 
	Theorem \ref{secondintegralchar}. Let $F_\alpha$ be the distribution function of $f_\alpha$. Then
	$\int_{X_{(1)}}^{t} (F_\alpha-\mathbb{F}_n)(x)\ \d{x}\leq 0$ with equality 
	attained if and only if $t \in \mathcal{S}_n(g_\alpha)$. 
	For $x \in (X_{(1)},X_{(n)})$ close enough to $X_{(n)}$, we claim that 
	$\liminf_\alpha f_\alpha(x)>0$. If not, we may assume without loss of generality 
	that $\lim_\alpha f_\alpha(x)=0$. We first note that there exists some $t \in \{1,\cdots n-1\}$ 
	and some subsequence $\{\alpha(\beta)\}_{\beta \in \N}$ with $\alpha(\beta) \nearrow 1$ 
	for which (1) $X_{(t)}$ is a knot point for $\{g_{\alpha(\beta)}\}$, and (2) $X_{(u)}$ is 
	not a knot point for any $\{g_{\alpha(\beta)}\}$ for $u\geq t+1$, i.e. $g_{\alpha(\beta)}$'s 
	are linear on $[X_{(t)},X_{(n)}]$. We drop $\beta$ for notational simplicity and assume 
	without loss of generality that both limits 
	$\lim_\alpha f_\alpha(X_{(n)}),\lim_\alpha f_\alpha(X_{(t)})$ exist. 
	Now Lemma \ref{lem:limit_linearity} shows that $\min\{\lim_\alpha f_\alpha(X_{(n)}),\lim_\alpha f_\alpha(X_{(t)})\}=0$ 
	since we have assumed $\lim_\alpha f_\alpha(x)=0$ for some $x \in (X_{(t)},X_{(n)})$. 
	This in turn implies that $\lim_\alpha f_\alpha(x)=0$ for all $x \in (X_{(t)},X_{(n)})$. 
	Now we consider the following two cases to derive a contradiction with the fact
	\begin{equation}\label{eqn:equal_intdist_knot}
	\int_{X_{(t)}}^{X_{(n)}}F_\alpha(x)\d{x}=\int_{X_{(t)}}^{X_{(n)}} \mathbb{F}_n(x)\d{x}
	\end{equation}
	that follows from Theorem \ref{secondintegralchar}, thereby proving 
	$\liminf_\alpha f_\alpha(x)>0$ for $x$ close enough to $X_{(n)}$.
	
	\noindent \textbf{[Case 1.]} If $\lim_\alpha f_\alpha(X_{(n)})=0$, 
	then the left hand side of (\ref{eqn:equal_intdist_knot}) converges to $X_{(n)}-X_{(t)}$ 
	while the right hand side is no larger than $\frac{n-1}{n}\big(X_{(n)}-X_{(t)}\big)$.
	
	\noindent \textbf{[Case 2.]}. If $\lim_\alpha f_\alpha(X_{(n)})>0$, 
	then we must necessarily have $\lim_\alpha f_\alpha(x)=0$ for all 
	$x \in [X_{(1)},X_{(n)})$ by convexity of $g_\alpha$: If 
	$\lim_\alpha f_\alpha(x_0)>0$ for some $x_0 \in [X_{(1)},X_{(t)}]$, then 
	$\lim_{\alpha}g_\alpha(x_0)\vee g_\alpha(X_{(n)})<\infty$ while 
	$\lim_{\alpha}g_\alpha(x)=\infty$ for all $x \in (X_{(t)},X_{(n)})$, which is absurd. 
	Note that this also forces $\lim_\alpha f_\alpha(X_{(n)})=\infty$, otherwise the 
	constraint $\int f_\alpha =1$ will be invalid eventually. Now the left hand side of 
	(\ref{eqn:equal_intdist_knot}) converges to $0$ while the right hand side is 
	bounded from below by ${1 \over n}(X_{(n)}-X_{(t)})$.
	
	Similarly we can show 
	$\liminf_\alpha f_\alpha(x)>0$ for $x$ close to $X_{(1)}$. Now (C1) follows by convexity of $f_\alpha$.
	
	(C2) can be seen by first noting $M:=\sup_\alpha \pnorm{f_\alpha}{\infty}<\infty$.
	This can be verified by Lemma \ref{unifbound} combined with the first 
	assertion proved above. This implies that the class $\{g^{(s_0)}_\alpha\}_\alpha$ 
	has a uniform lower bound $M^{s_0}$. Now (C2) follows by noting that the 
	domain of all $g^{(s_0)}_\alpha$ is $\mathrm{conv}(\underline{X})$. 
	Therefore all conditions needed for Lemma \ref{convsubsequence} are valid, 
	and hence we can extract a subsequence $\{g_{\alpha_n}^{(s_0)}\}_{n \in \N}$ such that
	\begin{equation*}
	\begin{split}
	\lim_{n \to \infty,x \to y}g^{(s_0)}_{\alpha_n}(x)=g^{(s_0)}(y),
	&\quad\textrm{ for all } y \in \intdom{g^{(s_0)}};\\
	\lim_{n \to \infty,x \to y}g^{(s_0)}_{\alpha_n}(x)\geq g^{(s_0)}(y),
	&\quad\textrm{ for all } y \in \R^d,
	\end{split}
	\end{equation*}
	holds for some $g^{(s_0)}\in \mathcal{G}$. This implies $f_{\alpha_n} \to_{a.e.} f^{(s_0)}$ 
	as $n \to \infty$ where $f^{(s_0)}:=\big(g^{(s_0)}\big)^{1/s_0}$. 
	Now repeat the above argument with another $s_1$ with a further 
	extracted subsequence $\{\alpha_{n(k)}\}$, we see that 
	$f_{\alpha_{n(k)}} \to_{a.e.} f^{(s_1)}(k \to \infty)$ for some $s_1$-concave $f^{(s_1)}$ 
	holds for the subsequence $\{\alpha_{n(k)}\}_{k \in \N}$. This implies that $f^{(s_0)}=_{a.e.}f^{(s_1)}$. 
	Since a convex function is continuous in the interior of the domain, we can choose a version 
	of upper semi-continuous $f$ such that $f=f^{(s)}$ a.e. for all $\{1/2<s<0\}\cap \Q$. 
	This implies that $f$ is $s$-concave for any rational $1/2<s<0$ and hence log-concave. 
	Next we show weighted $L_1$ convergence: For fixed $\kappa>0$, choose 
	$0>s_0>-1/(\kappa+1)$. Since there exists $a,b>0$ such that 
	$g_{\alpha_n}^{(s_0)}\geq g^{(s_0)}\geq a\pnorm{x}{}-b$ holds for all $n \in \N$, 
	we have an integrable envelope function:
	\[
	\big(1+\pnorm{x}{}\big)^\kappa \big(f_{\alpha_n}(x)\vee f(x)\big)
	\leq \big(1+\pnorm{x}{}\big)^\kappa\bigg(\big(a\pnorm{x}{}-b\big)\vee M\bigg)^{1/s_0}.
	\]
	Now an application of the dominated convergence theorem yields the 
	desired weighted $L_1$ convergence. Similar arguments show weighted 
	convergence is also valid in arbitrary $L_p$ norms ($p\geq 1$).
	
	Finally we show that $f=f_1$ by virtue of Theorem 2.2 in \cite{dumbgen2009maximum} 
	and Theorem \ref{projchar}. We note that by Lemma \ref{lem:limit_linearity}, 
	$f$ must be log-linear between consecutive data points. Now since $f_1$ and $f$ are 
	both log-linear between consecutive data points of $\{X_1,\ldots,X_n\}$, 
	we only have to consider test functions $h$ such that $h$ is piecewise linear on 
	consecutive data points. Recall $g_\alpha=f_\alpha^{\alpha-1}$ and $g:=-\log f$ 
	are the underlying convex functions for $f_\alpha$ and $f$. For any such $h$ with the 
	property that, $g+th\in \mathcal{G}$ for $t$ small enough, we wish to argue that 
	such $h$ is also a valid test for $f_\alpha$(i.e. $g_\alpha+th \in \mathcal{G}$ for $t>0$ 
	small enough), for a sequence of $\{\alpha_k\}$ converging up to 1 as $k \to \infty$. 
	Thus we only have to argue that for all $X_{(i)} \in \mathcal{S}(g)$, $X_{(i)} \in \mathcal{S}(g_\alpha)$ 
	for a sequence of $\{\alpha_k\}$ going up to 1 as $k \to \infty$. 
	Assume the contrary that $X_{(i)} \notin \mathcal{S}(g_\alpha)$ for all $\alpha$ close enough to $1$. 
	Then $\{g_\alpha\}$'s are all linear on a closed interval $I=[a,b]$ containing 
	$X_{(i)}$ for $\alpha$ close to $1$. Since $f_\alpha \to f$ uniformly on $I$ by 
	Theorem \ref{localconv}, in particular $f_\alpha(a)$ and $f_\alpha(b)$ converges, 
	Lemma \ref{lem:limit_linearity} entails that $f$ is log-linear over $I$, a contradiction to the 
	fact $X_{(i)} \in \mathcal{S}(g)$. Hence we can find a subsequence $\{\alpha_k\}$ going up to $1$ 
	as $k \to \infty$ such that for all $X_{(i)} \in \mathcal{S}(g)$, $X_{(i)} \in \mathcal{S}(g_{\alpha_k})$, 
	i.e. for all feasible test function $h$ of $f_1$, being linear on consecutive data points, is 
	also valid for $f_{\alpha_k}$. Now combining the fact that $f_{\alpha_k}$ converges in 
	$L_2$ metric to $f$ and Theorem 2.2 in \cite{dumbgen2009maximum} we conclude $f_1=f$.
\end{proof}

\subsection{Proofs for Section 3}\label{appendix:supp_proof_3}
\begin{proof}[Proof of Lemma \ref{lem:relate_support_set}]
The proof closely follows the first part of the proof of Proposition 2 \cite{kim2014global}. Suppose $\dim \big(\mathrm{csupp}(\nu)\big)=d$, we show $\mathrm{csupp}(\nu)\subset\overline{C}$. To see this, we take $x_0 \notin \overline{C}$, then there exists $\delta>0$ such that $B(x_0,\delta)\subset C^c$, and we claim that
\begin{equation}\label{claim1}
\textrm{For all }x^\ast \in B(x_0,\delta)\subset C^c, x^\ast \notin \mathrm{int(csupp}(\nu)).
\end{equation}
If (\ref{claim1}) holds, then $x_0 \notin \mathrm{csupp}(\nu)$ and hence $\mathrm{csupp}(\nu)\subset \overline{C}$. Now we turn to show (\ref{claim1}). Since $x^\ast \notin C=\{\liminf_{n \to \infty}f_n(x)>0\}$, we can find a subsequence $\{f_{n(k)}\}_{k \in \N}$ of $\{f_n\}_{n \in \N}$ such that $f_{n(k)}(x^\ast)<{1 \over k}$ holds for all $k \in \N$. Hence $x^\ast \notin \Gamma_k:=\{x \in \R^d: f_{n(k)}(x)\geq {1 \over k}\}$. Note that $\Gamma_k$ is a closed convex set, hence by Hyperplane Separation Theorem we can find $b_k \in \R^d$ with $\pnorm{b_k}{}=1$ such that $\{x\in \R^d:\iprod{b_k}{x}\leq \iprod{b_k}{x^\ast}\}\subset (\Gamma_k)^c.$
Without loss of generality we may assume $b_k \to b_{x^\ast}$ as $k \to \infty$ for some $b_{x^\ast}\in \R^d$ with $\pnorm{b_{x^\ast}}{}=1$. Now for fixed $R>0$ and $\eta>0$, define
\[A_{R,\eta}:=\{x \in \R^d: \iprod{b_{x^\ast}}{x}<\iprod{b_{x^\ast}}{x^\ast}-\eta,\pnorm{x}{}\leq R\}.\]
Choose $k_0 \in \N$ large enough such that $\pnorm{b_k-b_{x^\ast}}{}\leq {\eta \over {2R}}$ holds for all $k\geq k_0(x^\ast,\eta,R)$. Now for $R>\pnorm{x^\ast}{}$ and $x \in A_{R,\eta}$, we have
\[\iprod{b_k}{x-x^\ast}=\iprod{b_{x^\ast}}{x-x^\ast}+\iprod{b_k-b_{x^\ast}}{x-x^\ast}<-\eta+{\eta \over {2R}}(\pnorm{x}{}+\pnorm{x^\ast}{})\leq 0\]
holds for all $k\geq k_0(x^\ast,\eta,R)$. This implies for $R>\pnorm{x^\ast}{}$ and $\eta>0$,
\[A_{R,\eta}\subset \{x\in \R^d:\iprod{b_k}{x}\leq \iprod{b_k}{x^\ast}\}\subset (\Gamma_k)^c=\{x \in \R^d:f_{n(k)}(x)<{1 \over k}\}.\]
Now note $A_{R,\eta}$ is open, by Portmanteau Theorem we find that
\[
\nu(A_{R,\eta})\leq \liminf_{k \to \infty} \nu_{n(k)}(A_{R,\eta})= \liminf_{k \to \infty} \int_{A_{R,\eta}}f_{n(k)}(x)\ \d{x}\leq \liminf_{k \to \infty}\frac{\lambda_d(A_{R,\eta})}{k}=0.
\]
This implies
\[\nu\big(\{x \in \R^d:\iprod{b_{x^\ast}}{x}<\iprod{b_{x^\ast}}{x^\ast}\}\big)=\nu\bigg(\bigcup_{R=1}^\infty A_{R,1/R}\bigg)=\lim_{R \to \infty}\nu(A_{R,1/R})=0,\]
where the second equality follows from the fact $\{A_{R,1/R}\}$ is an increasing family as $R$ increases. By the assumption that $\dim \big(\mathrm{csupp}(\nu)\big)=d$, we find $x^\ast \notin \mathrm{int(csupp}(\nu))$, as we claimed in (\ref{claim1}).

Now Suppose $\dim C=d$, we claim $\overline{C}\subset \mathrm{csupp}(\nu)$. To see this, we only have to show $C\subset \mathrm{csupp}(\nu)$ by the closedness of $\mathrm{csupp}(\nu)$. Suppose not, then we can find $x_0 \in C\setminus \mathrm{csupp}(\nu)$. This implies that there exists $\delta>0$ such that $B(x_0,\delta)\cap \mathrm{csupp}(\nu)\neq \emptyset$. By the assumption that $\dim C=d$, we can find $x_1,\ldots,x_d \in B(x_0,\delta)\cap C$ such that $\{x_0,\ldots,x_d\}$ are in general position. By definition of $C$ we can find $\epsilon_0>0, n_0 \in \N$ such that $f_n(x_j)\geq \epsilon_0$ for all$j=0,1,\ldots,d\textrm{ and }n\geq n_0.$ By convexity, we conclude that
$f_n(x)\geq \epsilon_0,\textrm{ for all } x \in \mathrm{conv}(\{x_0,\ldots,x_d\})\textrm{ and }n\geq n_0.$
This gives
\begin{equation*}
\begin{split}
\nu\big(\mathrm{conv}(\{x_0,\ldots,x_d\})\big)&\geq \limsup_{n \to \infty}\nu_n\big(\mathrm{conv}(\{x_0,\ldots,x_d\})\big)\\
&\geq \epsilon_0 \lambda_d\big(\mathrm{conv}(\{x_0,\ldots,x_d\})\big)>0,
\end{split}
\end{equation*}
a contradiction with $B(x_0,\delta)\cap \mathrm{csupp}(\nu)\neq \emptyset$, thus completing the proof of the claim. To summarize, we have proved
\begin{enumerate}
\item If $\dim\big(\mathrm{csupp}(\nu)\big)=d$, then $\mathrm{csupp}(\nu)\subset \overline{C}$. This in turn implies $\dim C=d$, and hence $\overline{C}\subset \mathrm{csupp}(\nu)$. Now it follows that $\mathrm{csupp}(\nu)=\overline{C}$;
\item If $\dim C=d$, then $\overline{C}\subset \mathrm{csupp}(\nu)$. This in turn implies $\dim \big(\mathrm{csupp}(\nu)\big)=d$, and hence $\mathrm{csupp}(\nu)\subset \overline{C}$. Now it follows that $\mathrm{csupp}(\nu)=\overline{C}$.\qedhere
\end{enumerate}
\end{proof}
\begin{proof}[Proof of Lemma \ref{pwconv}]
The proof is essentially the same as the proof of Proposition 2 \cite{cule2010theoretical} by exploiting convexity at the level of the underlying basic convex function so we shall omit it.
\end{proof}
\begin{proof}[Proof of Lemma \ref{unifbound}]
Set $U_{n,t}=\{x\in \R^d:f_n(x)\geq t\}.$
We first claim that there exists $n_0 \in \N,\epsilon_0\in (0,1)$ such that $\lambda_d(U_{n,\epsilon_0})\geq \epsilon_0$
holds for all $n\geq n_0$. If not, then for all $k \in \N,l \in \N$, there exists $n_{k,l} \in \N$ such that $\lambda_d(U_{n_{k,l},1/l})\leq {1 \over l}$. Note that $\{\liminf_n f_n>0\}=\cup_{k \in \N}\cup_{l \in \N}\cap_{n\geq k}U_{n,1/l}$. Since  $\lambda_d\big(\bigcup_{l \in \N}\bigcap_{n\geq k}U_{n,1/l}\big)=\lim_{l \to \infty}\lambda_d\big(\bigcap_{n\geq k}U_{n,1/l}\big)\leq \lim_{l \to \infty}\lambda_d(U_{n_{k,l},1/l})=0$, we find that $C=\{\liminf_n f_n>0\}$ is a countable union of null set and hence $\lambda_d(C)=0$, a contradiction to the assumption $\dim C=d$. This shows the claim.

Denote $M_n:=\sup_{x \in \R^d} f_n(x), \epsilon_n\in \mathrm{Arg}\max f_n(x).$
Without loss of generality we assume $M_n\geq \frac{\epsilon_0}{(1+\kappa_s)^{1/s}}$ where $\kappa_s=(1/2)^s-1>0$, and we set $\lambda_n:=\frac{\kappa_s M_n^s}{\epsilon_0^s-M_n^s}\in [0,1]$. Now for $x \in U_{n,\epsilon_0}$, by convexity of $f_n^s$ we have
\[
f_n^s\left(\epsilon_n+\lambda_n(x-\epsilon_n)\right)\leq \lambda_n f_n^s(x)+(1-\lambda_n)f_n^s(\epsilon_n) \leq \lambda_n \epsilon_0^s+(1-\lambda_n)M_n^s=(M_n/2)^s.
\]
This implies $f_n(x)\geq M_n/2:=\Omega_n,$ for all $x \in V_{n,\epsilon_0}:=\{\epsilon_n+\lambda_n(x-\epsilon_n):x \in U_{n,\epsilon_0}\}$. Hence $V_{n,\epsilon_0}\subset U_{n,\Omega_n}$ and therefore $\lambda_d(V_{n,\epsilon_0})=\lambda_d(U_{n,\epsilon_0})\lambda_n^d$, thus
\[\lambda_d(U_{n,\Omega_n})\geq \lambda_d(V_{n,\epsilon_0})=\lambda_d(U_{n,\epsilon_0})\lambda_n^d\geq \epsilon_0\lambda_n^d,\]
holds for all $n\geq n_0$. On the other hand,
\[1=\int f_n\geq \Omega_n \lambda_d(U_{n,\Omega_n})\geq \Omega_n \epsilon_0\lambda_n^d,\]
and suppose the contrary that $M_n \to \infty$ as $n \to \infty$, then
\[1\geq\Omega_n \epsilon_0\lambda_n^d=\frac{\epsilon_0\kappa_s^d}{2(\epsilon_0^s-M_n^s)^d}M_n^{1+sd}\geq c M_n^{1+sd}\to \infty,\quad n \to \infty,\]
since $1+sd>0$ by assumption $-1/d<s<0$. Here $c=\frac{\epsilon_0^{1-sd}\kappa_s^d}{2}$. This gives a contradiction and the proof is complete.
\end{proof}

\begin{proof}[Proof of Theorem \ref{limitchar}]
We only have to show $\nu$ is absolutely continuous with respect to $\lambda_d$. To this end, for given $\epsilon>0$, choose $\delta=\epsilon/2M$, where $M:=\sup_{n}\pnorm{f_n}{\infty}<\infty$ by virtue of Lemma \ref{unifbound}. Now for Borel set $A\subset \R^d$ with $\lambda_d(A)\leq \delta$, we can take an open $A'\supset A$ such that $\lambda_d(A')\leq 2\delta$ by the regularity of Lebesgue measure. Then
\[\nu(A)\leq \nu(A')\leq \liminf_{n \to \infty}\nu_n(A')=\liminf_{n \to \infty}\int_{A'}f_n\leq 2\delta M=\epsilon,\]
as desired.
\end{proof}

\begin{proof}[Proof of Lemma \ref{weakconvbound}]
Let $g_n=f_n^s$ and $g=f^s$. Without loss of generality we assume $0 \in \intdom{g}$, and choose $\eta>0$ small enough such that $B_\eta:=\overline{B}(0,\eta)\subset \intdom{g}$. By the Lemma \ref{convlb}, we know there exists $a>0, R>0$ such that $\frac{g(x)-g(0)}{\pnorm{x}{}}\geq a,$
holds for all $\pnorm{x}{}\geq {R \over 2}$. Now we claim that there exists $n_0 \in \N$ such that $\frac{g_n(x)-g_n(0)}{\pnorm{x}{}}\geq {a \over 8},$
holds for all $\pnorm{x}{}\geq R$ and $n\geq n_0$. Note for each $n \in \N$, by convexity of $g_n(\cdot)$, we know that for fixed $x \in \R^d$, the quantity $\frac{g_n(\lambda x)-g_n(0)}{\pnorm{\lambda x}{}}$ is non-decreasing in $\lambda$, so we only have to show the claim for $\pnorm{x}{}=R$ and $n_0\geq n$. Suppose the contrary, then we can find a subsequence $\{g_{n(k)}\}$ and $\pnorm{x_{n(k)}}{}=R$ such that
$\frac{g_{n(k)}(x_{n(k)})-g_{n(k)}(0)}{\pnorm{x_{n(k)}}{}}<{a \over 8}.$
For simplicity of notation we think of $\{g_{n}\},\{x_{n}\}$ as $\{g_{n(k)}\},\{x_{n(k)}\}$. Now define $ A_n:=\mathrm{conv}(\{x_n,B_\eta\}); B_n:=\{y \in \R^d:\pnorm{y-x_n}{}\leq R/2\}; C_n:=A_n \cap B_n.$
By reducing $\eta>0$ if necessary, we may assume $B_\eta\cap B_n = \emptyset$. It is easy to see $C_n$ is convex and $\lambda_d(C_n)=\lambda_0$ is a constant independent of $n \in \N$. By Lemma \ref{pwconv}, we know that $g_n \to_{a.e.} g$ on $B_\eta$, and hence $\sup_{x \in B_\eta} \abs{g_n(x)-g(x)}\to 0(n \to \infty)$ by Theorem 10.8, \cite{rockafellar1997convex}. By further reducing $\eta>0$ if necessary, we may assume
$g_n(y)\leq g(0)+{aR \over 8},$
holds for all $y \in B_\eta$ and $n \in \N$. Now for any $x^\ast \in C_n$, write $x^\ast=\lambda x_n+(1-\lambda)y$, by noting $R/2 \leq \pnorm{x^\ast}{}\leq R$ and convexity of $g_n$, we get
\begin{equation*}
\begin{split}
\frac{g_n(x^\ast)-g_n(0)}{\pnorm{x^\ast}{}} &\leq \frac{\lambda g_n(x_n)+(1-\lambda)g_n(y)-g_n(0)}{\pnorm{x^\ast}{}}\\
&= \lambda\cdot\frac{g_n(x_n)-g_n(0)}{\pnorm{x_n}{}}\cdot\frac{\pnorm{x_n}{}}{\pnorm{x^\ast}{}}+(1-\lambda)\frac{g_n(y)-g_n(0)}{\pnorm{x^\ast}{}}\\
&\leq \lambda\cdot {a \over 8}{R \over {R/2}}+(1-\lambda)\frac{aR/8}{R/2}={a \over 4}.\\
\end{split}
\end{equation*}
This gives rise to
\begin{equation*}
\begin{split}
\liminf_{n \to \infty}\int_{C_n}(f_n-f)&\geq \liminf_{n \to \infty} \lambda_0\big((aR/4+g_n(0))^{1/s}-(aR/2+g(0))^{1/s}\big)\\
&= \lambda_0\big((aR/4+g(0))^{1/s}-(aR/2+g(0))^{1/s}\big)>0,
\end{split}
\end{equation*}
which is a contradiction to Lemma \ref{unifconv}. This establishes our claim. Now by Lemma \ref{pwconv}, we find that the set $\{\liminf_n f_n(\cdot)>0\}$ is full-dimensional, and hence by Lemma \ref{unifbound} we conclude $g_n(\cdot)$ is uniformly bounded away from zero. Also note by Lemma \ref{generalposition} we find $g(\cdot)$ must be bounded away from zero, which gives the desired assertion. 
\end{proof}

Before the proof of Theorem \ref{localconv}, we first state some useful lemmas that give good control of tails with local 
information of the $s$-concave densities; the proof can be found in Section \ref{sec:proofs_tail}.
\begin{lemma}\label{tailsconcave}
	Let $x_0,\ldots,x_d$ be $d+1$ points in $\R^d$ such that its convex hull 
	$\Delta=\mathrm{conv}(\{x_0,\ldots,x_d\})$ is non-void. 
	If $f(y)\leq \min_j\big({1 \over d}\sum_{i\neq j}f^s(x_i)\big)^{1/s}$, then
	\[
	f(y) \leq f_\mathrm{max}\bigg(1-{d \over r}+{d \over r}f_\mathrm{min} C(1+\pnorm{y}{}^2)^{1/2}\bigg)^{-r}.
	\]
	Here the constant $C=\lambda_d(\Delta)(d+1)^{-1/2}\sigma_{\mathrm{max}}(X)^{-1}$ where 	
	$X=
	\begin{pmatrix}
	x_0 &\ldots &x_d\\
	1 &\ldots & 1\\
	\end{pmatrix}
	$
	and $f_\mathrm{min}:=\min_{0\leq j\leq d}f(x_j),f_\mathrm{max}:=\max_{0\leq j\leq d}f(x_j)$.
\end{lemma}

\begin{lemma}\label{shiftballcontrol}
	Let $\nu$ be a probability measure with $s$-concave density $f$. 
	Suppose that $B(0,\delta)\subset \intdom{f}$ for some $\delta>0$. Then for any $y \in \R^d$,
	\[
	\sup_{x \in B(y,\delta_t)}f(x)\leq J_0\left({1 \over t}\left(\left(\frac{\nu(B(ty,\delta_t))}{J_0\lambda_d(B(ty,\delta_t))}\right)^{-1/r}-(1-t)\right)\right)^{-r},
	\]
	where $J_0:=\inf_{v \in B(0,\delta)} f(v)$ and $\delta_t=\delta\frac{1-t}{1+t}$.
\end{lemma}
Now we are in position to prove Theorem \ref{localconv}.
\begin{proof}[Proof of Theorem \ref{localconv}]
	That the sequence $\{ f_n \}_{n \in \N}$ converges uniformly on any compact subset in 
	$\mbox{int} (\mbox{dom} (f))$ follows directly from Lemma 3.2 and Theorem 10.8 \cite{rockafellar1997convex}.  
	Now we show that if $f$ is continuous at $y \in \R^d$ with $f(y) = 0$, then for any $\eta > 0$
	there exists $\delta = \delta (y,\eta)$ such that
	\begin{equation}\label{localconv:step2}
	\limsup_{n \to \infty}\sup_{x \in B(y,\delta(y,\eta))}f_n(x)\leq \eta.
	\end{equation}
	Assume without loss of generality that $B(0,\delta_0)\subset \intdom{f}$ 
	for some $\delta_0>0$. Let $J_0:=\inf_{x \in B(0,\delta_0)}f(x)$. 
	Then uniform convergence of $\{f_n\}$ to $f$ over $B(0,\delta_0)$ entails that
	\[
	\liminf_{n \to \infty}\inf_{x \in B(0,\delta_0)}f_n(x)\geq J_0.
	\]
	Hence with $\delta_t=\delta_0\frac{1-t}{1+t}$, it follows from Lemma \ref{shiftballcontrol} that
	\begin{equation*}
	\begin{split}
	\limsup_{n \to \infty}\sup_{x \in B(y,\delta_t)}f_n(x)
	&\leq J_0\bigg({1 \over t}\bigg(\left(\frac{\nu(B(ty,\delta_t))}{J_0\lambda_d(B(ty,\delta_t))}\right)^{-1/r}-(1-t)\bigg)\bigg)^{-r}\\
	&\leq J_0\bigg(\frac{J_0^{1/r}\big(\sup_{x \in B(ty,\delta_t)}f(x)\big)^{-1/r}-(1-t)}{t}\bigg)^{-r} \to 0\\
	\end{split}
	\end{equation*}
	as $t \nearrow 1$. This completes the proof for (\ref{localconv:step2}). So far we have shown that
	\[\lim_{n \to \infty}\sup_{x \in S\cap B(0,\rho)}\abs{f_n(x)-f(x)}=0
	\]
	holds for every $\rho\geq 0$, where $S$ is the closed set contained in the 
	continuity points of $f$. Our goal is to let $\rho \to \infty$ and conclude. 
	Let $\Delta=\mathrm{conv}(\{x_0,\ldots,x_d\})$ be a non-void simplex with 
	$x_0,\ldots,x_d \in \intdom{f}$. Note first by a closer look at the proof of Lemma \ref{weakconvbound},
	$f_n(x)\vee f(x)\leq \big((a\pnorm{x}{}-b)\big)_+^{1/s}$
	holds for all $x \in \R^d$ with some $a,b>0$. 
	Let $\rho_0:=\inf\{\rho\geq 0 : \big(a\rho-b)^{1/s}\leq f_\mathrm{min}/2\}$ 
	where $f_\mathrm{min}:=\min_{0\leq j\leq d}f(x_i)>0$. Then
	\begin{eqnarray*}
		\lefteqn{\{x \in \R^d:\pnorm{x}{}\geq \rho_0\}
			\subset \bigcap_{n\geq 1}\{f_n \leq f_\mathrm{min}/2\}\bigcap \{f\leq f_\mathrm{min}/2\} }\\
		&\subset & \bigcap_{n\geq n_0} \{f_n \leq (f_n)_\mathrm{min}\}\bigcap \{f\leq f_\mathrm{min}\}\\
		&\subset & \bigcap_{n\geq n_0} \{f_n \leq \min_j\big({1 \over d}
		\sum_{i\neq j}f_n^s(x_i)\big)^{1/s}\}\bigcap \{f\leq \min_j\big({1 \over d}\sum_{i\neq j}f^s(x_i)\big)^{1/s}\},
	\end{eqnarray*}
	where $n_0 \in \N$ is a large constant. The second inclusion follows from the fact that
	$\lim_{n \to \infty}f_n(x_i)=f(x_i)$ holds for $i=0,\ldots,d$. By Lemma \ref{tailsconcave} we conclude that
	\begin{equation*}
	\begin{split}
	&\limsup_{n \to \infty}\sup_{x:\pnorm{x}{}\geq \rho\vee\rho_0}\big(1+\pnorm{x}{})^\kappa \big(f_n(x)\vee f(x)\big)\\
	\leq&\sup_{x:\pnorm{x}{}\geq \rho\vee\rho_0}f_\mathrm{max}\big(1+\pnorm{x}{})^\kappa\bigg(1-{d \over r}
	+{d \over r}f_\mathrm{min}C\big(1+\pnorm{x}{}^2\big)^{1/2}\bigg)^{-r}\to 0,
	\end{split}
	\end{equation*}
	as $\rho \to \infty$. This completes the proof.
\end{proof}

\begin{proof}[Proof of Theorem \ref{thm:conv_derivatives}]
	Since $\nabla_\xi f_n(x)=-rg_n(x)^{1/s-1}\nabla_\xi g_n(x)$, 
	\begin{equation*}
	\begin{split}
	&\qquad\abs{\nabla_\xi f_n(x)-\nabla_\xi f(x)}\\
	&=r\abs{g_n(x)^{1/s}\nabla_\xi g_n(x)-g(x)^{1/s}\nabla_\xi g(x)}\\
	&\leq r\bigg(f_n(x)\abs{\nabla_\xi g_n(x)-\nabla_\xi g(x)}+\abs{f_n(x)-f(x)}\abs{\nabla_\xi g(x)}\bigg)\\
	&\leq 2r\sup_{x \in T}\abs{f(x)}\abs{\nabla_\xi g_n(x)-\nabla_\xi g(x)}
	+r\sup_{x \in T}\abs{f_n(x)-f(x)}\sup_{x \in T}\pnorm{\nabla g(x)}{2}
	\end{split}
	\end{equation*}
	holds for $n$ large enough by Theorem \ref{localconv}. By Theorem 23.4 in 
	\cite{rockafellar1997convex}, $\nabla_\xi g_n(x)=\tau_x^T \xi$ for some 
	$\tau_x \in \partial g_n(x)$ since $\partial g_n(x)$ is a closed set. 
	Thus the first term above is further bounded by 
	\[2r\sup_{x \in T}\abs{f(x)}\sup_{x \in T,\tau \in \partial g_n(x)}\pnorm{\tau-\nabla g(x)}{2},\]
	which vanishes as $n \to \infty$ in view of Lemma 3.10 in \cite{MR2850215}. 
	Note that $\nabla g(\cdot)$ is continuous on $T$ by Corollary 25.5.1 in 
	\cite{rockafellar1997convex}, and hence $\sup_{x \in T}\pnorm{\nabla g(x)}{2}<\infty$. 
	Now it is easy to see that the second term also vanishes as $n \to \infty$ by virtue of Theorem \ref{localconv}.
\end{proof}

\subsection{Proofs for Section 4}\label{appendix:supp_proof_4}

Before we prove Theorem \ref{thm:limittheory}, we will need the following tightness result.
\begin{theorem}\label{uniquetight}
	We have the following conclusions.
	\begin{enumerate}
		\item For fixed $K>0$, the modified local process 
		$\mathbb{Y}^{\mathrm{locmod}}_n(\cdot)$ converges weakly to a drifted integrated Gaussian process on $C[-K,K]$:
		\[
		\mathbb{Y}^{\mathrm{locmod}}_n(t) \to_d \frac{1}{\sqrt{f_0(x_0)}}\int_0^t W(s)\ \d{s}-\frac{rg_0^{(k)}(x_0)}{g_0(x_0)(k+2)!}t^{k+2},
		\]
		where $W(\cdot)$ is the standard two-sided Brownian motion starting from $0$ on $\R$.
		\item The localized processes satisfy
		\[
		\mathbb{Y}^{\mathrm{locmod}}_n(t)-\mathbb{H}^{\mathrm{locmod}}_n(t)\geq 0,
		\]
		with equality attained for all $t$ such that $x_0+tn^{-1/(2k+1)} \in \mathcal{S}(\hat{g}_n)$.
		\item The sequences $\{\hat{A}_n\}$ and $\{\hat{B}_n\}$ are tight.
	\end{enumerate}
\end{theorem}
The above theorem includes everything necessary in order to apply the `invelope' argument roughly indicated in Section \ref{subsection:limit_distribution_theory}. For a proof of this technical result, 
we refer the reader to Section \ref{appendix:proof_tightness}. 
Here we will provide proofs for our main results.
\begin{proof}[Proof of Theorem \ref{thm:limittheory}]
	By the same tightness and uniqueness argument adopted in \cite{groeneboom2001estimation}, 
	\cite{balabdaoui2007estimation}, and \cite{balabdaoui2009limit}, we only have 
	to find the rescaling constants. To this end we denote $\mathbb{H}(\cdot)$,$\mathbb{Y}(\cdot)$ 
	the corresponding limit of $\mathbb{H}_n^{\mathrm{locmod}}(\cdot)$ and 
	$\mathbb{Y}_n^{\mathrm{locmod}}(\cdot)$ in the uniform topology on the 
	space $C[-K,K]$, and let $\mathbb{Y}(t)=\gamma_1Y_k(\gamma_2 t),$
	where by Theorem \ref{uniquetight}, we know that
	\[
	\mathbb{Y}(t)=\frac{1}{\sqrt{f_0(x_0)}}\int_0^t W(s)\ \d{s}-\frac{rg_0^{(k)}(x_0)}{g_0(x_0)(k+2)!}t^{k+2}.
	\]
	Let $a:=\big(f_0(x_0)\big)^{-1/2}$ and $b:=\frac{rg_0^{(k)}(x_0)}{g_0(x_0)(k+2)!}$, 
	then by rescaling property of Brownian motion, we find that 
	$ \gamma_1\gamma_2^{3/2}=a, \gamma_1\gamma_2^{k+2}=b $.
	Solving for $\gamma_1,\gamma_2$ yields
	\begin{equation}\label{ab}
	\gamma_1=a^{\frac{2k+4}{2k+1}}b^{-\frac{3}{2k+1}},\quad\gamma_2=a^{-\frac{2}{2k+1}}b^{\frac{2}{2k+1}}.
	\end{equation}
	On the other hand, by (\ref{rel1}), let $n \to \infty$, we find that
	\begin{equation}\label{rel2}
	\begin{split}
	\begin{pmatrix}
	n^{\frac{k}{2k+1}}\big(\hat{g}_n(x_0+s_nt)-g_0(x_0)-s_ntg_0'(x_0)\big)\\
	n^{\frac{k-1}{2k+1}}\big(\hat{g}_n'(x_0+s_nt)-g_0'(x_0)\big)
	\end{pmatrix}
	\to_d
	\begin{pmatrix}
	\frac{g_0(x_0)}{-r}\frac{\d{}^2}{\d{t^2}}\mathbb{H}(t)\\
	\frac{g_0(x_0)}{-r}\frac{\d{}^3}{\d{t^3}}\mathbb{H}(t)
	\end{pmatrix}
	\end{split}
	\end{equation}
	It is easy to see that 
	$\frac{\d{}^2}{\d{t^2}}\mathbb{H}(t)=\gamma_1\gamma_2^2\frac{\d{}^2}{\d{t^2}}H_k(\gamma_2t)$ 
	nd $\frac{\d{}^3}{\d{t^3}}\mathbb{H}(t)=\gamma_1\gamma_2^3\frac{\d{}^3}{\d{t^3}}H_k(\gamma_2t)$. 
	Now by substitution in (\ref{ab}) we get the conclusion by direct calculation and the delta method.
\end{proof}

\begin{proof}[Proof of Theorem \ref{thm:estimation_mode}]
	The proof is essentially the same as that of Theorem 3.6 \cite{balabdaoui2009limit}.
\end{proof}

\begin{lemma}\label{totalmass}
	Assume (A1)-(A4). Then
	\[\int_{-\infty}^\infty \tilde{f}_\epsilon(x)\ \d{x} = 1+\pi_k\frac{r g^{(k)}(m_0)}{g(m_0)^{r+1}}\epsilon^{k+1}+o(\epsilon^{k+1}),\]
	where
	\[
	\pi_k = \frac{1}{(k+1)!}\left[3^{k-1}(2k^2-4k+3)+2k^2-1\right].
	\]
\end{lemma}
\begin{proof}[Proof of Lemma \ref{totalmass}]
	This is straightforward calculation by Taylor expansion. Note that
	\begin{equation*}
	\begin{split}
	\int_{-\infty}^\infty \tilde{g}_\epsilon^{-r}(x)\ \d{x}&=\int_{-\infty}^\infty (\tilde{g}_\epsilon^{-r}(x)-g^{-r}(x))\ \d{x}+1\\
	&=\int_{m_0-c_\epsilon \epsilon}^{m_0-\epsilon}\bigg(\tilde{g}_\epsilon^{-r}(x)-g^{-r}(x)\bigg)\ \d{x}\\
	&\qquad+\int_{m_0-\epsilon}^{m_0+\epsilon}\bigg(\tilde{g}_\epsilon^{-r}(x)-g^{-r}(x)\bigg)\ \d{x}+1\\
	&:=I+II+1.
	\end{split}
	\end{equation*}
	For $y>x$, we have $x^{-r}-y^{-r}=\sum_{n\geq 1}^\infty\binom{-r}{n}(-1)^n (y-x)^ny^{-r-n}$. Now for the first term above, we continue our calculation of its leading term by noting
	\begin{equation}\label{eqn:diff_pertub_left}
	\begin{split}
	&\qquad g(x)-\tilde{g}_\epsilon(x)\\
	&= g(x)-g(m_0-c_\epsilon \epsilon)-(x-m_0+c_\epsilon \epsilon)g'(m_0-c_\epsilon \epsilon)\\
	&= g(m_0)+\frac{g^{(k)}(m_0)}{k!}(x-m_0)^k-\bigg[g(m_0)+\frac{g^{(k)}(m_0)}{k!}(-c_\epsilon \epsilon)^k\bigg]\\
	&\qquad\qquad -(x-m_0+c_\epsilon \epsilon)\frac{g^{(k)}(m_0)}{(k-1)!}(-c_\epsilon \epsilon)^{k-1}+\textrm{ higher order terms}\\
	&=\frac{g^{(k)}(m_0)}{k!}\bigg[(x-m_0)^k-c_\epsilon^k \epsilon^k + kc_\epsilon^{k-1}\epsilon^{k-1}(x-m_0+c_\epsilon \epsilon)\bigg]+\textrm{ higher order terms}.
	\end{split}
	\end{equation}
	Here we used the fact $k$ is an even number, as shown in Lemma \ref{lem:k_even}. Thus we have
	\begin{equation*}
	\begin{split}
	&\quad\textrm{leading term of I}\\
	&=\int_{m_0-c_\epsilon \epsilon}^{m_0-\epsilon}r\bigg(g(x)-g(m_0-c_\epsilon \epsilon)-(x-m_0+c_\epsilon \epsilon)g'(m_0-c_\epsilon \epsilon)\bigg)g(x)^{-r-1}\ \d{x}\\
	& = \frac{rg^{(k)}(m_0)}{k!g(m_0)^{r+1}}\int_{m_0-c_\epsilon \epsilon}^{m_0-\epsilon}\bigg[(x-m_0)^k-c_\epsilon^k \epsilon^k + kc_\epsilon^{k-1}\epsilon^{k-1}(x-m_0+c_\epsilon \epsilon)\bigg]\ \d{x}+o(\epsilon^{k+1})\\
	& = \alpha_k \frac{rg^{(k)}(m_0)}{g(m_0)^{r+1}} \epsilon^{k+1} +o(\epsilon^{k+1})
	\end{split}
	\end{equation*}
	Here
	\[\alpha_k = \frac{1}{(k+1)!}\left[3^{k-1}(2k^2-4k+3)-1\right].
	\]
	For the second term, 
	\begin{equation}\label{eqn:diff_pertub_right}
	\begin{split}
	&\qquad g(x)-\tilde{g}_\epsilon(x)\\
	&= g(x)-g(m_0+\epsilon)-(x-m_0-\epsilon)g'(m_0+\epsilon)\\
	&=\frac{g^{(k)}(m_0)}{k!}\bigg[(x-m_0)^k-\epsilon^k-k\epsilon^{k-1}(x-m_0-\epsilon)\bigg]+\textrm{ higher order terms}.
	\end{split}
	\end{equation}
	Now similar calculations yield that the second term $=\beta_k\frac{r g^{(k)}(m_0)}{g(m_0)^{r+1}}\epsilon^{k+1}+o(\epsilon^{k+1})$ with
	\[
	\beta_k = \frac{2k^2}{(k+1)!}.
	\]
	This gives the conclusion.
\end{proof}

\begin{proof}[Proof of Lemma \ref{lem:osc_gap}]
	By definition of the Hellinger metric and Lemma \ref{totalmass}, we have
	\begin{equation*}
	\begin{split}
	2h^2(f_\epsilon,f) &=\int_{-\infty}^\infty \big(\sqrt{f_\epsilon(x)}-\sqrt{f(x)}\big)^2\ \d{x}\\
	&=\int_{-\infty}^\infty \bigg(\tilde{g}^{-r/2}_\epsilon(x)\left(1-{\pi_k \over 2}\frac{r g^{(k)}(m_0)}{g(m_0)^{r+1}}\epsilon^{k+1}+o(\epsilon^{k+1})\right)-g^{-r/2}(x)\bigg)^2\ \d{x}\\
	&\equiv\int_{-\infty}^\infty \left(\tilde{g}^{-r/2}_\epsilon(x)(1+\eta_k(\epsilon))-g^{-r/2}(x)\right)^2 \ \d{x}
	\end{split}
	\end{equation*}
	since 
	\begin{equation*}
	\begin{split}
	f_\epsilon(x)&=\tilde{g}_\epsilon^{-r}(x)\bigg(1+\pi_k\frac{r g^{(k)}(m_0)}{g(m_0)^{r+1}}\epsilon^{k+1}+o(\epsilon^{k+1})\bigg)^{-1}\\
	&=\tilde{g}_\epsilon^{-r}(x)\bigg(1-\pi_k\frac{r g^{(k)}(m_0)}{g(m_0)^{r+1}}\epsilon^{k+1}+o(\epsilon^{k+1})\bigg).
	\end{split}
	\end{equation*}
	Here $\eta_k(\epsilon)=O(\epsilon^{k+1})$. Splitting two terms apart in the above integral we get
	\begin{equation*}
	\begin{split}
	2h^2(f_\epsilon,f)&=\int_{-\infty}^\infty\bigg(\tilde{g}^{-r/2}_\epsilon(x)-g^{-r/2}(x)+\eta_k(\epsilon)\tilde{g}^{-r/2}_\epsilon(x)\bigg)^2\ \d{x}\\
	&=\int_{-\infty}^\infty \big(\tilde{g}^{-r/2}_\epsilon(x)-g^{-r/2}(x)\big)^2\ \d{x}+\big(\eta_k(\epsilon)\big)^2 \int_{-\infty}^\infty \tilde{g}^{-r}_\epsilon(x)\ \d{x}\\
	&\quad\quad+2\eta_k(\epsilon)\int_{-\infty}^\infty \tilde{g}^{-r/2}_\epsilon(x)\big(\tilde{g}^{-r/2}_\epsilon(x)-g^{-r/2}(x)\big)\ \d{x}\\
	&=I+II+III.
	\end{split}
	\end{equation*}
	Now for the first term,
	\begin{equation*}
	\begin{split}
	I &= \int_{m_0-c_\epsilon \epsilon}^{m_0+\epsilon}{r^2 \over 4}\big[g(x)-\tilde{g}_\epsilon(x)\big]^2g(x)^{-r-2}\ \d{x}+\textrm{ higher order terms}\\
	& = \frac{r^2}{4g(m_0)^{r+2}}\int_{m_0-c_\epsilon \epsilon}^{m_0+\epsilon}\big[g(x)-\tilde{g}_\epsilon(x)\big]^2\ \d{x}+\textrm{ higher order terms}\\
	& =  \frac{r^2}{4g(m_0)^{r+2}}\bigg(\int_{m_0-c_\epsilon \epsilon}^{m_0-\epsilon}+\int_{m_0-\epsilon}^{m_0+\epsilon}\bigg)\big[g(x)-\tilde{g}_\epsilon(x)\big]^2\ \d{x}+\textrm{ higher order terms}\\
	&= I_1+I_2+\textrm{ higher order terms}.
	\end{split}
	\end{equation*}
	By (\ref{eqn:diff_pertub_left}) and (\ref{eqn:diff_pertub_right}) we see that for $i=1,2$,
	\begin{equation*}
	\begin{split}
	I_i & =  \frac{r^2}{4g(m_0)^{r+2}}\int_{\mathcal{I}_i}\big[g(x)-\tilde{g}_\epsilon(x)\big]^2\ \d{x}\\
	&=\zeta_k^{(i)}\frac{r^2f(m_0)g^{(k)}(m_0)^2}{g(m_0)^2}\epsilon^{2k+1}+o(\epsilon^{2k+1}).
	\end{split}
	\end{equation*}
	Here $\mathcal{I}_1=[m_0-c_\epsilon\epsilon,m_0-\epsilon]$, $\mathcal{I}_2=[m_0-\epsilon,m_0+\epsilon]$, and
	\begin{equation*}
	\begin{split}
	\zeta_k^{(1)}&=\frac{1}{108(k!)^2(k+1)(k+2)(2k+1)}\bigg[-4\cdot 3^{k+2}(2k+1)(3^{k+2}+k^2+k-3)\\
	&\quad +(k+1)(k+2)\bigg(27(3^{2k+1}-1)+2\cdot 3^{2k}(2k+1)(2k(2k-9)+27)\bigg)\bigg].\\
	\zeta_k^{(2)}&=\frac{2k^2(2k^2 + 1)}{3(k!)^2(k+1)(2k+1)}.
	\end{split}
	\end{equation*}
	On the other hand, $II=O(\epsilon^{(2k+2)})=o(\epsilon^{2k+1})$ and 
	$\abs{III}\leq O(\epsilon^{k+1}\cdot \epsilon^{(2k+1)/2}\cdot\epsilon^{(2k+2)/2})=o(\epsilon^{2k+1})$ 
	by Cauchy-Schwarz. This completes the proof.
\end{proof}

\section{Appendix}\label{sec:appendix}

\subsection{Proofs of Lemmas \ref{tailsconcave} and \ref{shiftballcontrol}}\label{sec:proofs_tail}

\begin{lemma}\label{volumeratiobound}
	Let $\nu$ be a probability measure with $s$-concave density $f$, and $x_0,\ldots,x_d \in \R^d$ be $d+1$ points such that $\Delta:=\mathrm{conv}(\{x_0,\ldots,x_d\})$ is non-void.  
	If $f(x_0)\leq \big({1 \over d}\sum_{i=1}^d f^s(x_i)\big)^{1/s}$, then
	\[f(x_0)\leq \bar{g}^{-r}\bigg(1-{d \over r}+{d \over r}\frac{\lambda_d(\Delta)\bar{g}^{-r}}{\nu(\Delta)}\bigg)^{-r},\]
	where $\bar{g}:={1 \over d}\sum_{j=1}^d f^s(x_j)$.
\end{lemma}

\begin{proof}[Proof of Lemma \ref{volumeratiobound}]
	For any point $x \in \Delta$, we can find some $u=(u_1,\ldots,u_d)\in \Delta_d=\{u:\sum_{i=1}^d u_i\leq 1\}$ such that $x(u)=\sum_{i=0}^d u_ix_i$. Here $u_0:=1-\sum_{i=1}^d u_i\geq 0$. We use the following representation
	of integration on the unit simplex $\Delta_d$: For any measurable function $h:\Delta_d\to [0,\infty)$, we have
	$\int_{\Delta_d}h(u)\ \d{u}={1 \over d!}\mathbb{E}h(B_1,\ldots,B_d),$
	where $B_i=E_i/\sum_{j=0}^d E_j$ with independent, standard exponentially distributed random variables $E_0,\ldots,E_d$. 
	\begin{equation*}
	\begin{split}
	\frac{\nu(\Delta)}{\lambda_d(\Delta)}&=\frac{1}{\lambda_d(\Delta_d)}\int_{\Delta_d} g\big(x(u)\big)^{-r}\ \d{u}=\mathbb{E}g\bigg(\sum_{j=0}^d B_j x_j\bigg)^{-r}\\
	&\geq \mathbb{E}\bigg(\sum_{j=0}^d B_j g(x_j)\bigg)^{-r}=\mathbb{E}\bigg(B_0g_0+(1-B_0)\sum_{i=1}^d \tilde{B}_i g(x_i)\bigg)^{-r},
	\end{split}
	\end{equation*}
	where $\tilde{B}_i:=E_i/\sum_{j=1}^d E_j$ for $1\leq i \leq d$. Following \cite{cule2008auxiliary}, it is known that $B_0$ and $\{\tilde{B}_i\}_{i=1}^d$ are independent, and $\mathbb{E}[\tilde{B}_i]=1/d$. Hence it follows from Jensen's inequality that
	\begin{equation*}
	\begin{split}
	\frac{\nu(\Delta)}{\lambda_d(\Delta)}&\geq \mathbb{E}\left[\mathbb{E}\bigg(B_0g_0+(1-B_0)\sum_{i=1}^d \tilde{B}_i g(x_i)\bigg)^{-r}\bigg\vert B_0\right]\\
	&\geq \mathbb{E}\bigg(B_0 g_0+(1-B_0){1 \over d}\sum_{i=1}^d g(x_i)\bigg)^{-r}\\
	&=\mathbb{E}\left(B_0 g_0+(1-B_0)\bar{g}\right)^{-r}\\
	&=\int_0^1 d(1-t)^{d-1} \big(tg_0+(1-t)\bar{g}\big)^{-r}\ \d{t}\\
	&=\bar{g}^{-r}\int_0^1 d(1-t)^{d-1} \left(1-st\left((-1/s)\left(\frac{g_0}{\bar{g}}-1\right)\right)\right)\ \d{t}\\
	&=\bar{g}^{-r}J_{d,s}\bigg(-{1 \over s}\left(\frac{g_0}{\bar{g}}-1\right)\bigg),
	\end{split}
	\end{equation*}
	where
	\[J_{d,s}(y)=\int_0^1 d(1-t)^{d-1}(1-syt)^{1/s}\ \d{t}.\]
	We claim that
	\[J_{d,s}(y)\geq \int_0^1 d(1-t)^{d-1}(1-t)^y \d{t}=\frac{d}{d+y},\]
	holds for $s<0,y>0$. To see this, we write $(1-syt)^{1/s}=(1+yt/r)^{-(r/y)y}.$
	Then we only have to show $(1+yt/r)^{-r/y}\geq (1-t)$ for $0\leq t\leq 1$, or equivalently $(1+bt)\leq (1-t)^{-b}$ where we let $b=y/r$. Let $g(t):=(1-t)^{-b}-(1+bt)$. It is easy to verify that $g(0)=0$, $g'(t)=b(1-t)^{-b-1}-b$ with $g'(0)=0$, and $g''(t)=b(b+1)(1-t)^{-b-2}\geq 0$. Integrating $g''$ twice yields $g(t)\geq 0$, and hence we have verified the claim. Now we proceed with the calculation
	\[
	\frac{\nu(\Delta)}{\lambda_d(\Delta)}\geq \bar{g}^{-r}J_{d,s}\bigg(-{1 \over s}\left(\frac{g_0}{\bar{g}}-1\right)\bigg)\geq \bar{g}^{-r}\frac{d}{d-{1 \over s}\big(\frac{g_0}{\bar{g}}-1\big)}.\\
	\]
	Solving for $g_0$ and replacing $-1/s=r$ proves the desired inequality.
\end{proof}

\begin{proof}[Proof of Lemma \ref{tailsconcave}]
	For fixed $j \in \{0,\ldots,d\}$, note $\abs{\det(x_i-x_j):i\neq j}=\abs{\det X}$
	where
	$X=
	\begin{pmatrix}
	x_0 &\ldots &x_d\\
	1 &\ldots & 1\\
	\end{pmatrix}
	$
	. Also for each $ y\in \R^d$, since $\Delta=\mathrm{conv}(\{x_0,\ldots,x_d\})$ is non-void, $y$ must be in the affine hull of $\Delta$ and hence we can write $y=\sum_{i=0}^d\lambda_i x_i$ with $\sum_{i=0}^d \lambda_i=1$ (not necessary non-negative), i.e. $\lambda=X^{-1}\binom{y}{1}$. Let $\Delta_j(y):=\mathrm{conv}(\{x_i:i\neq j\}\cup\{y\})$. Then
	\begin{equation*}
	\begin{split}
	\lambda_d(\Delta_j(y))&={1 \over d!}\abs{\det
		\begin{pmatrix}
		x_0 &\ldots & x_{j-1} & y & x_{j+1} &\ldots & x_d\\
		1   &\ldots & 1       & 1 & 1       &\ldots & 1\\
		\end{pmatrix}
	}\\
	&={1 \over d!}\abs{\lambda_j}\abs{\det X}=\abs{\lambda_j}\lambda_d(\Delta).
	\end{split}
	\end{equation*}
	Hence,
	\begin{equation*}
	\begin{split}
	\max_{0\leq j \leq d} \lambda_d(\Delta_j(y))&\geq \lambda_d(\Delta)\max_j \abs{\lambda_j}= \lambda_d(\Delta)\pnorm{X^{-1}\binom{y}{1}}{\infty}\\
	&\geq \lambda_d(\Delta)(d+1)^{-1/2}\pnorm{X^{-1}\binom{y}{1}}{}\\
	&\geq \lambda_d(\Delta)(d+1)^{-1/2}\sigma_{\mathrm{max}}(X)^{-1}(1+\pnorm{y}{}^2)^{1/2}=C(1+\pnorm{y}{}^2)^{1/2}.
	\end{split}
	\end{equation*}
	
	Now the conclusion follows from Lemma \ref{volumeratiobound} by noting
	\begin{equation*}
	\begin{split}
	f(y)&\leq \bar{g}_j^{-r}\bigg(1-{d \over r}+{d \over r}\frac{\lambda_d(\Delta_j(y))\bar{g}_j^{-r}}{\nu(\Delta_j(y))}\bigg)^{-r}\leq f_\mathrm{max}\bigg(1-{d \over r}+{d \over r}f_\mathrm{min} C(1+\pnorm{y}{}^2)^{1/2}\bigg)^{-r},
	\end{split}
	\end{equation*}
	since $\bar{g}_j^{-r}=\big({1 \over d}\sum_{i \neq j}f^s(x_i)\big)^{1/s}$ and hence $f_\mathrm{min}\leq \bar{g}_j^{-r}\leq f_\mathrm{max}$, and the index $j$ is chosen such that $\lambda_d(\Delta_j(y))$ is maximized.
\end{proof}

\begin{proof}[Proof of Lemma \ref{shiftballcontrol}]
	The key point that for any point $x \in B(y,\delta_t)$
	\[B(ty,\delta_t)\subset (1-t)B(0,\delta)+tx\]
	can be shown in the same way as in the proof of Lemma 4.2 \cite{schuhmacher2011multivariate}. Namely, pick any $w \in B(ty,\delta_t)$, let $v:=(1-t)^{-1}(w-tx)$, then since
	\[\pnorm{v}{}=(1-t)^{-1}\pnorm{w-tx}{}=(1-t)^{-1}\pnorm{w-ty+t(y-x)}{}\leq (1-t)^{-1}(\delta_t+t\delta_t)=\delta,\]
	and hence $v \in B(0,\delta)$. This implies that $w=(1-t)v+tx \in (1-t)B(0,\delta)+tx$, as desired. By $s$-concavity of $f$, we have
	\begin{equation*}
	\begin{split}
	f(w) &\geq \big((1-t)f(v)^s+t f(x)^s\big)^{1/s}\\
	&\geq \big((1-t)J_0^s+tf(x)^s\big)^{1/s}\\
	&=J_0\left(1-t+t\left(\frac{f(x)}{J_0}\right)^s\right)^{1/s}.
	\end{split}
	\end{equation*}
	Averaging over $w \in B(ty,\delta_t)$ yields
	\[\frac{\nu(B(ty,\delta_t))}{\lambda_d(B(ty,\delta_t))}\geq J_0\left(1-t+t\left(\frac{f(x)}{J_0}\right)^s\right)^{1/s}.\]
	Solving for $f(x)$ completes the proof.
\end{proof}

\subsection{Proof of Theorem \ref{uniquetight}}\label{appendix:proof_tightness}

We first observe that
\begin{lemma}\label{lem:k_even}
	$k$ is an even integer and $g_0^{(k)}(x_0)>0$.
\end{lemma}
\begin{proof}[Proof of Lemma \ref{lem:k_even}]
By Taylor expansion of $g_0''$ around $x_0$, we find that locally for $x\approx x_0$,
\[g_0''(x)=\frac{g_0^{(k)}(x_0)}{(k-2)!}(x-x_0)^{k-2}+o\big((x-x_0)^{k-2}\big).\]
Also note $g''_0(x)\geq 0$ by convexity and local smoothness assumed in (A3). This gives that $k-2$ is even and $g_0^{(k)}(x_0)>0$.
\end{proof}

For further technical discussions, we denote throughout this subsection that for fixed $k$, $r_n:=n^{\frac{k+2}{2k+1}}; s_n:=n^{-\frac{1}{2k+1}}; x_n(t):=x_0+s_n t; \bm{l}_{n,x_0}:=[x_0, x_n(t)].$
Let $\tau_n^+:=\inf\{t \in \mathcal{S}_n(\hat{g}_n):t>x_0\}$, and $\tau_n^-:=\sup\{t \in \mathcal{S}_n(\hat{g}_n): t<x_0\}$. The key step in establishing the limit theory, is to establish a stochastic bound for the gap $\tau_n^+-\tau_n^-$ as follows.
\begin{theorem}\label{gapthm}
Assume (A1)-(A4) hold. Then
\[	\tau_n^+-\tau_n^-=O_p(s_n).\]
\end{theorem}
\begin{proof}
Define $\Delta_0(x):=(\tau_n^- - x){\bf 1}_{[\tau_n^-,\bar{\tau}]}(x)+(x-\tau_n^+){\bf 1}_{[\bar{\tau},\tau_n^+]}(x)$, and
$\Delta_1:= \Delta_0+\frac{\tau_n^+-\tau_n^-}{4}{\bf 1}_{[\tau_n^-,\tau_n^+]}$, where $\bar{\tau}=:\frac{\tau_n^-+\tau_n^+}{2}$. 
Thus we find that
\begin{equation*}
\begin{split}
\int \Delta_1\ \d{(\mathbb{F}_n-F_0)}&=\int \Delta_1\ \d{(\mathbb{F}_n-\hat{F}_n)}+\int \Delta_1\ \d{(\hat{F}_n-F_0)}\\
&\geq -\frac{\tau_n^+ -\tau_n^-}{4}\abs{\int_{\tau_n^-}^{\tau_n^+} \d{(\mathbb{F}_n-\hat{F}_n)}}+\int \Delta_1(\hat{f}_n-f_0)\ \d{\lambda}\\
&\geq -\frac{\tau_n^+ -\tau_n^-}{2n}+\int \Delta_1(\hat{f}_n-f_0)\ \d{\lambda},
\end{split}
\end{equation*}
where the last line follows from Corollary \ref{distchar}. Now let $R_{1n}:=\int \Delta_1(\hat{f}_n-f_0)\ \d{\lambda},R_{2n}:=\int \Delta_1\ \d{(\mathbb{F}_n-F_0)}.$
The conclusion follows directly from the following lemma.
\end{proof}
\begin{lemma}\label{gaperror}
	Suppose (A1)-(A4) hold. Then $R_{1n} =O_p(\tau_n^+-\tau_n^-)^{k+2}$ and $R_{2n} =O_p(r_n^{-1}).$
\end{lemma}

\begin{proof}[Proof of Lemma \ref{gaperror}]
	Define $p_n:=\hat{g}_n/g_0$ on $[\tau_n^+,\tau_n^-]$. It is easy to see that $\tau_n^+-\tau_n^-=o_p(1)$, so with large probability, for all $n \in \N$ large enough, $\inf_{x\in[\tau_n^+,\tau_n^-]}f_0(x)>0$ by (A2).
	\begin{equation*}
	\begin{split}
	R_{1n}& = \int_{\tau_n^-}^{\tau_n^+} \Delta_1(x)\big(\hat{f}_n(x)-f_0(x)\big)\ \d{x}= \int_{\tau_n^-}^{\tau_n^+} \Delta_1(x)f_0(x)\bigg(\frac{\hat{f}_n(x)}{f_0(x)}-1\bigg)\ \d{x}\\
	& = \int_{\tau_n^-}^{\tau_n^+} \Delta_1(x)f_0(x)\bigg(\sum_{j=1}^{k-1} \binom{-r}{j}(p_n(x)-1)^j+\binom{-r}{k}\theta_{x,n}^{-r-k}(p_n(x)-1)^k\bigg)\ \d{x},
	\end{split}
	\end{equation*}
	where $\theta_{x,n}\in[1\wedge\frac{\hat{g}_n(x)}{g_0(x)}, 1\vee\frac{\hat{g}_n(x)}{g_0(x)}]$. Now define
	\begin{equation*}
	\begin{split}
	S_{nj}&=\int_{\tau_n^-}^{\tau_n^+} \Delta_1(x)f_0(x)\binom{-r}{j}(p_n(x)-1)^j\ \d{x}, 1\leq j \leq k-1,\\
	S_{nk}&=\int_{\tau_n^-}^{\tau_n^+} \Delta_1(x)f_0(x)\binom{-r}{k}\theta_{x,n}^{-r-k}(p_n(x)-1)^k\ \d{x}.\\
	\end{split}
	\end{equation*}
	Expand $f_0$ around $\bar{\tau}$, then we have
	\begin{equation*}
	\begin{split}
	S_{nj}&= \sum_{l=0}^{k-1} \int_{\tau_n^-}^{\tau_n^+} \Delta_1(x)\frac{f_0^{(l)}(\bar{\tau})}{l!}(x-\bar{\tau})^l\binom{-r}{j}(p_n(x)-1)^j\ \d{x}\\
	&\quad\quad +\int_{\tau_n^-}^{\tau_n^+} \Delta_1(x)\frac{f_0^{(l)}(\eta_{n,x,k})}{k!}(x-\bar{\tau})^k\binom{-r}{k}(p_n(x)-1)^k\ \d{x},\\
	S_{nk}&=\sum_{l=0}^{k-1} \int_{\tau_n^-}^{\tau_n^+} \Delta_1(x)\frac{f_0^{(l)}(\bar{\tau})}{l!}\theta_{x,n}^{-r-k}(x-\bar{\tau})^l\binom{-r}{j}(p_n(x)-1)^k\ \d{x}\\
	&\quad\quad +\int_{\tau_n^-}^{\tau_n^+} \Delta_1(x)\frac{f_0^{(l)}(\eta_{n,x,k})}{k!}\theta_{x,n}^{-r-k}(x-\bar{\tau})^k\binom{-r}{k}(p_n(x)-1)^k\ \d{x}.
	\end{split}
	\end{equation*}
	Now we see the dominating term is the first term in $S_{n1}$ since all other terms are of higher orders, and $\abs{\theta_{x,n}-1}=o_p(1)$ uniformly locally in $x$ in view of Theorem \ref{localconv}. We denote this term $Q_{n1}$. Note that $1/g_0(x_0)=1/g_0(\tau)+o_p(1)$ uniformly in $\tau$ around $x_0$, and that $\hat{g}_n$ is piecewise linear, yielding
	\begin{equation*}
	\begin{split}
	\frac{Q_{n1}}{-r f_0(\bar{\tau})}&=\int_{\tau_n^-}^{\tau_n^+}\Delta_1(x)\frac{1}{g_0(x)}\big(\hat{g}_n(x)-g_0(x)\big)\ \d{x}\\
	&=\bigg(\frac{1}{g_0(x_0)}+o_p(1)\bigg)\int_{\tau_n^-}^{\tau_n^+} \Delta_1(x)\big(\hat{g}_n(x)-g_0(x)\big)\ \d{x}\\
	&=\bigg(\frac{1}{g_0(x_0)}+o_p(1)\bigg)\bigg[\big(\hat{g}_n(\bar{\tau})-g_0(\bar{\tau})\big)\int_{\tau_n^-}^{\tau_n^+} \Delta_1(x)\ \d{x}\\
	&\qquad\qquad+\big(\hat{g}'_n(\bar{\tau})-g'_0(\bar{\tau})\big)\int_{\tau_n^-}^{\tau_n^+} \Delta_1(x)(x-\bar{\tau})\ \d{x}\\
	&\qquad\qquad\qquad -\sum_{j=2}^k \frac{g_0^{(j)}(\bar{\tau})}{j!}\int_{\tau_n^-}^{\tau_n^+} \Delta_1(x)(x-\bar{\tau})^j\ \d{x}\\
	&\qquad\qquad\qquad\qquad-\int_{\tau_n^-}^{\tau_n^+} \epsilon_n(x)\Delta_1(x)(x-\bar{\tau})^k\ \d{x}\bigg],\\
	\end{split}
	\end{equation*}
	where the first two terms in the bracket is zero by construction of $\Delta_1$. Now note that
	\[
	\int_{\tau_n^-}^{\tau_n^+} \Delta_1(x)(x-\bar{\tau})^j\ \d{x}=
	\begin{cases}
	0 & j=0, \textrm { or } j\textrm{ is odd};\\
	\frac{j}{2^{j+2}(j+1)(j+2)}\big(\tau_n^+-\tau_n^-\big)^{j+2} & j>0, \textrm{ and }j \textrm{ is even},\\
	\end{cases}
	\]
	and that $g_0^{(j)}(\bar{\tau})=\frac{1}{(k-j)!}(g_0^{(k)}(x_0)+o_p(1))\big(\bar{\tau}-x_0)^{k-j}$. This means that for $j\geq 2$ and $j$ even,
	\begin{equation*}
	\begin{split}
	\frac{g_0^{(j)}(\bar{\tau})}{j!}\int_{\tau_n^-}^{\tau_n^+} \Delta_1(x)(x-\bar{\tau})^j\ \d{x}&=\frac{j(g_0^{(k)}(x_0)+o_p(1))}{(k-j)!(j+2)!2^{j+2}}(\bar{\tau}-x_0)^{k-j}(\tau_n^+-\tau_n^-)^{j+2}\\
	&=\frac{j(g_0^{(k)}(x_0)+o_p(1))}{(k-j)!(j+2)!2^{j+2}} O_p(1)(\tau_n^+-\tau_n^-)^{k+2}.\\
	\end{split}
	\end{equation*}
	Further note that $\pnorm{\epsilon_n}{\infty}=o_p(1)$ as $\tau_n^+-\tau_n^-\to_p 0$, we get $Q_{n1}=O_p(\tau_n^+-\tau_n^-)^{k+2}.$
	This establishes the first claim. The proof for $R_{2n}$ follows the same line as in the proof of Lemma 4.4 \cite{balabdaoui2009limit} p1318-1319.
	\end{proof}

\begin{lemma}\label{derivative}
	We have the following:
	\begin{equation*}
	\begin{split}
	f_0^{(j)}(x_0)&=j!\binom{-r}{j}g_0(x_0)^{-r-j}\big(g_0'(x_0)\big)^j, 1\leq j \leq k-1;\\
	f_0^{(k)}(x_0)&=k!\binom{-r}{k}g_0(x_0)^{-r-k}\big(g_0'(x_0)\big)^k-rg_0(x_0)^{-r-1}g_0^{(k)}(x_0).
	\end{split}
	\end{equation*}
\end{lemma}
\begin{proof}
	This follows from direct calculation.
\end{proof}
\begin{lemma}\label{basiclem}
	For any $M>0$, we have
	\begin{equation*}
	\begin{split}
	\sup_{\abs{t}\leq M}\abs{\hat{g}_n'(x_0+s_n t)-\hat{g}_0'(x_0)}&=O_p(s_n^{k-1});\\
	\sup_{\abs{t}\leq M}\abs{\hat{g}_n(x_0+s_n t)-g_0(x_0)-s_n tg_0'(x_0)}&=O_p(s_n^k).
	\end{split}
	\end{equation*}
\end{lemma}
The proof is identical to Lemma 4.4 in \cite{groeneboom2001estimation} so we shall omit it.
\begin{lemma}\label{en}
Let
\[\hat{e}_n(u):=\hat{f}_n(u)-\sum_{j=0}^{k-1}\frac{f_0^{(j)}(x_0)}{j!}(u-x_0)^j-f_0(x_0)\binom{-r}{k}\bigg(\frac{g_0'(x_0)}{g_0(x_0)}\bigg)^k(u-x_0)^k.\]
Then for any $M>0$, we have	$\sup_{\abs{t}\leq M}\abs{\hat{e}_n(x_0+s_n t)}=O_p(s_n^k).$
\end{lemma}
\begin{proof}
	Note that
	\begin{equation}\label{en1}
	\begin{split}
	\hat{f}_n(u)-f_0(x_0)&=f_0(x_0)\bigg[\frac{\hat{f}_n(u)}{f_0(x_0)}-1\bigg]=f_0(x_0)\bigg[\bigg(\frac{\hat{g}_n(u)}{g_0(x_0)}\bigg)^{-r}-1\bigg]\\
	&=f_0(x_0)\bigg(\sum_{j=1}^k\binom{-r}{j}\bigg(\frac{\hat{g}_n(u)}{g_0(x_0)}-1\bigg)^j+\underbrace{\sum_{j\geq k+1}\binom{-r}{j}\bigg(\frac{\hat{g}_n(u)}{g_0(x_0)}-1\bigg)^j}_{=:\hat{\Psi}_{k,n,1}(u)}\bigg).
	\end{split}
	\end{equation}
	Define $
	\hat{\Psi}_{k,n,1}(u):=\sum_{j\geq k+1}\binom{-r}{j}\left(\frac{\hat{g}_n(u)}{g_0(x_0)}-1\right)^j=\sum_{j\geq k+1}\binom{-r}{j}\frac{1}{g_0(x_0)^j}\big(\hat{g}_n(u)-g_0(x_0)\big)^j. $
	Note that
	\begin{equation*}
	\begin{split}
	\big(\hat{g}_n(u)-g_0(x_0)\big)^j&=\big(\hat{g}_n(u)-g_0(x_0)-(u-x_0)g_0'(x_0)+(u-x_0)g_0'(x_0)\big)^j\\
	&=\sum_{l=1}^j \binom{j}{l}\big[\hat{g}_n(u)-g_0(x_0)-(u-x_0)g_0'(x_0)\big]^l(u-x_0)^{j-l}g_0'(x_0)^{j-l}\\
	&\quad\quad+(u-x_0)^j g_0'(x_0)^j \\
	&=O_p(s_n^{kl}\cdot s_n^{j-l})+O_p(s_n^j)\\
	&\qquad\qquad \textrm{ uniformly on }\{u:\abs{u-x_0}\leq M n^{-1/(2k+1)}\}\\
	&=O_p(n^{-\frac{j}{2k+1}}),
	\end{split}
	\end{equation*}
	if $j\geq k+1$. Here the third line follows from Lemma \ref{basiclem}. This implies
	$\hat{\Psi}_{k,n,1}(u)=o_p(n^{-\frac{k}{2k+1}}),$
	uniformly on $\{u:\abs{u-x_0}\leq M n^{-1/(2k+1)}\}$. Using the same expansion in the first term on the right hand side of (\ref{en1}), we arrive at
	\begin{equation*}
	\begin{split}
	&\underbrace{\hat{f}_n(u)-f_0(x_0)}_{(1)}\\
	=&\underbrace{f_0(x_0)\sum_{j=1}^k\binom{-r}{j}\frac{1}{[g_0(x_0)]^j}\sum_{r=1}^j \binom{j}{r}\big[\hat{g}_n(u)-g_0(x_0)-(u-x_0)g_0'(x_0)\big]^r(u-x_0)^{j-r}g_0(x_0)^{j-r}}_{(2)}\\
	&\quad\quad +\underbrace{f_0(x_0)\sum_{j=1}^k\binom{-r}{j}\bigg(\frac{g_0'(x_0)}{g_0(x_0)}\bigg)^j(u-x_0)^j}_{(3)}+\underbrace{f_0(x_0)\hat{\Psi}_{k,n,1}(u)}_{(4)}.\\
	\end{split}
	\end{equation*}
	By Lemma \ref{derivative}, we see that $\hat{e}_n(u)=(1)-(3)=(2)+(4)=O_p(s_n^k)$ 
	uniformly on $\{u:\abs{u-x_0}\leq M n^{-1/(2k+1)}\}$. This yields the desired result.
\end{proof}
We are now ready for the proof of Theorem \ref{uniquetight}.
\begin{proof}[Proof of Theorem \ref{uniquetight}]
	For the first assertion, note that
	\begin{equation*}
	\begin{split}
	&[f_0(x_0)]^{-1}\bigg(\hat{f}_n(u)-\sum_{j=0}^{k-1}\frac{f_0^{(j)}(x_0)}{j!}(u-x_0)^j\bigg)\\
	=&[f_0(x_0)]^{-1}\bigg(\hat{f}_n(u)-f_0(x_0)-\sum_{j=1}^{k-1}\frac{f_0^{(j)}(x_0)}{j!}(u-x_0)^j\bigg)\\
	=&[f_0(x_0)]^{-1}\bigg(f_0(x_0)\bigg(\sum_{j=1}^k\binom{-r}{j}\bigg(\frac{\hat{g}_n(u)}{g_0(x_0)}-1\bigg)^j+\hat{\Psi}_{k,n,1}(u)\bigg)-\sum_{j=1}^{k-1}\frac{f_0^{(j)}(x_0)}{j!}(u-x_0)^j\bigg)\\
	&\qquad\qquad\textrm{ by } (\ref{en1})\\
	=&\hat{\Psi}_{k,n,1}(u)+\sum_{j=1}^k\binom{-r}{j}\bigg(\frac{\hat{g}_n(u)}{g_0(x_0)}-1\bigg)^j-[f_0(x_0)]^{-1}\sum_{j=1}^{k-1}\frac{f_0^{(j)}(x_0)}{j!}(u-x_0)^j\\
	=&\hat{\Psi}_{k,n,1}(u)+\binom{-r}{1}\bigg(\frac{\hat{g}_n(u)}{g_0(x_0)}-1\bigg)-\frac{1}{f_0(x_0)}f_0'(x_0)(u-x_0)\\
	&\quad\quad +\sum_{j=2}^k\binom{-r}{j}\bigg(\frac{\hat{g}_n(u)}{g_0(x_0)}-1\bigg)^j-[f_0(x_0)]^{-1}\sum_{j=2}^{k-1}\frac{f_0^{(j)}(x_0)}{j!}(u-x_0)^j\\
	=& \hat{\Psi}_{k,n,1}(u)-\frac{r}{g_0(x_0)}\bigg(\hat{g}_n(u)-g_0(x_0)-g_0'(x_0)(u-x_0)\bigg)+\sum_{j=2}^k\binom{-r}{j}\bigg(\frac{\hat{g}_n(u)}{g_0(x_0)}-1\bigg)^j\\
	&\quad\quad-[f_0(x_0)]^{-1}\sum_{j=2}^{k-1}\frac{f_0^{(j)}(x_0)}{j!}(u-x_0)^j\\
	=& -\frac{r}{g_0(x_0)}\bigg(\hat{g}_n(u)-g_0(x_0)-g_0'(x_0)(u-x_0)\bigg)+\hat{\Psi}_{k,n,2}(u),
	\end{split}
	\end{equation*}
	where
	\[\hat{\Psi}_{k,n,2}(u):=\hat{\Psi}_{k,n,1}(u)+\sum_{j=2}^k\binom{-r}{j}\bigg(\frac{\hat{g}_n(u)}{g_0(x_0)}-1\bigg)^j-[f_0(x_0)]^{-1}\sum_{j=2}^{k-1}\frac{f_0^{(j)}(x_0)}{j!}(u-x_0)^j.\]
	Now we calculate
	\begin{equation*}
	\begin{split}
	&\int_{\bm{l}_{n,x_0}}\int_{x_0}^v \hat{\Psi}_{k,n,2}(u)\d{u}\d{v}\\
	=&{1 \over 2}t^2n^{-{2 \over {2k+1}}}\sup_{u\in \bm{l}_{n,x_0}}\abs{\hat{\Psi}_{k,n,1}(u)}+\sum_{j=2}^k\binom{-r}{j}\int_{\bm{l}_{n,x_0}}\int_{x_0}^v \bigg(\frac{\hat{g}_n(u)}{g_0(x_0)}-1\bigg)^j\ \d{u}\d{v}\\
	&\quad\quad-[f_0(x_0)]^{-1}\sum_{j=2}^{k-1}\frac{f_0^{(j)}(x_0)}{j!}\int_{\bm{l}_{n,x_0}}\int_{x_0}^v (u-x_0)^j\ \d{u}\d{v}\\
	=& o_p(r_n^{-1})+\sum_{j=2}^k\binom{-r}{j}\bigg(\frac{g_0'(x_0)}{g_0(x_0)}\bigg)^j \int_{\bm{l}_{n,x_0}}\int_{x_0}^v (u-x_0)^j\ \d{u}\d{v}\\
	&\qquad-\sum_{j=2}^{k-1}\binom{-r}{j}\bigg(\frac{g_0'(x_0)}{g_0(x_0)}\bigg)^j \int_{\bm{l}_{n,x_0}}\int_{x_0}^v (u-x_0)^j\ \d{u}\d{v}\\
	&\qquad+\bigg(\sum_{j=2}^k\binom{-r}{j}\frac{1}{[g_0(x_0)]^j}\\
	&\qquad\qquad\times\int_{\bm{l}_{n,x_0}}\int_{x_0}^v\sum_{l=1}^j \binom{j}{l}\big(\hat{g}_n(u)-g_0(x_0)-g_0'(x_0)(u-x_0)\big)^l(u-x_0)^{j-l}[g_0'(x_0)]^{j-l}\ \d{u}\d{v}\bigg)\\
	=& o_p(r_n^{-1})+\binom{-r}{k}\bigg(\frac{g_0'(x_0)}{g_0(x_0)}\bigg)^k \int_{\bm{l}_{n,x_0}}\int_{x_0}^v (u-x_0)^k\ \d{u}\d{v}\\
	& +\bigg(\sum_{j=2}^k\binom{-r}{j}\frac{1}{[g_0(x_0)]^j}\\
	&\qquad\times\int_{\bm{l}_{n,x_0}}\int_{x_0}^v\sum_{l=1}^j \binom{j}{l}\big(\hat{g}_n(u)-g_0(x_0)-g_0'(x_0)(u-x_0)\big)^l(u-x_0)^{j-l}[g_0'(x_0)]^{j-l}\ \d{u}\d{v}\bigg)\\
	:=& o_p(r_n^{-1})+(2)+(1).
	\end{split}
	\end{equation*}
	Consider (1): for each $(j,l)$ satisfying $1\leq l\leq j\leq k$ and $j\geq 2$, we have
	\begin{equation*}
	\begin{split}
	(1)&: r_n \int_{\bm{l}_{n,x_0}}\int_{x_0}^v \big(\hat{g}_n(u)-g_0(x_0)-g_0'(x_0)(u-x_0)\big)^l(u-x_0)^{j-l}[g_0'(x_0)]^{j-l}\ \d{u}\d{v}\\
	&=n^{\frac{k+2}{2k+1}}\cdot O(n^{-\frac{2}{2k+1}})\cdot O_p(n^{-\frac{kl}{2k+1}})\cdot O_p(n^{-\frac{j-l}{2k+1}})= O_p(n^{-\frac{k(l-1)+(j-l)}{2k+1}})=o_p(1).
	\end{split}
	\end{equation*}
	Consider (2) as follows:
	\begin{equation*}
	\begin{split}
	(2)&= \binom{-r}{k}\bigg(\frac{g_0'(x_0)}{g_0(x_0)}\bigg)^k \int_{\bm{l}_{n,x_0}}\int_{x_0}^v (u-x_0)^k\ \d{u}\d{v}\\
	&={1 \over {(k+1)(k+2)}}\binom{-r}{k}\bigg(\frac{g_0'(x_0)}{g_0(x_0)}\bigg)^k t^{k+2}r_n^{-1}.
	\end{split}
	\end{equation*}
	Hence we have
	\[
	r_n \int_{\bm{l}_{n,x_0}}\int_{x_0}^v \hat{\Psi}_{k,n,2}(u)\d{u}\d{v}={1 \over {(k+1)(k+2)}}\binom{-r}{k}\bigg(\frac{g_0'(x_0)}{g_0(x_0)}\bigg)^k t^{k+2}+o_p(1).
	\]
	Note by definition we have
	\begin{equation}\label{ymod}
	\mathbb{Y}^{\mathrm{locmod}}_n(t)=\frac{\mathbb{Y}^{\mathrm{loc}}_n(t)}{f_0(x_0)}-r_n \int_{\bm{l}_{n,x_0}}\int_{x_0}^v \hat{\Psi}_{k,n,2}(u)\d{u}\d{v}.
	\end{equation}
	Let $n \to \infty$, by the same calculation in the proof of Theorem 6.2 \cite{groeneboom2001estimation}, we have
	\begin{equation*}
	\begin{split}
	\mathbb{Y}^{\mathrm{locmod}}_n(t)& \to_d \frac{1}{\sqrt{f_0(x_0)}}\int_0^t W(s)\ \d{s}\\
	&\quad\quad+\bigg[\frac{f_0^{(k)}(x_0)}{(k+2)!f_0(x_0)}-{1 \over {(k+1)(k+2)}}\binom{-r}{k}\bigg(\frac{g_0'(x_0)}{g_0(x_0)}\bigg)^k\bigg] t^{k+2}\\
	&=\frac{1}{\sqrt{f_0(x_0)}}\int_0^t W(s)\ \d{s}-\frac{rg_0^{(k)}(x_0)}{g_0(x_0)(k+2)!}t^{k+2},
	\end{split}
	\end{equation*}
	where the last line follows from Lemma \ref{derivative}. Now we turn to the second assertion. It is easy to check by the definition of $\hat{\Psi}_{k,n,2}(\cdot)$ that
	\begin{equation}\label{hmod}
	\mathbb{H}^{\mathrm{locmod}}_n(t)=\frac{\mathbb{H}^{\mathrm{loc}}_n(t)}{f_0(x_0)}-r_n \int_{\bm{l}_{n,x_0}}\int_{x_0}^v \hat{\Psi}_{k,n,2}(u)\d{u}\d{v}.
	\end{equation}
	On the other hand, simple calculation yields that $\mathbb{Y}^{\mathrm{loc}}_n(t)-\mathbb{H}^{\mathrm{loc}}_n(t)=r_n\big(\mathbb{H}_n(x_0+s_nt)-\hat{H}_n(x_0+s_nt)\big)\geq 0$
	where the inequality follows from Theorem \ref{secondintegralchar}. Combined with (\ref{ymod}) and (\ref{hmod}) we have shown the second assertion. Finally we show tightness of $\{\hat{A}_n\}$ and $\{\hat{B}_n\}$. By Theorem \ref{gapthm}, we can find $M>0$ and $\tau \in\mathcal{S}(\hat{g}_n)$ such that $0\leq \tau-x_0\leq M n^{-1/(2k+1)}$ with large probability. Now note
	\begin{equation*}
	\begin{split}
	\abs{\hat{A}_n}&\leq r_ns_n\abs{\big(\hat{F}_n(x_0)-\hat{F}_n(\tau)\big)-\big(\mathbb{F}_n(x_0)-\mathbb{F}_n(\tau)\big)}+\frac{r_ns_n}{n}\\
	&\leq r_ns_n\abs{\int^\tau_{x_0} \bigg(\hat{f}_n(u)-\sum_{j=0}^{k-1}\frac{f_0^{(j)}(x_0)}{j!}(u-x_0)^j\bigg)\ \d{u}}\\
	&\quad\quad+r_ns_n\abs{\int^\tau_{x_0} \bigg(\sum_{j=0}^{k-1}\frac{f_0^{(j)}(x_0)}{j!}(u-x_0)^j-f_0(u)\bigg)\ \d{u}}\\
	&\quad\quad\quad +r_ns_n\abs{\int^\tau_{x_0}\d{(\mathbb{F}_n-F_0)}}+n^{-k/(2k+1)}\\
	&=:\hat{A}_{n1}+\hat{A}_{n2}+\hat{A}_{n3}+n^{-k/(2k+1)}.
	\end{split}
	\end{equation*}
	We calculate three terms respectively.
	\begin{equation*}
	\begin{split}
	\hat{A}_{n1}&\leq r_ns_n\abs{\int^\tau_{x_0}\hat{e}_n(u)\ \d{u}}+r_ns_n\abs{\int^\tau_{x_0}f_0(x_0)\binom{-r}{k}\bigg(\frac{g_0'(x_0)}{g_0(x_0)}\bigg)^{k}(u-x_0)^k\ \d{u}}\\
	&=O_p(r_ns_n\cdot s_n^{k+1})+o_p(r_ns_n\cdot s_n^{k+1})=O_p(1),\quad \textrm{ by Lemma }\ref{en}\\
	\hat{A}_{n2}&\leq r_ns_n\abs{\int^\tau_{x_0} \frac{f_0^{(k)}(x_0)}{k!}(u-x_0)^k\ \d{u}}+r_ns_n\abs{\int^\tau_{x_0}(u-x_0)^k \epsilon_n(u)\ \d{u}}\\
	&=O_p(1),\quad\textrm{ since }\pnorm{\epsilon_n}{\infty}\to_p 0\textrm{ as }x_0-\tau \to_p 0.
	\end{split}
	\end{equation*}
	For $\hat{A}_{n3}$, we follow the lines of Lemma 4.1 \cite{balabdaoui2009limit} again to conclude. Fix $R>0$, and consider the function class
	$\mathcal{F}_{x_0,R}:=\{{\bf 1}_{[x_0,y]}:x_0\leq y\leq x_0+R\}.$
	Then $F_{x_0,R}(z):={\bf 1}_{[x_0,x_0+R]}(z)$ is an envelop function for $\mathcal{F}_{x_0,R}$, and
	$\mathbb{E}F_{x_0,R}^2=\int_{x_0}^{x_0+R}\ \d{z}= R.$
	Now let $s=k,d=1$ in Lemma 4.1 \cite{balabdaoui2009limit}, we have
	\[\hat{A}_{n3}=\abs{\int_{x_0}^\tau \d{(\mathbb{F}_n-F_0)}(z)}\leq \abs{\tau-x_0}^{k+1}+O_p(1)n^{-\frac{k+1}{2k+1}}=O_p(1).\]
	This completes the proof for tightness for $\{A_n\}$. $\{B_n\}$ follows from similar argument so we omit the details.
\end{proof}

\subsection{Auxiliary convex analysis}
\begin{lemma}[Lemma 4.3, \cite{dumbgen2011approximation}]\label{Lip}
	For any $\varphi(\cdot)\in \mathcal{G}$ with non-empty domain, and $\epsilon>0$, define
	\[\varphi^{(\epsilon)}(x):=\sup_{(v,c)}(v^T x+c)\]
	where the supremum is taken over all pairs of $(v,c)\in \R^d\times \R$ such that
	\begin{enumerate}
		\item $\pnorm{v}{}\leq {1 \over \epsilon}$;
		\item $\varphi(y)\geq v^T y +c$ holds for all $ y \in \R^d$.
	\end{enumerate}
	Then $\varphi^{(\epsilon)} \in \mathcal{G}$ with Lipschitz constant ${1 \over \epsilon}$. Furthermore,
	$$\varphi^{(\epsilon)}\nearrow \varphi,\textrm{ as }\epsilon \searrow 0,$$
	where the convergence is pointwise for all $x \in \R^d$.
\end{lemma}
\begin{lemma}[Lemma 2.13, \cite{dumbgen2011approximation}]\label{intpoint}
	Given $Q \in \mathcal{Q}_0$, a point $x \in \R^d$ is an interior point of $\textrm{csupp}(Q)$ if and only if
	\[h(Q,x)\equiv \sup \{Q(C):C\subset \R^d \textrm{ closed and convex}, x\notin \mathrm{int}(C)\}<1.\]
	Moreover, if $\{Q_n\}\subset \mathcal{Q}$ converges weakly to $Q$, then
	\[\limsup_{n \to \infty}h(Q_n,x)\leq h(Q,x)\]
	holds for all $ x \in \R^d$.
\end{lemma}

\begin{lemma}\label{convlb}
	If $g \in \mathcal{G}$, then there exists $a,b>0$ such that for all $x \in \R^d$, 
	$ g(x)\geq a\pnorm{x}{}-b.$
\end{lemma}
\begin{proof}
	The proof is essentially the same as for Lemma 1, \cite{cule2010theoretical}, so we shall omit it.
\end{proof}

Consider the class of functions
\[\mathcal{G}_M:=\left\{g \in \mathcal{G}:\int g^{\beta}\ \d{x}\leq M\right\}.\]
\begin{lemma}\label{controllevelset}
	For a given $g \in \mathcal{G}_M$, denote $D_r:=D(g,r):=\{g\leq r\}$ to be the level set of $g(\cdot)$ at level $r$, and $\epsilon:=\inf g$. Then for $r> \epsilon$, we have
	\[
	\lambda(D_r)\leq\frac{M(-s)(r-\epsilon)^d}{(s+1)\int_0^{r-\epsilon} v^d (v+\epsilon)^{1/s}\ \d{v}},
	\]
	where $\beta=1+1/s$, and $-1<s<0$.
\end{lemma}
\begin{proof}
	For  $u \in [\epsilon, r]$, by convexity of $g(\cdot)$, we have
	\[\lambda(D_u)\geq \left(\frac{u-\epsilon}{r-\epsilon}\right)^d\lambda(D_r).\]
	This can be seen as follows: Consider the epigraph $\Gamma_g$ of $g(\cdot)$, where
	$\Gamma_g=\{(t,x)\in \R^d\times\R:x\geq g(t)\}.$
	Let $x_0 \in \R^d$ be a minimizer of $g$. Consider the convex set $C_r=\mathrm{conv}\big(\Gamma_g\cap\{g=r\},(x_0,\epsilon)\big)\subset \Gamma_g \cap \{g\leq r\}.$
	where the inclusion follows from the convexity of $\Gamma_g$ as a subset of $\R^{d+1}$. The claimed inequality follows from
	\[\lambda_d(D_u)=\lambda_d\big(\pi_d(\Gamma_g\cap\{g=u\})\big)\geq \lambda_d\big(\pi_d(C_r\cap\{g=u\})\big)=\left(\frac{u-\epsilon}{r-\epsilon}\right)^d\lambda_d(D_r),\]
	where $\pi_d:\R^d\times \R\to \R^d$ is the natural projection onto the first component. Now we do the calculation as follows:
	\begin{equation*}
	\begin{split}
	M&\geq \int_{D_r} \big(g(x)^{1/s+1}-r^{1/s+1}\big)\ \d{x}\\
	&=-\left({1 \over s}+1\right)\int_{D_r}\bigg(\int_{\epsilon}^r {\bf 1}(u\geq g(x))u^{1/s}\ \d{u}\bigg)\ \d{x}\\
	&=-\left({1 \over s}+1\right)\int_{\epsilon}^r u^{1/s}\ \d{u}\int_{D_r}{\bf 1}(u\geq g(x))\ \d{x}\\
	&=-\left({1 \over s}+1\right)\int_\epsilon^r \lambda(D_u)u^{1/s}\ \d{u}\\
	&\geq -\left({1 \over s}+1\right)\int_\epsilon^r \left(\frac{u-\epsilon}{r-\epsilon}\right)^d\lambda(D_r) u^{1/s}\ \d{u}\\
	&=\lambda(D_r)\cdot \frac{(s+1)\int_\epsilon^r(u-\epsilon)^d u^{1/s}\ \d{u}}{(-s)(r-\epsilon)^d}.\\
	\end{split}
	\end{equation*}
	By a change of variable in the integral we get the desired inequality.
\end{proof}
\begin{lemma}\label{inftymeasure}
	Let $G$ be a convex set in $\R^d$ with non-empty interior, and a sequence $\{y_n\}_{n \in \N}$ with $\pnorm{y_n}{} \to \infty$ as $n \to \infty$. Then there exists $\{x_1,\ldots,x_d\}\subset G$ such that
	\[\lambda_d\big(\conv{x_1,\ldots,x_d,y_{n(k)}}\big)\to \infty,\]
	as $k \to \infty$ where $\{y_{n(k)}\}_{k \in \N}$ is a suitable subsequence of $\{y_n\}_{n \in \N}$.
\end{lemma}
\begin{proof}
	Without loss of generality we assume $0 \in \intdom{G}$, and we first choose a convergence subsequence $\{y_{n(k)}\}_{k \in \N}$ from $\{y_n/\pnorm{y_n}{}\}_{n \in \N}$. Now if we let $a:=\lim_{k \to \infty} y_{n(k)}/\pnorm{y_{n(k)}}{}$, then $\pnorm{a}{}=1$. Since $G$ has non-empty interior, $\{a^T x=0\}\cap G$ has non-empty relative interior. Thus we can choose $x_1,\ldots,x_d \subset \{a^T x=0\}\cap G$ such that $\lambda_{d-1}(K)\equiv\lambda_{d-1}\big(\conv{x_1,\ldots,x_d}\big)>0$. Note that
	\[\dist{y_{n(k)}}{\mathrm{aff}(K)}=\dist{y_{n(k)}}{\{a^T x=0\}}=\langle y_{n(k)},a\rangle=\pnorm{y_{n(k)}}{}\langle y_{n(k)}/\pnorm{y_{n(k)}}{},a\rangle\to \infty,\]
	as $k \to \infty$. Since
	\[\lambda_d\big(\conv{x_1,\ldots,x_d,y_{n(k)}}\big)=\lambda_d\big(\conv{K,y_{n(k)}}\big)= c \lambda_{d-1}(K)\cdot\dist{y_{n(k)}}{\mathrm{aff}(K)},\]
	for some constant $c=c(d)>0$, the proof is complete as we let $k \to \infty$.
\end{proof}
\begin{lemma}[Lemma 4.2, \cite{dumbgen2011approximation}]\label{convsubsequence}
	Let $\bar{g}$ and $\{g_n\}_{n \in \N}$ be functions in $\mathcal{G}$ such that $g_n\geq \bar{g}$, for all $ n \in \N$. Suppose the set
	$C:=\{x \in \R^d:\limsup_{n \to \infty} g_n(x)<\infty\}$
	is non-empty. Then there exist a subsequence $\{g_{n(k)}\}_{k \in \N}$ of $\{g_n\}_{n \in \N}$, and a function $g\in \mathcal{G}$ such that $C \subset \mathrm{dom}(g)$ and
	\begin{equation}
	\begin{split}
	&\lim_{k \to \infty,x \to y}g_{n(k)}(x)=g(y),\quad \textrm{for all } y \in \mathrm{int(dom}(g)),\\
	&\liminf_{k \to \infty, x \to y}g_{n(k)}(x)\geq g(y),\quad \textrm{for all } y \in \R^d.\\
	\end{split}
	\end{equation}
\end{lemma}
\begin{lemma}\label{globallowerbound}
	Let $\{g_n\}$ be a sequence of non-negative convex functions satisfying the following conditions:
	\begin{enumerate}
		\item[(A1).] There exists a convex set $G$ with non-empty interior such that for all $x_0 \in \mathrm{int}(G)$, we have
		$\sup_{n \in \N} g_n(x_0)<\infty.$
		\item[(A2).] There exists some $M>0$ such that
		$\sup_{n \in \N} \int \big(g_n(x)\big)^\beta\ \d{x}\leq M<\infty.$
	\end{enumerate}
	Then there exists $a,b>0$ such that for all $x \in \R^d$ and $k \in \N$
	\[g_{n(k)}(x)\geq a\pnorm{x}{}-b,\]
	where $\{g_{n(k)}\}_{k \in \N}$ is a suitable subsequence of $\{g_n\}_{n \in \N}$.
\end{lemma}
\begin{proof}
	Without loss of generality we may assume $G$ is contained in all $\intdom{g_n}$. We first note (A1)-(A2) implies that $\{\widehat{x}_n \in \textrm{Arg}\min_{x \in \R^d}g_n(x)\}_{n=1}^\infty$ is a bounded sequence, i.e.
	\begin{equation}\label{boundedmin}
	\sup_{n \in \N}\pnorm{\widehat{x}_n}{}<\infty,
	\end{equation}
	Suppose not, then without loss of generality we may assume $\pnorm{\widehat{x}_n}{}\to \infty$ as $n \to \infty$. By Lemma \ref{inftymeasure}, we can choose $\{x_1,\ldots,x_d\}\subset G$ such that
	$\lambda_d\big(\conv{x_1,\ldots,x_d,\widehat{x}_{n(k)}}\big)\to\infty,$
	as $k \to \infty$ for some subsequence $\{\widehat{x}_{n(k)}\}\subset\{\widehat{x}_n\}$. For simplicity of notation we think of $\{\widehat{x}_n\}$ as such an appropriate subsequence. Denote $\epsilon_n:=\inf_{x \in \R^d} g_n(x)$, and $M_2:=\sup_{n \in \N}\epsilon_n\leq \sup_{n \in \N} g_n(x_0)<\infty$ by (A1). Again by (A1) and convexity we may assume that
	\[\sup_{x\in\conv{x_1,\ldots,x_d,\widehat{x}_n}} g_{n}(x)\leq M_1,\]
	holds for some $M_1>0$ and all $n \in \N$. This implies that
	\[\int g_{n}^\beta (x)\ \d{x}\geq M_1^\beta \lambda_d\big(\conv{x_1,\ldots,x_d,\hat{x}_{n}}\big) \to \infty,\]
	as $n \to \infty$, which gives a contradiction to (A2). This shows (\ref{boundedmin}).
	
	Now we define $\underline{g}(\cdot)$ be the convex hull of $\tilde{g}(x):=\inf_{n \in \N}g_n(x)$, then $\underline{g}\leq g_n$ holds for all $n \in \N$. We claim that $\underline{g}(x) \to \infty$ as $\pnorm{x}{} \to \infty$. By Lemma \ref{controllevelset}, for fixed $\eta>1$, we have
	\begin{equation*}
	\begin{split}
	\lambda_d\big(D(g_n,\eta M_2)\big)&\leq \frac{M(-s)(\eta M_2-\epsilon_n)^d}{(s+1)\int_0^{\eta M_2-\epsilon_n}v^d(v+\epsilon_n)^{1/s}\ \d{v}}\\
	&\leq \frac{M(-s)(\eta M_2)^d}{(s+1)\int_0^{(\eta-1)M_2} v^d(v+M_2)^{1/s}\ \d{v}}<\infty,
	\end{split}
	\end{equation*}
	where $D(g_n,\eta M_2):=\{g_n\leq \eta M_2\}$. Hence
	\begin{equation}\label{finitemeasure}
	\sup_{n \in \N} \lambda_d \big(D(g_n,\eta M_2)\big)<\infty.
	\end{equation}
	holds for every $\eta>1$. Now combining (\ref{boundedmin}) and (\ref{finitemeasure}), we claim that, for fixed $\eta$ large enough, it is possible to find $R=R(\eta)>0$ such that
	\begin{equation}\label{globalblowup}
	g_n(x)\geq \eta M_2,
	\end{equation}
	holds for all $x \geq R(\eta)$ and $ n \in \N$. If this is not true, then for all $ k \in \N$, we can find $n(k) \in \N$ and $\bar{x}_k \in \R^d$ with $\pnorm{\bar{x}_k}{}\geq k$ such that $g_{n(k)}(\bar{x}_k)\leq \eta M_2$. We consider two cases to derive a contradiction. 
	
	\noindent \textbf{[Case 1.]} If for some $ n_0 \in \N$ there exists infinitely many $k \in \N$ with $n(k)=n_0$, then we may assume without loss of generality that we can find some a sequence $\{\bar{x}_k\}_{k \in \N}$ with $\pnorm{\bar{x}_k}{} \to \infty$ as $k \to \infty$, and $g_{n_0}(\bar{x}_k)\leq \eta M_2$. Since the support $g_{n_0}$ has non-empty interior, by Lemma \ref{inftymeasure}, we can find $x_1,\ldots,x_d \in \mathrm{supp}(g_{n_0})$ such that $\lambda_d\big(\mathrm{conv}(x_1,\ldots,x_d,\bar{x}_{k(j)})\big)\to \infty$ as $j \to \infty$ holds for some subsequence $\{\bar{x}_{k(j)}\}_{j \in \N}$ of $\{\bar{x}_k\}_{k \in \N}$. Let $\bar{M}:=\max_{1\leq i \leq d} g_{n_0}(x_i)$, then we find
	$\lambda_d\big(D(g_{n_0},\bar{M}\vee \eta M_2)\big)=\infty.$ This contradicts with (\ref{finitemeasure}).
	
	\noindent \textbf{[Case 2.]} If $\#\{k \in \N: n=n(k)\}<\infty$ for all $ n \in \N$, then without loss of generality we may assume that for all $ k \in \N$, we can find $\bar{x}_k \in \R^d$ with $\pnorm{\bar{x}_k}{}\geq k$ such that
	$g_k(x_k)\leq \eta M_2.$
	Recall by assumption (A1) convex set $G$ has non-empty interior, and is contained in the support of $g_n$ for all $n \in \N$. Again by Lemma \ref{inftymeasure}, we may take $x_1,\ldots,x_d \in C$ such that
	$\lambda_d\big(\mathrm{conv}(x_1,\ldots,x_d,\bar{x}_{k(j)})\big)\to \infty$ as $j \to \infty$
	holds for some subsequence $\{\bar{x}_{k(j)}\}_{j \in \N}$ of $\{\bar{x}_k\}_{k \in \N}$. In view of (A1), we conclude by convexity that
	$\bar{M}:=\max_{1\leq i \leq d} \sup_{j \in \N}g_{{k(j)}}(x_i)<\infty.$
	This implies
	\[\lambda_d\big(D(g_{n_{k(j)}},\bar{M}\vee \eta M_2)\big)\geq \lambda_d\big(\mathrm{conv}(x_1,\ldots,x_d,\bar{x}_{k(j)})\big)\to \infty,\quad j \to \infty,\]
	which gives a contradiction.
	
	Combining these two cases we have proved (\ref{globalblowup}). This implies that $\tilde{g}(x)\to \infty$ as $\pnorm{x}{} \to \infty$, whence verifying the claim that $\underline{g}(x) \to \infty$ as $\pnorm{x}{} \to \infty$. Hence in view of Lemma \ref{convlb}, we find that there exists $a,b>0$ such that $g_n(x)\geq a\pnorm{x}{}-b $
	holds for all $x \in \R^d$ and $n \in \N$.
\end{proof}
\begin{lemma}\label{generalposition}
	Assume $x_0,\ldots,x_d\in \R^d$ are in general position. If $g(\cdot)$ is a non-negative function with $\Delta\equiv\mathrm{conv}(x_0,\ldots,x_d)\subset\mathrm{dom}(g)$, and $g(x_0)=0$. Then for $r\geq d$, we have
	$\int_\Delta \big(g(x)\big)^{-r}\ \d{x}=\infty.$
\end{lemma}
\begin{proof}
	We may assume without loss of generality that $x_0=0,x_i={\bf e}_i \in \R^d$, where ${\bf e}_i$ is the unit directional vector with $1$ in its $i$-th coordinate and $0$ otherwise. Then $\Delta=\Delta_0:=\{x \in \R^d:\sum_{i=1}^d x_i\leq 1,x_i\geq 0,\forall i=1,\ldots,d\}$. Denote $a_i=g(x_i)\geq 0$. We may assume there is at least one $a_i\neq 0$. Then by convexity of $g$ we find
	$g(x)\leq \sum_{i=1}^d a_i x_i$
	for all $x \in \Delta_0$. This gives
	\begin{equation*}
	\begin{split}
	\int_{\Delta_0}\big(g(x)\big)^{-r}\ \d{x}&\geq \int_{\Delta_0}\big(\sum_{i=1}^d a_ix_i\big)^{-r}\ \d{x}\geq \int_{\Delta_0}\frac{1}{(\max_{i=1,\ldots,d} a_i)^{r}\pnorm{x}{1}^r}\ \d{x}\\
	&\geq \frac{1}{(\max_{i=1,\ldots,d} a_i)^{r}d^{r/2}}\int_{C_0}{1 \over \pnorm{x}{2}^r}\ \d{x}=\infty,\\
	\end{split}
	\end{equation*}
	where $C_0:={\{\pnorm{x}{2}\leq{1 \over \sqrt{d}}\}}\cap\{x_i\geq 0,i=1,\ldots,d\}$. Note we used the fact that $\pnorm{x}{1}\leq \sqrt{d}\pnorm{x}{2}$.
\end{proof}
\begin{lemma}[Theorem 1.11, \cite{bhattacharya2010normal}]\label{unifconv}
	Let $f_n \to_d f$, and $\mathcal{D}$ be the class of all Borel measurable, convex subsets in $\R^d$. Then
	$\lim_{n \to \infty}\sup_{D\in\mathcal{D}}\abs{\int_D (f_n-f)}=0.$
\end{lemma}

\section*{Acknowledgements}
The authors owe thanks to Charles Doss, Roger Koenker and 
Richard Samworth, as well as two referees and an 
Associate Editor for helpful comments, suggestions and minor corrections.




\end{document}